\documentclass[11pt]{article}
\usepackage[a4paper,margin=26mm]{geometry}
\usepackage[T1]{fontenc}
\usepackage{lmodern}
\usepackage{amsmath,amssymb,amsthm,mathtools}
\usepackage{microtype}
\usepackage{enumitem}
\usepackage[hidelinks]{hyperref}
\hypersetup{pdftitle={Bramson correction and critical-wave convergence for lattice Fisher--KPP equations},
  pdfauthor={Yuanyang Hu},
  pdfsubject={Logarithmic localization along coordinate axes and critical-wave convergence for lattice Fisher--KPP solutions},
  pdfkeywords={Fisher--KPP, lattice diffusion, Bramson correction, weighted mass}}
\numberwithin{equation}{section}
\newcommand{\Z}{\mathbb Z}
\newcommand{\R}{\mathbb R}
\newcommand{\N}{\mathbb N}
\newcommand{\ee}{\mathrm e}
\newcommand{\dd}{\mathrm d}
\newcommand{\ii}{\mathrm i}
\newcommand{\1}{\mathbf 1}
\newcommand{\norm}[1]{\lVert #1\rVert}
\newcommand{\pheat}{\mathsf p}
\DeclareMathOperator{\supp}{supp}
\DeclareMathOperator{\diag}{diag}
\DeclareMathOperator{\arsinh}{arsinh}
\newtheorem{theorem}{Theorem}[section]
\newtheorem{lemma}[theorem]{Lemma}
\newtheorem{proposition}[theorem]{Proposition}
\newtheorem{corollary}[theorem]{Corollary}
\theoremstyle{remark}
\newtheorem{remark}[theorem]{Remark}

\title{Bramson correction and convergence to the critical wave\\
for Fisher--KPP equations on $\mathbb Z^d$}
\author{Yuanyang Hu\\[2pt]
\small School of Mathematics and Statistics, Henan University\\
\small Kaifeng 475004, China\\
\small \texttt{yuanyhu@henu.edu.cn}}
\date{}

\begin{document}
\maketitle

\begin{abstract}
We consider Fisher--KPP equations with nearest-neighbor diffusion on
$\mathbb Z^d$, $d\geq2$, with nonzero finitely supported initial data.
We prove a logarithmic delay of the front along each signed coordinate
axis and show that the transition region has uniformly bounded width.
On every fixed-width half-tube around an axis, the solution converges to
translates of the minimal-speed lattice traveling wave, with a bounded
phase. We also obtain an upper bound with a logarithmic correction in
every direction. The proof uses weighted estimates, bounds for tilted
random walks, and a product lower solution. Concavity of the reaction is
not assumed.
\end{abstract}



\section{Introduction}\label{sec:introduction}

We consider the Cauchy problem
\begin{equation}\label{eq:main}
\left\{
\begin{aligned}
    u_t(t,x)&=\Delta_{\Z^d}u(t,x)+f(u(t,x)),
        && t>0,\quad x\in\Z^d,\\
    u(0,x)&=u_0(x),
        && x\in\Z^d.
\end{aligned}
\right.
\end{equation}
Here $d\geq2$,
$e_1,\ldots,e_d$ are the coordinate vectors, and
\begin{equation}\label{eq:lattice-lap}
    \Delta_{\Z^d}\phi(x)
    =\sum_{k=1}^d\bigl(\phi(x+e_k)+\phi(x-e_k)-2\phi(x)\bigr).
\end{equation}
Our assumptions on the reaction are
\begin{equation}\label{eq:kpp}
\begin{gathered}
    f\in C^2([0,1]),\qquad f(0)=f(1)=0,\qquad a:=f'(0)>0,
    \qquad f'(1)<0,\\
    0<f(s)\leq as\qquad(0<s<1).
\end{gathered}
\end{equation}
The initial datum satisfies
\begin{equation}\label{eq:u0}
    0\leq u_0\leq1,\qquad u_0\not\equiv0,
    \qquad \supp u_0\text{ is finite}.
\end{equation}
In particular, concavity of $f$ is not assumed.

For the Fisher--KPP equation on $\mathbb{R}$, Bramson \cite{Bramson} established the logarithmic delay and convergence
to a translated critical wave. Hamel, Nolen, Roquejoffre and Ryzhik
\cite{HNRR} gave a comparison proof of the delay and identified the
profile along level sets. A PDE proof of convergence to a single wave
for a compact perturbation of a step function was subsequently given
by Nolen, Roquejoffre and Ryzhik \cite{NRRSingleWave}. On $\Z$, Besse,
Faye, Roquejoffre and Zhang
\cite[Theorems~1 and~2]{BesseFayeRoquejoffreZhang} proved the correction
$3/(2\lambda_*)$, with $\lambda_*$ as in \eqref{eq:critical-system},
and convergence to a family of critical-wave translates for nontrivial
data that vanish sufficiently far to the right.

In $\mathbb{R}^d$, G\"artner
\cite[Theorem~3.1]{Gartner} studied the transition region
$G_\varepsilon(t)=\{x:\varepsilon<u(t,x)<1-\varepsilon\}$. For compactly
supported initial data, \cite[Corollary~4.1]{Gartner}, under the additional
integrability condition imposed there on the nonlinearity, removes the
spherical-symmetry assumption and yields a bounded-width annular
localization. In the normalization $u_t=\frac12\Delta u+f(u)$, the
corresponding radius is
\[
    m(t)=v_*t-\frac{d+2}{2v_*}\log t+O(1),
    \qquad v_*=\sqrt{2f'(0)}.
\]
After rescaling the diffusion, this is the scale
$2t-(d+2)\log t/2$ for $u_t=\Delta u+u-u^2$.
A central step in G\"artner's proof is the comparison of Brownian
transition densities killed at a moving boundary with those killed at
its linear interpolation; the estimates are uniform in the starting and
terminal points. Related Brownian moving-boundary first-exit estimates
were obtained by Uchiyama \cite{UchiyamaFirstExit}. Ducrot \cite{Ducrot} later studied the large-time
behavior of Fisher--KPP solutions in $\mathbb{R}^d$ with compactly
supported initial data. Roquejoffre, Rossi and Roussier-Michon \cite{RRR} proved convergence
to the critical profile with a limiting angular shift that is Lipschitz
on the sphere.

The logarithmic delay also occurs in periodic media on $\mathbb{R}$
\cite{HNRRPeriodic}. In periodic media in $\mathbb{R}^d$, Shabani
\cite[Theorem~1.3]{Shabani} gives directional level-set estimates in which the logarithmic
coefficient involves the supporting normal of the propagation set.
For nonlocal diffusion on $\mathbb{R}$, Graham \cite{Graham} proves
the Bramson correction under suitable assumptions on the dispersal
kernel. Boutillon \cite{Boutillon} studies logarithmic delays for the
\emph{linearized} nonlocal equation by large-deviation estimates;
the linear and nonlinear logarithmic scales must be distinguished.
The dependence on the initial tail is also essential: Alfaro, Giletti
and Xiao \cite{AlfaroGilettiXiao} obtain delays or advances for initial
data with a polynomial factor in the critical exponential tail. The present paper assumes finite support throughout.

For heterogeneous KPP equations, Huang and W.~Shen \cite{HuangShen},
W.~Shen \cite{ShenVariational,ShenDiscrete}, Liang, Yi and Zhao
\cite{LiangYiZhao}, and Liang and Zhao \cite{LiangZhao} developed
spreading-speed and traveling-wave theories in periodic, almost periodic,
discrete, and abstract monostable settings. Giletti \cite{Giletti}
proved convergence to minimal-speed pulsating fronts in one-dimensional
heterogeneous KPP equations. Spectral and Liouville-type questions in
periodic KPP and nonlocal time-periodic problems were studied by Vo
\cite{VoLiouville} and Zhongwei Shen--Vo \cite{ShenVo}. Here we use
the translation-invariant structure of $\Delta_{\Z^d}$ in \eqref{eq:lattice-lap}.

On $\Z^d$, the Laplacian is not rotationally invariant, so G\"artner's
radial reduction does not apply. We use the supporting hyperplanes in \eqref{eq:wulff-duality}
and exponentially tilted lattice walks. The positive
whole-space kernel contributes $t^{-d/2}$, but the sharp nonlinear upper
bound requires two additional factors $t^{-1/2}$: one from decay of the
critical weighted mass and one from survival above an affine supporting
half-space. For the lower bound, diffusion in the $d-1$ transverse coordinates
produces the loss $(d-1)/(2(t+T))$ in
\eqref{eq:time-dependent-logistic}, with the fixed shift $T\geq1$
chosen in Lemma~\ref{lem:product-subsolution}. The upper estimate
of Theorem~\ref{thm:sharp-directional-upper} and the lower estimate of
Theorem~\ref{thm:nonlinear-lower} therefore match on every signed
coordinate axis. In arbitrary directions
we prove a logarithmically corrected upper bound, but no matching lower
bound is claimed.

Related first-passage results for branching random walks were obtained
by Blanchet, Cai, Mohanty and Zhang \cite{BlanchetCaiMohantyZhang}, and
by Blanchet and Zhang \cite{BlanchetZhang}, who treat discrete-time non-lattice walks. Neither branching-random-walk result is used below: the only stochastic representation is \eqref{eq:two-semigroup-conventions} for the linear nearest-neighbor semigroup, not for the nonlinear reaction $f$.

For the probability arguments in Sections~\ref{sec:kernels}
and~\ref{sec:sharp-upper}, we refer to Norris
\cite[Chapters~2--4 and Section~6.5]{Norris} for continuous-time Markov
chains, Poisson constructions, semigroups, martingales, stopping times,
and the strong Markov property; to Lawler and Limic
\cite[Chapter~2]{LawlerLimic} for lattice local central limit estimates;
and to Spitzer \cite{Spitzer} for one-dimensional random-walk fluctuation
and first-passage theory. The parameter-uniform tilted estimates required
in Sections~\ref{sec:kernels} and~\ref{sec:sharp-upper} are proved here
rather than imported from those references.

Write $x=(j,y)$, with $j\in\Z$ and $y\in\Z^{d-1}$.
The notation $|\cdot|$ denotes the Euclidean norm in any dimension.
Define
\begin{equation}\label{eq:c-lambda-def}
    c_*:=\min_{\lambda>0}
       \frac{2(\cosh\lambda-1)+a}{\lambda}.
\end{equation}
The minimum is attained at a unique $\lambda_*>0$, and
\begin{equation}\label{eq:critical-system}
\left\{
\begin{aligned}
    c_*\lambda_*&=2(\cosh\lambda_*-1)+a,\\
    c_*&=2\sinh\lambda_*.
\end{aligned}
\right.
\end{equation}
These facts are proved in Lemma~\ref{lem:critical-pair}. Set
\begin{equation}\label{eq:kappa}
    D_*:=\cosh\lambda_*,\qquad
    \kappa_d:=\frac d2+1,\qquad
    \delta_*:=\frac{c_*}{4D_*}=\frac12\tanh\lambda_*.
\end{equation}
For $\theta\in(0,1)$, define the axial level position by
\begin{equation}\label{eq:level-definition}
    R_\theta(t):=\sup\{j\in\Z:u(t,je_1)\geq\theta\},
    \qquad \sup\varnothing:=-\infty.
\end{equation}
Set
\begin{equation}\label{eq:m-scale}
    m_d(t):=c_*t-\frac{\kappa_d}{\lambda_*}\log t,
    \qquad t\geq2.
\end{equation}

Theorem~\ref{thm:nonlinear-target} bounds
$|R_\theta(t)-m_d(t)|$ and localizes the axial transition region
between the levels $\varepsilon$ and $1-\varepsilon$.

\begin{theorem}\label{thm:nonlinear-target}
Assume \eqref{eq:kpp}--\eqref{eq:u0}. Then for every $\theta\in(0,1)$ there
are $C_\theta>0$ and $T_\theta\geq2$ such that
\begin{equation}\label{eq:bramson-target}
    \left|R_\theta(t)-c_*t+
                \frac{d+2}{2\lambda_*}\log t\right|
    \leq C_\theta\qquad(t\geq T_\theta).
\end{equation}
Moreover, for every $\varepsilon\in(0,1/2)$ there are
$C_\varepsilon^{\rm tr}>0$ and $T_\varepsilon\geq2$ such that, for
$t\geq T_\varepsilon$,
\[
    \inf_{0\leq j\leq m_d(t)-C_\varepsilon^{\rm tr}}u(t,je_1)
       \geq1-\varepsilon,
    \qquad
    \sup_{j\geq m_d(t)+C_\varepsilon^{\rm tr}}u(t,je_1)<\varepsilon.
\]
\end{theorem}

For finitely supported data on $\Z^d$, the logarithmic coefficient increases from $3/(2\lambda_*)$ to
$(d+2)/(2\lambda_*)$. The constants $C_\theta,T_\theta$ may depend
on $d,f,u_0,\theta$.
For $\varepsilon\in(0,1/2)$ define the axial transition set
\[
    \mathcal G^{\rm ax}_\varepsilon(t)
      :=\{j\in\Z_{\geq0}:\varepsilon<u(t,je_1)<1-\varepsilon\}.
\]
The second assertion of Theorem~\ref{thm:nonlinear-target} gives
\[
    \mathcal G^{\rm ax}_\varepsilon(t)
      \subset [m_d(t)-C_\varepsilon^{\rm tr},m_d(t)+C_\varepsilon^{\rm tr}]\cap\Z
\]
for all sufficiently large $t$. Thus the transition region has bounded
width along each coordinate axis.
Let $R$ be any signed permutation matrix, meaning a matrix with exactly
one nonzero entry, equal to $1$ or $-1$, in each row and each column. If
$u_R(t,x):=u(t,Rx)$, then $\Delta_{\Z^d}(u\circ R)=(\Delta_{\Z^d}u)\circ R$
and $u_R(0,\cdot)=u_0\circ R$ still satisfies \eqref{eq:u0}. Since there are finitely many
signed permutation matrices, we may take the maximum of their transition
widths and time thresholds. The function $u_R$ satisfies the hypotheses of
Theorem~\ref{thm:nonlinear-target}; applying its transition-zone
inequality to $u_R(t,je_1)=u(t,jRe_1)$ gives
for $j\in\Z_{\geq0}$, $\varepsilon<u(t,jRe_1)<1-\varepsilon$
only when
$j\in[m_d(t)-C_\varepsilon^{\rm tr},m_d(t)+C_\varepsilon^{\rm tr}]$ for all large $t$.  We next prove convergence to the critical wave
on fixed-width half-tubes about these axes.

Let $\phi_*$ denote the decreasing critical wave, normalized by
$\phi_*(0)=1/2$. Thus
\begin{equation}\label{eq:critical-wave}
\left\{
\begin{aligned}
 -c_*\phi_*'(r)&=\phi_*(r-1)-2\phi_*(r)+\phi_*(r+1)+f(\phi_*(r)),
 \qquad r\in\R,\\
 \phi_*(-\infty)&=1,\qquad \phi_*(+\infty)=0.
\end{aligned}
\right.
\end{equation}
For existence and strict monotonicity, see
\cite[Section~1, equations~(1.9)--(1.10)]{BesseFayeRoquejoffreZhang};
for uniqueness up to translation (and precise wave-tail asymptotics), see
\cite{ChenFuGuo}.
In particular, $\phi_*^{-1}:(0,1)\to\R$ is well defined.
By an axial half-tube we mean a set
$\{(j,y)\in\Z\times\Z^{d-1}:j\geq0,\ |y|\leq B\}$
with fixed $B\geq0$.

\begin{theorem}\label{thm:axial-profile}
Assume \eqref{eq:kpp}--\eqref{eq:u0}. There are $C_\zeta>0$ and a function
$\zeta:[2,\infty)\to[-C_\zeta,C_\zeta]$ such that, for every $B\geq0$,
\begin{equation}\label{eq:axial-profile-convergence}
 \lim_{t\to\infty}
 \sup_{\substack{j\in\Z,\ j\geq0\\y\in\Z^{d-1},\ |y|\leq B}}
 \left|u(t,j,y)-\phi_*\bigl(j-m_d(t)+\zeta(t)\bigr)\right|=0.
\end{equation}
The function $\zeta$ is independent of the tube width $B$. It also satisfies,
for every $S>0$,
\begin{equation}\label{eq:phase-slow-variation}
 \lim_{t\to\infty}\sup_{|s|\leq S}|\zeta(t+s)-\zeta(t)|=0.
\end{equation}
Moreover, let $\theta\in(0,1)$ and let $t_n\to\infty$ and
$j_n\in\Z_{\geq0}$ satisfy $u(t_n,j_n,0)\to\theta$. Then, for
every $S,L>0$,
\begin{equation}\label{eq:level-profile-convergence}
 \lim_{n\to\infty}
 \sup_{\substack{|s|\leq S\\k\in\Z,\ y\in\Z^{d-1},\ |k|+|y|\leq L}}
 \left|u(t_n+s,j_n+k,y)
       -\phi_*\bigl(k-c_*s+\phi_*^{-1}(\theta)\bigr)\right|=0.
\end{equation}
For every signed permutation matrix $R$,
\eqref{eq:axial-profile-convergence}--\eqref{eq:level-profile-convergence}
also hold for $u_R(t,x):=u(t,Rx)$; equivalently, they hold along the
signed coordinate axis $Re_1$, with its own bounded phase function.
\end{theorem}

This is the axial counterpart of
\cite[Theorem~2]{BesseFayeRoquejoffreZhang} for data localized on
$\Z^d$. The logarithmic coefficient is $(d+2)/(2\lambda_*)$, and
the limiting wave solves the one-dimensional problem \eqref{eq:critical-wave}.
The proof identifies limits on $\Z^d$ before restricting them to the axis. It does not assume transverse independence of
the initial datum. Neither \eqref{eq:phase-slow-variation} nor
\eqref{eq:axial-profile-convergence} asserts that $\zeta(t)$ has a limit.
The lower estimates used in the proof, including transverse sections
of width $O(\sqrt t)$, are stated in
Theorem~\ref{thm:nonlinear-lower}.

Theorem~\ref{thm:sharp-directional-upper} gives an upper bound in every
observation direction. We distinguish the observation direction from
the supporting normal. Put
\begin{equation}\label{eq:H}
    H_0(p):=2\sum_{k=1}^d(\cosh p_k-1),\qquad H(p):=H_0(p)+a,
    \qquad p\in\R^d,
\end{equation}
and let $I$ be the Legendre transform of $H_0$. An explicit expression is
\begin{equation}\label{eq:I}
    I(v)=\sum_{k=1}^d
       \left[v_k\arsinh\!\left(\frac{v_k}{2}\right)
                  -\sqrt{4+v_k^2}+2\right].
\end{equation}
For $e\in\mathbb S^{d-1}$, where $\mathbb S^{d-1}$ denotes the
Euclidean unit sphere in $\R^d$, let $w_*(e)>0$ be the unique solution of
$I(w_*(e)e)=a$, and define
\begin{equation}\label{eq:pstar}
\begin{gathered}
    v_e:=w_*(e)e,\qquad
    p_e:=\nabla I(v_e),\qquad
    \alpha_e:=p_e\cdot e,\\
    A_e:=\diag(\cosh(p_{e,1}),\ldots,\cosh(p_{e,d})).
\end{gathered}
\end{equation}
In particular, $p_{e,k}=\arsinh(v_{e,k}/2)$ and $\alpha_e>0$.
The maps $e\mapsto w_*(e)$, $e\mapsto v_e$, $e\mapsto p_e$,
$e\mapsto\alpha_e$, and $e\mapsto A_e$ are continuous. In particular,
Lemma~\ref{lem:critical-pair} proves that
$\mathcal P:=\{p_e:e\in\mathbb S^{d-1}\}$ is compact and
$\inf_e\alpha_e>0$. For $p\in\mathcal P$, write $v_p:=\nabla H_0(p)$.

\begin{theorem}\label{thm:sharp-directional-upper}
Assume \eqref{eq:kpp}--\eqref{eq:u0}. There are $C_{\rm mass}>0$ and $C_{\rm up}\geq1$, depending
only on $d,f,u_0$, such that, for every $e\in\mathbb S^{d-1}$,
\begin{equation}\label{eq:critical-mass-main}
    \sum_{x\in\Z^d}\ee^{p_e\cdot(x-tv_e)}u(t,x)
       \leq C_{\rm mass}(1+t)^{-1/2}\qquad(t\geq0).
\end{equation}
For $t\geq4$ and $x\in\Z^d$, put
$z_e(t,x)=p_e\cdot(x-tv_e)+\kappa_d\log t$. Whenever $z_e(t,x)\geq0$,
\begin{equation}\label{eq:sharp-directional-envelope}
    u(t,x)\leq C_{\rm up}(1+z_e(t,x))\ee^{-z_e(t,x)}.
\end{equation}
In particular, for every $\theta\in(0,1)$ and $\rho\geq0$ there is
$L_{\theta,\rho}>0$, independent of $e$, such that
\begin{equation}\label{eq:sharp-directional-location}
    u(t,x)<\theta
    \quad\text{if }t\geq4,\quad |x-re|\leq\rho,\quad r\geq0,
    \quad r\geq w_*(e)t-\frac{\kappa_d}{\alpha_e}\log t
                                      +L_{\theta,\rho}.
\end{equation}
\end{theorem}

Estimate \eqref{eq:sharp-directional-envelope} holds on an entire
supporting half-space, not just in a neighborhood of a ray. When
$e=e_1$, the weight does not depend on $y$; hence the sharp axial upper
bound is uniform over all transverse sites. Combining it with
Theorem~\ref{thm:nonlinear-lower} proves
Theorem~\ref{thm:nonlinear-target}.

The lower construction uses a local expansion of the signed
solution $w$ in Theorem~\ref{thm:linear-main}, with odd longitudinal
initial data. This expansion concerns a signed solution, not a positive
Dirichlet kernel. The critically tilted operator is
\begin{equation}\label{eq:Lstar}
\begin{split}
    (\mathcal L_*w)(j,y)
    & =\ee^{\lambda_*}\bigl(w(j-1,y)-w(j,y)\bigr)\\
    &\quad+\ee^{-\lambda_*}\bigl(w(j+1,y)-w(j,y)\bigr)
          +\Delta_{\Z^{d-1}}w(j,y).
\end{split}
\end{equation}

\begin{theorem}\label{thm:linear-main}
Let $g:\Z\to\R$ be nonzero, odd and finitely supported, with $g(\ell)\geq0$
for $\ell\geq1$. Let $h:\Z^{d-1}\to[0,\infty)$ be nonzero and finitely
supported. Let $w$ solve
\begin{equation}\label{eq:linear-conjugated}
\left\{
\begin{aligned}
    w_t&=\mathcal L_*w,\\
    w(0,j,y)&=g(j)h(y).
\end{aligned}
\right.
\end{equation}
Define
\[
    M_1:=\sum_{\ell\geq1}\ell g(\ell),\qquad
    M_\perp:=\sum_{z\in\Z^{d-1}}h(z),\qquad
    K_*:=\frac{M_1M_\perp}{(4\pi)^{d/2}D_*^{3/2}}.
\]
Then $K_*>0$. For every $A>0$, there are $C_A^{\rm err}>0$ and $T_A\geq1$ such that,
with $s=j-c_*t$,
\begin{equation}\label{eq:wd-asymptotic}
\begin{split}
 \left|w(t,j,y)
   -K_*\frac{s+\delta_*}{t^{\kappa_d}}
       \exp\left(-\frac{s^2}{4D_*t}-\frac{|y|^2}{4t}\right)\right|
 \leq C_A^{\rm err}\frac{1+|s|}{t^{\kappa_d+1/2}}
\end{split}
\end{equation}
whenever $t\geq T_A$, $|s|\leq A\sqrt t$ and $|y|\leq A\sqrt t$.
Consequently, there is $C_A^{\rm bd}\geq1$ such that, after increasing
$T_A$ if necessary,
\begin{equation}\label{eq:wd-two-sided}
    (C_A^{\rm bd})^{-1}\frac{1+s}{t^{\kappa_d}}
    \leq w(t,j,y)\leq
    C_A^{\rm bd}\frac{1+s}{t^{\kappa_d}}
\end{equation}
for $t\geq T_A$, $0\leq s\leq A\sqrt t$ and $|y|\leq A\sqrt t$.
The constants may depend on $A,d,\lambda_*,g,h$, but not on $t,j,y$.
\end{theorem}

The constant $\delta_*$ contributes already when $|j-c_*t|=O(1)$.
In Lemma~\ref{lem:fourier-expansion} it is produced by the cubic term in
the tilted Fourier symbol, and Lemma~\ref{lem:longitudinal} retains that
term before Fourier inversion.

For the linear problem \eqref{eq:linear-original}, the logarithmic
coefficient is $d/(2\alpha_e)$.
Theorem~\ref{thm:nonlinear-upper} records this linear scale and the
nonlinear upper bound obtained from $u\leq V$.

\begin{theorem}\label{thm:nonlinear-upper}
Assume \eqref{eq:kpp}--\eqref{eq:u0}, and let $u$ solve \eqref{eq:main}.
For every $\rho\geq0$, there is $C_\rho\geq1$, depending only on
$d,f,u_0,\rho$, such that
\begin{equation}\label{eq:directional-upper}
    u(t,x)\leq
    \min\left\{1,
       C_\rho t^{-d/2}
       \exp\bigl(-\alpha_e(r-w_*(e)t)\bigr)\right\}
\end{equation}
for every $t\geq1$, $e\in\mathbb S^{d-1}$, $r\geq0$ and $x\in\Z^d$
with $|x-re|\leq\rho$. Thus, for every $\theta\in(0,1)$, there is
$L_{\theta,\rho}>0$ such that $u(t,x)<\theta$ whenever $t\geq1$,
$e\in\mathbb S^{d-1}$, $r\geq0$, $|x-re|\leq\rho$, and
\begin{equation}\label{eq:linear-scale-upper}
    r\geq w_*(e)t-\frac{d}{2\alpha_e}\log t+L_{\theta,\rho}.
\end{equation}

Let $V$ solve the linear problem
\begin{equation}\label{eq:linear-original}
\left\{
\begin{aligned}
    V_t&=\Delta_{\Z^d}V+aV,\\
    V(0)&=u_0.
\end{aligned}
\right.
\end{equation}
For $B>0$, set
\[
    q=q(t,x,e):=x-v_et+\frac{d}{2\alpha_e}e\log t.
\]
Then, uniformly in $e\in\mathbb S^{d-1}$ and $x\in\Z^d$ satisfying
$|q|\leq B$,
\begin{equation}\label{eq:linear-profile}
    V(t,x)=
    \frac{M(e)}{(4\pi)^{d/2}\sqrt{\det A_e}}
        \ee^{-p_e\cdot q}
        \left(1+O_B\!\left(\frac{(1+\log t)^2}{t}\right)\right),
    \qquad t\to\infty,
\end{equation}
where $M(e):=\sum_{z\in\Z^d}\ee^{p_e\cdot z}u_0(z)>0$.
\end{theorem}

The coefficient $d/(2\alpha_e)$ is the linear coefficient. The sharper
nonlinear estimate in Theorem~\ref{thm:sharp-directional-upper} moves
this upper bound back by an additional $\alpha_e^{-1}\log t$.
Estimate \eqref{eq:linear-profile} applies to every lattice site with
$|q|\leq B$, including when $e$ is an irrational direction.

The proof of Theorem~\ref{thm:sharp-directional-upper} uses, for
$p\in\mathcal P$,
$W_p(t,x):=\ee^{p\cdot(x-tv_p)}u(t,x)$ and its mass
$M_p(t):=\sum_{x\in\Z^d}W_p(t,x)$.
The operator $\mathcal L_p$ is defined in \eqref{eq:general-generator}.
The identity $\partial_tW_p=\mathcal L_pW_p-\mathcal R_p$ in
\eqref{eq:weighted-absorption-equation} has the nonnegative term
$\mathcal R_p(t,x):=\ee^{p\cdot(x-tv_p)}(au(t,x)-f(u(t,x)))$.
The mass balance \eqref{eq:weighted-mass-balance}, the first-moment
balance \eqref{eq:weighted-moment-balance}, and the projected crossing
estimate \eqref{eq:crossing-probability} imply the dyadic contraction
\eqref{eq:mass-dyadic-contraction}. Its first iteration gives
$M_p(t)\leq C(1+\log(1+t))(1+t)^{-1/2}$; inserting this bound into
\eqref{eq:weighted-moment-balance} makes the positive first moment
uniformly bounded, and the second iteration gives
\eqref{eq:critical-mass-decay}. For the pointwise upper bound, the
backward tilted walk is stopped at the affine half-space
\eqref{eq:terminal-halfspace}. Lemma~\ref{lem:halfspace-survival}, the
kernel estimate \eqref{eq:kernel-sup}, and
\eqref{eq:critical-mass-decay} are combined in
\eqref{eq:stopped-upper-bound} to obtain
\eqref{eq:sharp-directional-envelope}; the exit value is evaluated at
the site reached by the walk and bounded using $u\leq1$.

For the lower bound, Lemma~\ref{lem:logconcavity} determines the sign of
the solution with odd longitudinal initial data, and
\eqref{eq:integrable-positive-mode} gives the integrable bound used to
construct the factor $\chi(t)$ in the proof of
Lemma~\ref{lem:autonomous-leading-lower}. Lemma~\ref{lem:delayed-pulse}
then supplies the explicit moving subsolution
$\Phi_\varepsilon(j-c_*t+b_\beta\log(t+T))$, with
$\Phi_\varepsilon,b_\beta,T$ specified in that lemma.
The product subsolution
\eqref{eq:product-subsolution} inserts the transverse heat kernel, and
Lemma~\ref{lem:amplification} converts the positive seed
\eqref{eq:diffusive-section-seed} into the axial interval estimate
\eqref{eq:axial-inner-lower} and the diffusive transverse estimate
\eqref{eq:transverse-section-lower}. No multidimensional zero-number
principle is used.

Sections~\ref{sec:preliminaries}--\ref{sec:upper} collect the
preliminary estimates and prove the first upper bounds.
Section~\ref{sec:sharp-upper} proves the upper bound with the logarithmic
correction. Sections~\ref{sec:lower} and~\ref{sec:amplification} prove
the lower bound on the coordinate axes. Section~\ref{sec:profiles} proves
convergence to the critical wave. Section~\ref{sec:remaining} gives the
weighted mass estimates and the bounds in the remaining directions.

\section{Preliminaries}\label{sec:preliminaries}

We write $|x|_1:=\sum_k|x_k|$ and
$|x|_\infty:=\max_k|x_k|$.
We write $\N=\{1,2,\ldots\}$ and $\1_E$ for the
indicator of a set or event $E$. The sequence spaces $\ell^1$ and
$\ell^\infty$ use counting measure, and $\norm{\cdot}_\infty$ denotes
the supremum norm. For $r\in\R$, write $r_+=\max\{r,0\}$ and
$r_-=\max\{-r,0\}$. We also write
$\lfloor r\rfloor:=\max\{k\in\Z:k\leq r\}$ and
$\lceil r\rceil:=\min\{k\in\Z:k\geq r\}$ for the floor and ceiling of
$r$, respectively. In proofs, $C$ denotes a positive bound that may
increase from one estimate to the next; a parameter subscript on such a
generic constant indicates allowed dependence. Constants named in theorem
or lemma statements, or assigned a value in a proof, are fixed within each
application. We distinguish such constants by subscripts or superscripts
when they occur together, and state any enlargement explicitly. Uniformity
in a parameter is stated where it is used. A subscript in an $O$-term
indicates allowed dependence of the implicit constant on that parameter;
every statement containing such an $O$-term specifies any additional
variables with respect to which the bound is uniform.

For $n\in\N$, use $e_1,\ldots,e_n$ for the coordinate vectors in
$\Z^n$. Let $\Delta_{\Z^n}$ be \eqref{eq:lattice-lap} in dimension
$n$, and define
\[
 (\mathcal B_nh)(x):=\sum_{k=1}^n\bigl(h(x+e_k)+h(x-e_k)\bigr).
\]
The shifts are isometries on both $\ell^1(\Z^n)$ and
$\ell^\infty(\Z^n)$, so $\|\mathcal B_n\|\leq2n$ and
$\|\Delta_{\Z^n}\|\leq4n$ on either space. Since
$\Delta_{\Z^n}=\mathcal B_n-2n\,\mathrm{Id}$, its exponential has the
operator-norm convergent expansion
\[
 \ee^{t\Delta_{\Z^n}}
   =\ee^{-2nt}\sum_{m=0}^{\infty}\frac{t^m}{m!}\mathcal B_n^m
   \qquad(t\geq0).
\]
Define the heat kernel by
$\pheat_n(t,x):=(\ee^{t\Delta_{\Z^n}}\1_{\{0\}})(x)$. The integer
$(\mathcal B_n^m\1_{\{0\}})(x)$ counts the nearest-neighbor paths of
length $m$ from $0$ to $x$; its sum over $x$ is $(2n)^m$.
There is such a path with $m=|x|_1$. Consequently,
\begin{equation}\label{eq:heat-semigroup}
\begin{gathered}
 (\ee^{t\Delta_{\Z^n}}h)(x)
    =\sum_{z\in\Z^n}\pheat_n(t,x-z)h(z)
    \qquad(t\geq0,\ h\in\ell^\infty(\Z^n)),\\
 \sum_{x\in\Z^n}\pheat_n(t,x)=1\quad(t\geq0),
 \qquad
 \pheat_n(t,x)>0\quad(t>0,\ x\in\Z^n).
\end{gathered}
\end{equation}
The convolution converges absolutely, and the nonnegative terms in the
expansion show that $\ee^{t\Delta_{\Z^n}}$ preserves nonnegativity.

For comparison, extend $f$ by
\begin{equation}\label{eq:f-extension}
 \widetilde f(s):=
 \begin{cases}
    as, & s<0,\\
    f(s), & 0\leq s\leq1,\\
    f'(1)(s-1), & s>1.
 \end{cases}
\end{equation}
The values and first derivatives match at both endpoints. Thus
$\widetilde f\in C^1(\R)$ and
$\|\widetilde f'\|_{L^\infty(\R)}=L_f$, where
$L_f:=\|f'\|_{L^\infty([0,1])}$. In particular, $\widetilde f$ is
globally Lipschitz and agrees with $f$ on $[0,1]$. Assumption
\eqref{eq:kpp} also gives $\widetilde f(s)\leq as$ for $s\geq0$.

Write $z\sim x$ if $|z-x|=1$. For $n\in\N$ and $D\subset\Z^n$,
define its exterior boundary by
\[
 \partial_{\rm ext}D:=\{z\in\Z^n\setminus D:
                 z\sim x\text{ for some }x\in D\}.
\]
Let $J\subset\R$ be an interval and $D\subset\Z^n$. A bounded
measurable function $U:J\times\Z^n\to\R$ is a subsolution on $D$ if,
for every $x\in D$, $t\mapsto U(t,x)$ is locally absolutely continuous
and
\[
 \partial_tU(t,x)\leq\Delta_{\Z^n}U(t,x)+\widetilde f(U(t,x))
\]
for almost every $t\in J$; a supersolution satisfies the reverse
inequality. The Laplacian at $x\in D$ includes all nearest neighbors,
including those in $\partial_{\rm ext}D$. Values there are the exterior
boundary data; the differential inequality is imposed only at sites
in $D$. For functions with values in $[0,1]$, the reaction is $f$.
For $z_t=\Delta_{\Z^n}z+F(t,z)$, a bounded measurable subsolution
satisfies $z_t\leq\Delta_{\Z^n}z+F(t,z)$ almost everywhere in time
at each site in $D$, and a supersolution satisfies the reverse inequality.
Both are locally absolutely continuous in time at every site in $D$.

Let $E$ be a Banach space, $J\subset\R$ an interval, and $A:E\to E$
a bounded linear operator. If $U\in C^1(J;E)$ and $G\in C(J;E)$ satisfy
$U'=AU+G$, then for $s,t\in J$ with $s<t$,
\begin{equation}\label{eq:variation-constants}
 U(t)=\ee^{(t-s)A}U(s)
      +\int_s^t\ee^{(t-r)A}G(r)\,\dd r.
\end{equation}
Indeed, differentiation of $r\mapsto\ee^{(t-r)A}U(r)$ gives
$\ee^{(t-r)A}G(r)$, which can be integrated from $s$ to $t$.
Equation~\eqref{eq:variation-constants} also holds if $G$ is Bochner
integrable on $[s,t]$ and
$U$ is a continuous solution of
$U(r)=U(s)+\int_s^r(AU(q)+G(q))\,\dd q$.
To see this, substitute the right-hand side of
\eqref{eq:variation-constants} into this integral equation and
interchange the integrals; boundedness of $A$ and integrability of $G$
justify the interchange. For two solutions with the same value at $s$, the norm of their
difference satisfies Gronwall's inequality with coefficient $\|A\|$,
which proves uniqueness. For the semigroup formulation, see
\cite[Chapters~1 and~4]{Pazy}.

We use Gronwall in the following form. Let $T>0$ and $K\geq0$ be fixed.
If a nonnegative bounded measurable function $m$ on $[0,T]$ satisfies
$m(t)\leq K\int_0^tm(r)\,\dd r$ for every $t\in[0,T]$, then
$m\equiv0$. Indeed, $M(t):=\int_0^tm(r)\,\dd r$ is absolutely
continuous, $M(0)=0$, and $(\ee^{-Kt}M(t))'\leq0$ almost everywhere.
Since $M\geq0$, this gives $M=0$; the assumed inequality then gives
$m(t)=0$ for every $t$.

\begin{lemma}\label{lem:comparison}
Under \eqref{eq:kpp}--\eqref{eq:u0}, problem \eqref{eq:main} has a
unique solution in $C^1([0,\infty);\ell^\infty(\Z^d))$, with
$0<u(t,x)<1$ for $t>0$, and
\begin{equation}\label{eq:linear-comparison}
    (\ee^{t\Delta_{\Z^d}}u_0)(x)
       \leq u(t,x)\leq
    \ee^{at}(\ee^{t\Delta_{\Z^d}}u_0)(x).
\end{equation}

Let $n\in\N$, $T>0$ and $D\subset\Z^n$.
Let $U,W:[0,T]\times\Z^n\to\R$ be bounded measurable functions,
absolutely continuous in time at each $x\in D$. Let $\mathcal I$ be a
bounded interval containing their ranges, and let
$F:[0,T]\times\R\to\R$ be continuous. Assume that, for some $L\geq0$,
$|F(t,r)-F(t,s)|\leq L|r-s|$ for $t\in[0,T]$ and
$r,s\in\mathcal I$.
Suppose that, for each $x\in D$ and almost every $t\in(0,T)$,
\begin{equation}\label{eq:comparison-differential}
 \left\{
 \begin{aligned}
 \partial_tU(t,x)&\leq\Delta_{\Z^n}U(t,x)+F(t,U(t,x)),\\
 \partial_tW(t,x)&\geq\Delta_{\Z^n}W(t,x)+F(t,W(t,x)).
 \end{aligned}
 \right.
\end{equation}
If
\begin{equation}\label{eq:comparison-data}
 \left\{
 \begin{aligned}
 U(0,x)&\leq W(0,x)\qquad(x\in D),\\
 U(t,z)&\leq W(t,z)\qquad(0\leq t\leq T,\ z\in\partial_{\rm ext}D).
 \end{aligned}
 \right.
\end{equation}
then $U\leq W$ on $[0,T]\times D$. This includes $D=\Z^n$, for which
$\partial_{\rm ext}D=\varnothing$, and $F(t,s)=\widetilde f(s)$.

The order $U\leq W$ on $[0,T]\times D$ also holds if $U$ remains
absolutely continuous in time at each $x\in D$ and $W$ instead has
finitely many jump times $0<t_1<\cdots<t_N<T$, common to all sites.
At each $x\in D$, require $W$ to be locally Lipschitz between jumps,
continuous at $0,T$, and to have finite one-sided limits satisfying
$W(t_j,x)=W(t_j+,x)\geq W(t_j-,x)$ for $1\leq j\leq N$.
Both functions must be bounded and measurable, and
\eqref{eq:comparison-differential} must hold almost everywhere between
jumps. Require \eqref{eq:comparison-data}, as well as continuity of $F$
and its uniform Lipschitz bound with constant $L$ on $\mathcal I$.
\end{lemma}

\begin{proof}
Define the vector field
\[
   \mathcal F(U):=\Delta_{\Z^d}U+\widetilde f(U)
   \quad\text{on }\ell^\infty(\Z^d),
\]
where $\widetilde f(U)$ is evaluated pointwise. The bounds
$\|\Delta_{\Z^d}\|_{\ell^\infty\to\ell^\infty}\leq4d$ and
$\|\widetilde f'\|_\infty=L_f$ give
\[
 \|\mathcal F(U)-\mathcal F(V)\|_\infty
       \leq(4d+L_f)\|U-V\|_\infty.
\]
Choose $\tau>0$ with $\tau(4d+L_f)<1$. On
$C([0,\tau];\ell^\infty(\Z^d))$, the map
$V\mapsto u_0+\int_0^\cdot\mathcal F(V(s))\,\dd s$ is a contraction
in the supremum norm. Its fixed point is the unique $C^1$ solution of
problem \eqref{eq:main} with $f$ replaced by $\widetilde f$, initial
datum $u_0$, and time interval $[0,\tau]$. Since $\tau$ is independent
of the bounded initial datum, repeating this fixed-point construction on
successive intervals of length $\tau$ gives a unique solution on
$[0,\infty)$. This is the bounded-operator instance of the semilinear
Cauchy theory in \cite[Chapter~6]{Pazy}.

We now prove $U\leq W$ on $[0,T]\times D$ under
\eqref{eq:comparison-differential}--\eqref{eq:comparison-data}, first
when both functions are absolutely continuous in time at each site
in $D$. Put $G(t,s)=F(t,s)+Ls$. For each $t$, this function is nondecreasing
and $2L$-Lipschitz in $s\in\mathcal I$. At each $x\in D$,
\[
 \partial_tU+(2n+L)U\leq\sum_{z\sim x}U(t,z)+G(t,U),\qquad
 \partial_tW+(2n+L)W\geq\sum_{z\sim x}W(t,z)+G(t,W).
\]
Multiply by $\ee^{(2n+L)t}$ and integrate in time. The initial
difference $U(0,x)-W(0,x)$ is nonpositive. Define
$P=(U-W)_+$ in $D$ and $P=0$ outside $D$. Since $U-W\leq0$ on
$\partial_{\rm ext}D$, subtraction and the properties of $G$ give
\begin{equation}\label{eq:comparison-positive-part}
 P(t,x)\leq\int_0^t\ee^{-(2n+L)(t-r)}
       \left[\sum_{z\sim x}P(r,z)+2LP(r,x)\right]\dd r
       \qquad(x\in D).
\end{equation}
The function $m(t):=\|P(t,\cdot)\|_\infty$ is bounded and measurable,
since $D$ is countable, and satisfies
$m(t)\leq(2n+2L)\int_0^tm(r)\,\dd r$. Gronwall's inequality
gives $m=0$, hence $U\leq W$ in $D$.

If $W$ jumps at $t_1,\ldots,t_N$, put $t_0=0$ and $t_{N+1}=T$.
For $0\leq j\leq N$, integrate \eqref{eq:comparison-differential}
on compact subintervals of $(t_j,t_{j+1})$. Let the left
endpoint decrease to the preceding jump time, or to $0$ on the first
interval. The right-hand limits and boundedness of $U,W,F(t,U),F(t,W)$
give \eqref{eq:comparison-positive-part} with lower integration limit
$t_j$ on $(t_j,t_{j+1})$, or $0$ on the first interval, provided
$U\leq W$ at that initial time. Gronwall gives $U\leq W$ before the next jump. At every
$t_j$ with $1\leq j\leq N$, continuity of $U$ and the left limit
of $W$ give
\[
 U(t_j,x)=U(t_j-,x)\leq W(t_j-,x)\leq W(t_j,x).
\]
This ordered value starts the next interval. Induction over the finitely
many jumps, followed by continuity at $T$, gives $U\leq W$
on $[0,T]\times D$.

Since $\widetilde f(0)=\widetilde f(1)=0$, the functions $0$ and $1$
satisfy \eqref{eq:main} with $f$ replaced by $\widetilde f$.
The pairs $(U,W)=(0,u)$ and $(U,W)=(u,1)$ satisfy
\eqref{eq:comparison-differential}--\eqref{eq:comparison-data} with
$D=\Z^d$ and $F(t,s)=\widetilde f(s)$. The two applications of
$U\leq W$ give $0\leq u\leq1$. Thus
$\widetilde f(u)=f(u)$, and $u$ solves \eqref{eq:main}.
Apply \eqref{eq:variation-constants} with $A=\Delta_{\Z^d}$ and then
with $A=\Delta_{\Z^d}+a$ to obtain
\[
\begin{aligned}
 u(t)&=\ee^{t\Delta_{\Z^d}}u_0
      +\int_0^t\ee^{(t-s)\Delta_{\Z^d}}f(u(s))\,\dd s,\\
 u(t)&=\ee^{t(\Delta_{\Z^d}+a)}u_0
      +\int_0^t\ee^{(t-s)(\Delta_{\Z^d}+a)}
                  \bigl(f(u(s))-au(s)\bigr)\,\dd s.
\end{aligned}
\]
Positivity of the heat semigroup, $f(u)\geq0$, and $f(u)-au\leq0$
give \eqref{eq:linear-comparison}. Since $u_0\not\equiv0$, strict
positivity of the kernel in \eqref{eq:heat-semigroup} gives $u(t,x)>0$
for $t>0$.

Finally, put $v=1-u$. Since $f(1)=0$, the bound on $f'$ gives
$f(1-v)\leq L_fv$ for $0\leq v\leq1$, and hence
$v_t\geq\Delta_{\Z^d}v-L_fv$.
The function $\underline v(t):=\ee^{-L_ft}\ee^{t\Delta_{\Z^d}}(1-u_0)$
solves $\underline v_t=\Delta_{\Z^d}\underline v-L_f\underline v$
with $\underline v(0)=v(0)$. Both $v$ and $\underline v$ are bounded
on finite time intervals. Applying
\eqref{eq:comparison-differential}--\eqref{eq:comparison-data} with
$U=\underline v$, $W=v$, $D=\Z^d$, and $F(t,s)=-L_fs$ gives
$v\geq\underline v$. Since $u_0$ has finite
support, $1-u_0$ is nonzero and nonnegative.
Equation~\eqref{eq:heat-semigroup} therefore gives
$\underline v(t,x)>0$ for $t>0$, proving $u(t,x)<1$.
\end{proof}

Keeping the other assumptions in \eqref{eq:kpp}, the proof of
Theorem~\ref{thm:nonlinear-target} does not use the strict sign $f'(1)<0$;
positivity of $f$ on $(0,1)$ suffices for the amplification step.
The strict sign is used in the stable-state comparison in
Lemma~\ref{lem:entire-positive-comparison}, and hence in the proof of
Theorem~\ref{thm:axial-profile}.

\begin{lemma}\label{lem:moving-halfline-comparison}
Let $T_0<T_1$ and $B\in C^1([T_0,T_1])$ be strictly increasing.
Let $U,W:[T_0,T_1]\times\Z\to\R$ be bounded and locally Lipschitz in
time at each site. Let $F:[T_0,T_1]\times\R\to\R$ be continuous and
uniformly Lipschitz in its second variable on a bounded interval
containing the ranges of $U$ and $W$. Suppose that, for each $j\in\Z$,
\[
 \left\{
 \begin{aligned}
 U_t(t,j)&\leq\Delta_{\Z}U(t,j)+F(t,U(t,j)),\\
 W_t(t,j)&\geq\Delta_{\Z}W(t,j)+F(t,W(t,j)).
 \end{aligned}
 \right.
\]
for almost every $t\in(T_0,T_1)$ with $j\leq B(t)$. If
\[
\left\{
\begin{aligned}
 U(T_0,j)&\leq W(T_0,j)\qquad(j\leq B(T_0)),\\
 U(t,j)&\leq W(t,j)\qquad(T_0\leq t\leq T_1,\ B(t)<j\leq B(t)+1).
\end{aligned}
\right.
\]
then $U(t,j)\leq W(t,j)$ for $T_0\leq t\leq T_1$ and $j\leq B(t)$.
\end{lemma}

\begin{proof}
Partition $[T_0,T_1]$ at all times in $(T_0,T_1)$ when $B$ takes an
integer value. There are finitely many such times because $B$ is
continuous and strictly increasing. On each interval $[a,b)$ of this
partition, $\lfloor B(t)\rfloor=k$ is constant. The domain is
$D_k=\{j\in\Z:j\leq k\}$, with $\partial_{\rm ext}D_k=\{k+1\}$.
For every $t\in(a,b)$, apply Lemma~\ref{lem:comparison} on $[a,t]$,
shifting the initial time to $a$. The assumed inequalities
$U_t\leq\Delta_{\Z}U+F(t,U)$ and $W_t\geq\Delta_{\Z}W+F(t,W)$
hold in $D_k$, and $U(s,k+1)\leq W(s,k+1)$ for $a\leq s\leq t$.
At $a=T_0$, the initial hypothesis gives $U(a,j)\leq W(a,j)$
for $j\in D_k$. At a later partition point with $B(a)=k$, taking
$s\uparrow a$ in $U(s,k)\leq W(s,k)$ gives the order at the new site
$k$, while continuity preserves the order at $j<k$.
Thus $U\leq W$ on $[a,b)\times D_k$. Continuity extends this order
to $\{b\}\times D_k$.

If $B(b)=k+1$, the new interior vertex is $k+1$. For $t<b$ close to
$b$, the assumed exterior inequality gives $U(t,k+1)\leq W(t,k+1)$.
Taking $t\uparrow b$ gives $U(b,k+1)\leq W(b,k+1)$, which supplies
the order on the enlarged domain at time $b$. Induction over the
crossings proves the conclusion, including a crossing at $T_1$.
If $T_1$ is not a crossing time, continuity gives the order there.
\end{proof}

\begin{lemma}\label{lem:critical-pair}
The function $I$ in \eqref{eq:I} is the Legendre transform of $H_0$.
For each $e\in\mathbb S^{d-1}$, the quantities in \eqref{eq:pstar}
are uniquely defined and depend continuously on $e$. They satisfy
\begin{equation}\label{eq:duality}
    \nabla H_0(p_e)=v_e,
    \qquad H(p_e)=p_e\cdot v_e=\alpha_e w_*(e).
\end{equation}
The set $\{p_e:e\in\mathbb S^{d-1}\}$ is compact, and
$\inf_e\alpha_e>0$. In particular, $\inf_e|p_e|>0$.

For $n\in\mathbb S^{d-1}$ define
\[
    c(n):=\min_{\lambda>0}\frac{H(\lambda n)}{\lambda}.
\]
The minimizer $\lambda(n)$ is unique. With
$\mathcal W_a:=\{v\in\R^d:I(v)\leq a\}$, one has
\begin{equation}\label{eq:wulff-duality}
    \mathcal W_a=\bigcap_{n\in\mathbb S^{d-1}}
          \{v:v\cdot n\leq c(n)\},
    \qquad
    w_*(e)=\min_{n\cdot e>0}\frac{c(n)}{n\cdot e}.
\end{equation}
The minimizing normal in the last expression is unique and equals
$n_e=p_e/|p_e|$, with $\lambda(n_e)=|p_e|$.
In particular, $w_*(e_1)=c_*$, $p_{e_1}=\lambda_*e_1$,
and \eqref{eq:critical-system} holds.
\end{lemma}

\begin{proof}
The supremum defining the Legendre transform separates into $d$
one-dimensional suprema. The maximizer in the $k$th coordinate solves
$2\sinh p_k=v_k$, giving \eqref{eq:I}. Direct differentiation yields
\[
    \nabla I(v)_k=\arsinh(v_k/2),\qquad
    D^2I(v)=\diag\bigl((4+v_1^2)^{-1/2},\ldots,
                       (4+v_d^2)^{-1/2}\bigr).
\]
The positive definite Hessian gives strict convexity. Each summand in
\eqref{eq:I} is nonnegative and tends to infinity as the absolute value
of its argument tends to infinity. Since
$\max_k|v_k|\geq|v|/\sqrt d$, this proves $I(v)\to\infty$ as
$|v|\to\infty$. Also $I(0)=0$.
Writing $(e)_k$ for the $k$th component of $e$, we have
\[
    \frac{\dd}{\dd r}I(re)
       =\sum_{k=1}^d(e)_k\arsinh(r(e)_k/2)>0
       \qquad(r>0).
\]
Thus $I(re)=a$ has exactly one solution $r=w_*(e)>0$.
Choose $r_0>0$ so small that $\max_{|v|=r_0}I(v)<a$, and choose
$R_0>0$ so large that $\min_{|v|=R_0}I(v)>a$.
Continuity and the growth of $I$ at infinity justify these choices.
Then $r_0<w_*(e)<R_0$ for every $e\in\mathbb S^{d-1}$.

To prove continuity, let $e_m\to e$. Every subsequence of $w_*(e_m)$
has a further subsequence converging to some $r\in[r_0,R_0]$.
Continuity of $I$ gives $I(re)=a$, so uniqueness implies $r=w_*(e)$.
Hence $w_*(e_m)\to w_*(e)$. The formulas in \eqref{eq:pstar} now
give continuity of $v_e,p_e,\alpha_e$ and $A_e$.
The maximizing vector in $I(v_e)=\sup_p(p\cdot v_e-H_0(p))$ is
$p=p_e$. Therefore $\nabla H_0(p_e)=v_e$ and
$I(v_e)=p_e\cdot v_e-H_0(p_e)=a$, proving \eqref{eq:duality}.
Moreover,
$\alpha_e=\left.\frac{\dd}{\dd r}I(re)\right|_{r=w_*(e)}>0$.
Continuity on the compact sphere gives compactness of $\{p_e\}$ and
$\alpha_0:=\min_e\alpha_e>0$. Since $|p_e|\geq p_e\cdot e=\alpha_e$,
we also have $\inf_e|p_e|\geq\alpha_0$.

For a fixed $n\in\mathbb S^{d-1}$, put
$q_n(\lambda)=H(\lambda n)/\lambda$.
The inequality $2(\cosh s-1)\geq s^2$ gives
$q_n(\lambda)\geq a/\lambda+\lambda$, so $q_n$ tends to infinity
at both endpoints of $(0,\infty)$. Its derivative is
$q_n'(\lambda)=F_n(\lambda)/\lambda^2$, where
\[
 F_n(\lambda):=\lambda n\cdot\nabla H_0(\lambda n)-H_0(\lambda n)-a.
\]
Here $F_n(0)=-a$, and
\[
 F_n'(\lambda)
   =\lambda n^{\mathsf T}D^2H_0(\lambda n)n
   =2\lambda\sum_{k=1}^d n_k^2\cosh(\lambda n_k)
   \geq2\lambda>0.
\]
In particular, $F_n(\lambda)\geq\lambda^2-a$. It has exactly one
zero, which is the unique minimizer $\lambda(n)$ of $q_n$.
Define the support point
$\widehat v_n:=\nabla H_0(\lambda(n)n)$.
The identity $F_n(\lambda(n))=0$ gives
$I(\widehat v_n)=a$ and $\widehat v_n\cdot n=c(n)$.
This point is indexed by the normal $n$; it need not equal the radial
point $v_n=w_*(n)n$ defined in \eqref{eq:pstar}.

If $I(v)\leq a$, the defining inequality for the Legendre transform gives
$p\cdot v\leq H_0(p)+a$ for every $p$. Taking $p=\lambda n$ and
minimizing over $\lambda>0$ gives $v\cdot n\leq c(n)$ for every
$n\in\mathbb S^{d-1}$. Conversely, if $I(v)>a$, put $p=\nabla I(v)$.
Then $p\neq0$, and $p\cdot v-H_0(p)=I(v)>a$ gives
\[
 v\cdot\frac p{|p|}
    >\frac{H(p)}{|p|}\geq c\left(\frac p{|p|}\right).
\]
This proves the first identity in \eqref{eq:wulff-duality}.
Since $\widehat v_n\in\mathcal W_a$ and $\widehat v_n\cdot n=c(n)$,
we also have $\max_{v\in\mathcal W_a}v\cdot n=c(n)$.

For $v_e=w_*(e)e$, the half-space inequalities give
$w_*(e)\leq c(n)/(n\cdot e)$ whenever $n\cdot e>0$.
Set $n_e=p_e/|p_e|$. Equation~\eqref{eq:duality} gives
$F_{n_e}(|p_e|)=0$, so $\lambda(n_e)=|p_e|$.
Also $n_e\cdot e=\alpha_e/|p_e|>0$, and
\[
 \frac{c(n_e)}{n_e\cdot e}
     =\frac{H(p_e)}{\alpha_e}=w_*(e).
\]
If equality holds for another unit vector $n$ with $n\cdot e>0$,
put $p=\lambda(n)n$. Then $p\cdot v_e=H(p)$, so
$p\cdot v_e-H_0(p)=a=I(v_e)$. Uniqueness of the maximizing vector
in the Legendre transform gives $p=p_e$, and hence $n=n_e$.
This proves the second identity in \eqref{eq:wulff-duality} and
uniqueness of its minimizer.

For $n=e_1$, $q_{e_1}(\lambda)$ is the expression minimized in
\eqref{eq:c-lambda-def}. Thus $\lambda(e_1)=\lambda_*$ and $c(e_1)=c_*$.
The equation $F_{e_1}(\lambda_*)=0$ gives $c_*=2\sinh\lambda_*$;
together with the definition of $c_*$, this is
\eqref{eq:critical-system}. Hence $\widehat v_{e_1}=c_*e_1$ and
$I(c_*e_1)=a$. Uniqueness of the radial solution gives
$w_*(e_1)=c_*$ and $p_{e_1}=\lambda_*e_1$.
\end{proof}

\section{Tilted kernel estimates}\label{sec:kernels}

Lemmas~\ref{lem:conjugation} and~\ref{lem:fourier-expansion} hold in every
dimension $d\geq1$. The remaining results use $d\geq2$, as in
Theorem~\ref{thm:linear-main}. We use $H_0$ from \eqref{eq:H} and the
heat kernel $\pheat_d$ defined in Section~\ref{sec:preliminaries}. For
$p\in\R^d$, put
\[
    v(p):=\nabla H_0(p)=(2\sinh p_k)_{k=1}^d,\qquad
    A(p):=\diag(\cosh p_1,\ldots,\cosh p_d),
\]
and define
\begin{equation}\label{eq:general-generator}
    (\mathcal L_p\phi)(x)
      :=\sum_{k=1}^d\left[
          \ee^{p_k}\bigl(\phi(x-e_k)-\phi(x)\bigr)
        + \ee^{-p_k}\bigl(\phi(x+e_k)-\phi(x)\bigr)\right].
\end{equation}
Thus $\mathcal L_{\lambda_*e_1}=\mathcal L_*$. Its adjoint with respect
to counting measure is $\mathcal A_p$, given by
\begin{equation}\label{eq:adjoint-generator}
\begin{split}
 (\mathcal A_p\phi)(x)
   &:=\sum_{k=1}^d\left[\ee^{p_k}\bigl(\phi(x+e_k)-\phi(x)\bigr)
              +\ee^{-p_k}\bigl(\phi(x-e_k)-\phi(x)\bigr)\right],\\
 \sum_x\phi(x)\mathcal L_p\psi(x)
   &=\sum_x\psi(x)\mathcal A_p\phi(x).
\end{split}
\end{equation}
The summation identity holds for $\phi\in\ell^\infty(\Z^d)$ and
$\psi\in\ell^1(\Z^d)$: each sum is absolutely convergent, and changing
$x$ to $x\pm e_k$ proves the identity. Since lattice shifts are
isometries, both operators are bounded on $\ell^1$ and $\ell^\infty$,
with norms at most $4\sum_k\cosh p_k$ for fixed $p$. Their exponentials
are therefore defined by operator-norm convergent power series.
For a walk with jump rates $\ee^{p_k}$ in direction $e_k$ and
$\ee^{-p_k}$ in direction $-e_k$, $\mathcal L_p$ appears in the forward
equation for its probability masses, and $\mathcal A_p$ is its generator
on test functions. The construction in Lemma~\ref{lem:conjugation} and
the expectation formulas \eqref{eq:two-semigroup-conventions} specify
these two uses.

\begin{lemma}\label{lem:conjugation}
For $p\in\R^d$ and $t\geq0$, define
$K_p(t,x):=(\ee^{t\mathcal L_p}\1_{\{0\}})(x)$. Then
\begin{equation}\label{eq:kernel-tilt}
    K_p(t,x)=\ee^{p\cdot x-tH_0(p)}\pheat_d(t,x).
\end{equation}
One has $K_p(0,x)=\1_{\{0\}}(x)$, $K_p(t,x)>0$ for $t>0$, and
$\sum_xK_p(t,x)=1$ for $t\geq0$. For every bounded $h$,
\begin{equation}\label{eq:tilted-convolution}
    (\ee^{t\mathcal L_p}h)(x)
       =\sum_{z\in\Z^d}K_p(t,x-z)h(z).
\end{equation}
For $t\geq1$,
\begin{equation}\label{eq:kernel-sup}
    \sup_{x\in\Z^d}K_p(t,x)\leq C_{{\rm ker},d} t^{-d/2},
\end{equation}
where $C_{{\rm ker},d}$ can be chosen independently of $p\in\R^d$.
For the solution $V$ of \eqref{eq:linear-original},
\begin{equation}\label{eq:conjugation}
    V(t,x)=\ee^{-p\cdot x+tH(p)}
          \bigl(\ee^{t\mathcal L_p}g_p\bigr)(x),
    \qquad g_p(z):=\ee^{p\cdot z}u_0(z).
\end{equation}
\end{lemma}

\begin{proof}
Let $N_1^+,N_1^-,\ldots,N_d^+,N_d^-$ be mutually independent
Poisson processes of rate $1$, with $N_k^\pm(0)=0$ for $1\leq k\leq d$.
Define
\[
    X(t)=\sum_{k=1}^d\bigl(N_k^+(t)-N_k^-(t)\bigr)e_k.
\]
A jump of $N_k^+$ moves $X$ by $e_k$, and a jump of $N_k^-$ moves it
by $-e_k$. Thus $X(0)=0$, and $X$ jumps from $x$ to each neighbor
$x\pm e_k$ at rate $1$, so its generator is $\Delta_{\Z^d}$.
The Poisson-clock construction of this walk, the associated jump
chain/holding times, and the forward/backward semigroup equations are
covered in \cite[Sections~2.4, 2.6 and~2.8]{Norris}.
For each $k$ and each sign,
$\mathbb E\ee^{\theta N_k^\pm(t)}=\exp(t(\ee^\theta-1))$ for
$\theta\in\R$. Independence therefore gives
$\mathbb E\ee^{p\cdot X(t)}=\ee^{tH_0(p)}$.
For $p\in\R^d$, let $N_k^{p,+},N_k^{p,-}$, $1\leq k\leq d$, be
mutually independent Poisson processes starting from zero, with rates
$\ee^{p_k},\ee^{-p_k}$, respectively, and define
$X^p(t):=\sum_{k=1}^d(N_k^{p,+}(t)-N_k^{p,-}(t))e_k$.
For $t>0$, compare the probabilities of the count vector
$(n_1^+,n_1^-,\ldots,n_d^+,n_d^-)\in\Z_{\geq0}^{2d}$ under the tilted
and rate-one laws. With $x_n:=\sum_k(n_k^+-n_k^-)e_k$, their ratio is
\[
 \prod_{k=1}^d
 \frac{\ee^{-t\ee^{p_k}}(t\ee^{p_k})^{n_k^+}/n_k^+!}
      {\ee^{-t}t^{n_k^+}/n_k^+!}
 \frac{\ee^{-t\ee^{-p_k}}(t\ee^{-p_k})^{n_k^-}/n_k^-!}
      {\ee^{-t}t^{n_k^-}/n_k^-!}
 =\ee^{p\cdot x_n-tH_0(p)}.
\]
Summing over the count vectors with $x_n=x$ gives
$q_p(t,x):=\mathbb P\{X^p(t)=x\}
=\ee^{p\cdot x-tH_0(p)}\pheat_d(t,x)$.
The jump rates give the forward equation
$\partial_tq_p=\mathcal L_pq_p$, with $q_p(0)=\1_{\{0\}}$.
Since $\mathcal L_p$ is bounded on $\ell^1$, uniqueness in
\eqref{eq:variation-constants} identifies $q_p$ with $K_p$ and proves
\eqref{eq:kernel-tilt}. The law $q_p(t,\cdot)$ has total mass one.
For $t>0$, the count vector $n_k^+=(x_k)_+$, $n_k^-=(-x_k)_+$ has
positive probability, so $K_p(t,x)>0$. Translation invariance of
$\mathcal L_p$ gives \eqref{eq:tilted-convolution}: its right-hand side
solves $U_t=\mathcal L_pU$, $U(0)=h$, in $\ell^\infty$.
The derivatives and sums can be interchanged because
$K_p\in C^1([0,\infty);\ell^1)$ and $h$ is bounded; uniqueness in
\eqref{eq:variation-constants} then identifies $U=\ee^{t\mathcal L_p}h$.

For $\xi\in[-\pi,\pi]^d$, independence of the Poisson processes gives
$\mathbb E\ee^{-\ii\xi\cdot X^p(t)}=\ee^{t\Lambda_p(\xi)}$, where
\begin{equation}\label{eq:symbol}
    \Lambda_p(\xi)
      =\sum_{k=1}^d\left[
          \ee^{p_k}(\ee^{-\ii\xi_k}-1)
          +\ee^{-p_k}(\ee^{\ii\xi_k}-1)\right].
\end{equation}
Fourier inversion for $K_p(t,\cdot)\in\ell^1$ gives
\begin{equation}\label{eq:fourier-kernel}
    K_p(t,x)=\frac1{(2\pi)^d}\int_{[-\pi,\pi]^d}
               \ee^{t\Lambda_p(\xi)}\ee^{\ii x\cdot\xi}\dd\xi.
\end{equation}
Since $\Re\Lambda_p(\xi)=-2\sum_k\cosh(p_k)(1-\cos\xi_k)$,
$1-\cos\xi_k\geq2\xi_k^2/\pi^2$ and $\cosh p_k\geq1$ imply
\[
 K_p(t,x)\leq\frac1{(2\pi)^d}\int_{\R^d}
                \ee^{-4t|\xi|^2/\pi^2}\dd\xi
       =\left(\frac{\sqrt\pi}{4}\right)^d t^{-d/2}.
\]
Thus \eqref{eq:kernel-sup} holds with
$C_{{\rm ker},d}=(\sqrt\pi/4)^d$, independently of $p$.
Finally, insert \eqref{eq:kernel-tilt} into
$V(t,x)=\ee^{at}\sum_z\pheat_d(t,x-z)u_0(z)$ and use
\eqref{eq:tilted-convolution}. Since $H(p)=H_0(p)+a$, this gives
\eqref{eq:conjugation}. The initial datum $g_p$ has finite support,
so the calculation is valid even though multiplication by
$\ee^{p\cdot x}$ is unbounded on $\ell^\infty$ when $p\neq0$.
\end{proof}

Lawler and Limic \cite[Chapter~2]{LawlerLimic} prove local central limit
estimates for lattice random walks.  Lemma~\ref{lem:fourier-expansion} treats the continuous-time walk
$X^p$, uniformly for $p$ in a compact set, and includes the cubic
Fourier term.
For a positive diagonal matrix $A$, let
\[
    \gamma_A(z):=\frac1{(4\pi)^{d/2}\sqrt{\det A}}
                    \exp\left(-\frac14z^{\mathsf T}A^{-1}z\right).
\]
We write $\gamma_p=\gamma_{A(p)}$. Derivatives of $\gamma_p$
are taken with respect to $z$.
For a finitely supported real function $g$ on $\Z^d$, define
\[
    M:=\sum_x g(x),\qquad m_i:=\sum_x x_i g(x),\qquad
    Q_{ij}:=\sum_x x_ix_jg(x).
\]
Write $m=(m_1,\ldots,m_d)$ and $Q=(Q_{ij})_{1\leq i,j\leq d}$.

\begin{lemma}\label{lem:fourier-expansion}
Let $P\subset\R^d$ be compact, and let $g:\Z^d\to\R$ have finite
support.
Put $z=(x-tv(p))/\sqrt t$. As $t\to\infty$, uniformly in
$p\in P$ and $x\in\Z^d$,
\begin{equation}\label{eq:mass-expansion}
\begin{split}
    (\ee^{t\mathcal L_p}g)(x)
     &=t^{-d/2}M\gamma_p(z)\\
     &\quad-t^{-(d+1)/2}\left[
           \sum_i m_i\partial_i\gamma_p(z)
           +\frac M6\sum_k v_k(p)\partial_k^3\gamma_p(z)\right]
           +O\bigl(t^{-(d+2)/2}\bigr).
\end{split}
\end{equation}
If $M=0$, the sharper expansion is
\begin{equation}\label{eq:zero-mass-expansion}
\begin{split}
    (\ee^{t\mathcal L_p}g)(x)
     &=-t^{-(d+1)/2}\sum_i m_i\partial_i\gamma_p(z)\\
     &\quad+t^{-(d+2)/2}\left[
          \frac12\sum_{i,j}Q_{ij}\partial_{ij}\gamma_p(z)
          +\frac16\sum_{i,k}m_i v_k(p)
                              \partial_i\partial_k^3\gamma_p(z)\right]
          +O\bigl(t^{-(d+3)/2}\bigr).
\end{split}
\end{equation}
For every finite set $S\subset\Z^d$ and $B>0$, the error constants may
be chosen to depend only on $d,P,S,B$ whenever $\supp g\subset S$ and
$\|g\|_{\ell^1}\leq B$. In particular, $g$ may vary with $p$ subject to
these two bounds.
\end{lemma}

\begin{proof}
Fix a finite set $S\subset\Z^d$ and $B>0$, and assume
$\supp g\subset S$ and $\|g\|_{\ell^1}\leq B$. Throughout this proof,
$C$ may depend on $d,P,S,B$, but not on $t,p,x$ or the choice of $g$ within
these bounds. With
$\widehat g(\xi)=\sum_xg(x)\ee^{-\ii x\cdot\xi}$, Taylor's formula gives
\begin{equation}\label{eq:data-expansion}
    \widehat g(\xi)=M-\ii\sum_i m_i\xi_i
                   -\frac12\sum_{i,j}Q_{ij}\xi_i\xi_j+O(|\xi|^3).
\end{equation}
The remainder is bounded by
$\tfrac16|\xi|^3\sum_x|x|^3|g(x)|$, and this sum is bounded in terms
of $S,B$. Taylor expansion of \eqref{eq:symbol} gives, uniformly for
$p\in P$,
\begin{equation}\label{eq:symbol-expansion}
    \Lambda_p(\xi)+\ii v(p)\cdot\xi
      =-\xi^{\mathsf T}A(p)\xi
         +\frac{\ii}{6}\sum_kv_k(p)\xi_k^3+O(|\xi|^4).
\end{equation}
The quartic remainder is uniform because $\ee^{\pm p_k}$ are bounded
on $P$. When $M=0$, the product of the cubic symbol term and the
linear term of $\widehat g$ contributes to the second term in
\eqref{eq:zero-mass-expansion}.

Put $\Omega_t=[-\pi\sqrt t,\pi\sqrt t]^d$ and
$z=(x-tv(p))/\sqrt t$. By \eqref{eq:tilted-convolution} and
\eqref{eq:fourier-kernel}, the substitution $\xi=\eta/\sqrt t$ gives
\begin{equation}\label{eq:rescaled-fourier}
 (\ee^{t\mathcal L_p}g)(x)
 =\frac{t^{-d/2}}{(2\pi)^d}\int_{\Omega_t}
  \ee^{t\Lambda_p(\eta/\sqrt t)+\ii\sqrt t\,v(p)\cdot\eta}
  \widehat g(\eta/\sqrt t)\ee^{\ii z\cdot\eta}\dd\eta.
\end{equation}
Set $c_0=4/\pi^2$. On $\Omega_t$,
$\Re(t\Lambda_p(\eta/\sqrt t))\leq-c_0|\eta|^2$ and
$|\widehat g(\eta/\sqrt t)|\leq B$. Also
$\eta^{\mathsf T}A(p)\eta\geq|\eta|^2$.
For every $m\geq0$ and $N>0$,
\[
 \int_{|\eta|>t^{1/12}}(1+|\eta|)^m\ee^{-c_0|\eta|^2}\dd\eta
 \leq C_{m,N}t^{-N}\qquad(t\geq1).
\]
These bounds control both the integrand in \eqref{eq:rescaled-fourier}
and each polynomial times a Gaussian used below outside
$|\eta|\leq t^{1/12}$. We may therefore work on this ball and extend
the Gaussian integrals to $\R^d$, with errors smaller than any fixed
power of $t^{-1}$.

On $|\eta|\leq t^{1/12}$, write the exponent in
\eqref{eq:rescaled-fourier} as
$-\eta^{\mathsf T}A(p)\eta+\ii b_t+r_t$, where
$b_t=(6\sqrt t)^{-1}\sum_kv_k(p)\eta_k^3$ is real and
$|r_t|\leq Ct^{-1}|\eta|^4\leq Ct^{-2/3}$ by
\eqref{eq:symbol-expansion}. For all sufficiently large $t$,
$|r_t|\leq1$, and
$|\ee^{\ii b_t+r_t}-1-\ii b_t|
\leq\ee|r_t|+b_t^2/2$.
Multiplication by $\ee^{-\eta^{\mathsf T}A(p)\eta}$ yields
\begin{equation}\label{eq:integrable-remainder}
\begin{split}
 &\ee^{t\Lambda_p(\eta/\sqrt t)+\ii\sqrt t\,v(p)\cdot\eta}\\
 &\qquad=\ee^{-\eta^{\mathsf T}A(p)\eta}
       \left(1+\frac{\ii}{6\sqrt t}\sum_kv_k(p)\eta_k^3\right)
          +E_t(p,\eta),\\
 &|E_t(p,\eta)|
       \leq\frac Ct\bigl(|\eta|^4+|\eta|^6\bigr)
                        \ee^{-c_0|\eta|^2}.
\end{split}
\end{equation}

For general $M$, multiply \eqref{eq:integrable-remainder} by
$\widehat g(\eta/\sqrt t)
=M-\ii t^{-1/2}m\cdot\eta+O(t^{-1}|\eta|^2)$.
The terms not retained in \eqref{eq:mass-expansion} have
$L^1(\dd\eta)$ norm at most $C/t$: use
$|\widehat g|\leq B$ for the product with $E_t$, and
$|\widehat g(\eta/\sqrt t)-M|\leq Ct^{-1/2}|\eta|$ for the
product with the cubic symbol term. The Gaussian identity
\[
    \gamma_p(z)=\frac1{(2\pi)^d}\int_{\R^d}
            \ee^{-\eta^{\mathsf T}A(p)\eta}\ee^{\ii z\cdot\eta}\dd\eta
\]
converts the factors $-\ii\eta_i$ and $\ii\eta_k^3$ into
$-\partial_i\gamma_p$ and $-\partial_k^3\gamma_p$, respectively.
After multiplication by $t^{-d/2}$ in \eqref{eq:rescaled-fourier},
the $L^1$ error proves \eqref{eq:mass-expansion}.

If $M=0$, retain the quadratic term in \eqref{eq:data-expansion}.
On $|\eta|\leq t^{1/12}$ the product is
\[
\begin{split}
 &\ee^{t\Lambda_p(\eta/\sqrt t)+\ii\sqrt t\,v(p)\cdot\eta}
       \widehat g(\eta/\sqrt t)\\
 &=\ee^{-\eta^{\mathsf T}A(p)\eta}
   \left[-\frac{\ii}{\sqrt t}\sum_i m_i\eta_i
      +\frac1t\left(-\frac12\sum_{i,j}Q_{ij}\eta_i\eta_j
          +\frac16\sum_{i,k}m_iv_k(p)\eta_i\eta_k^3\right)\right]
      +R_t^{(0)}(p,\eta).
\end{split}
\]
Here
\[
 |R_t^{(0)}(p,\eta)|\leq Ct^{-3/2}
       (|\eta|^3+|\eta|^5+|\eta|^7)\ee^{-c_0|\eta|^2}.
\]
Indeed, the remainder in \eqref{eq:data-expansion}, multiplied by
the Gaussian, contributes $Ct^{-3/2}|\eta|^3\ee^{-c_0|\eta|^2}$.
The bound
$|\widehat g(\eta/\sqrt t)+\ii t^{-1/2}m\cdot\eta|
\leq Ct^{-1}|\eta|^2$ controls the terms multiplied by the cubic symbol,
giving the $|\eta|^5$ term. Finally,
$|\widehat g(\eta/\sqrt t)|\leq Ct^{-1/2}|\eta|$ and
\eqref{eq:integrable-remainder} control $E_t\widehat g$, giving the
$|\eta|^5+|\eta|^7$ terms.
Thus $\|R_t^{(0)}(p,\cdot)\|_{L^1(|\eta|\leq t^{1/12})}
\leq Ct^{-3/2}$. Fourier inversion turns $-\eta_i\eta_j$ into
$\partial_{ij}\gamma_p$ and $\eta_i\eta_k^3$ into
$\partial_i\partial_k^3\gamma_p$, proving
\eqref{eq:zero-mass-expansion}. Both errors are uniform in $x$ because
$|\ee^{\ii z\cdot\eta}|=1$, and uniform in $p,g$ by the bounds in
terms of $d,P,S,B$ established above.
\end{proof}

\begin{lemma}\label{lem:factorization}
Let $g,h,w$ be as in Theorem~\ref{thm:linear-main}. Then
\begin{equation}\label{eq:factorized-solution}
    w(t,j,y)=w_1(t,j)h_t(y),
\end{equation}
where
\[
\left\{
\begin{aligned}
    (w_1)_t&=\mathcal L_\parallel w_1,\\
    w_1(0)&=g,
\end{aligned}
\right.
\qquad
\left\{
\begin{aligned}
    \partial_t h_t&=\Delta_{\Z^{d-1}}h_t,\\
    h_0&=h.
\end{aligned}
\right.
\]
and
\begin{equation}\label{eq:Lparallel}
    (\mathcal L_\parallel\phi)(j)
       =\ee^{\lambda_*}(\phi(j-1)-\phi(j))
          +\ee^{-\lambda_*}(\phi(j+1)-\phi(j)).
\end{equation}
Moreover,
\begin{equation}\label{eq:transverse-asymptotic}
    h_t(y)=\frac{M_\perp}{(4\pi t)^{(d-1)/2}}
                 \ee^{-|y|^2/(4t)}+O(t^{-d/2}),
\end{equation}
as $t\to\infty$, uniformly in $y\in\Z^{d-1}$. The implicit constant
depends only on $d,h$.
\end{lemma}

\begin{proof}
On $\ell^\infty(\Z\times\Z^{d-1})$, the bounded operators
$\mathcal L_\parallel$ and $\Delta_{\Z^{d-1}}$ act on $j$ and $y$,
respectively, and commute. Expanding their exponentials in power series
therefore gives
$\ee^{t\mathcal L_*}=\ee^{t\mathcal L_\parallel}
\ee^{t\Delta_{\Z^{d-1}}}$. Acting on $(j,y)\mapsto g(j)h(y)$ yields
$w(t,j,y)=(\ee^{t\mathcal L_\parallel}g)(j)
(\ee^{t\Delta_{\Z^{d-1}}}h)(y)$, which proves
\eqref{eq:factorized-solution}. Apply
\eqref{eq:mass-expansion} in dimension $n=d-1$ with tilt $p=0$,
initial datum $h$, and $z=y/\sqrt t$. The moments are
$M=M_\perp$ and $m_i=\sum_y y_i h(y)$. Since
$v(0)=0$, the cubic-symbol term vanishes.  The first-moment term has size
$t^{-(n+1)/2}\sum_i m_i\partial_i\gamma_0(z)$; the derivatives of the
fixed Gaussian $\gamma_0$ are bounded, so this term is
$O(t^{-d/2})$.  The remainder is
$O(t^{-(n+2)/2})=O(t^{-(d+1)/2})$.  Keeping the mass term gives
\eqref{eq:transverse-asymptotic}.
\end{proof}

\begin{lemma}\label{lem:longitudinal}
Let $w_1$ be the longitudinal factor in
Lemma~\ref{lem:factorization}, and put $s=j-c_*t$.
For every $A>0$, as $t\to\infty$,
\begin{equation}\label{eq:longitudinal-asymptotic}
    w_1(t,j)=
       \frac{M_1}{\sqrt{4\pi}D_*^{3/2}}
          \frac{s+\delta_*}{t^{3/2}}
                \ee^{-s^2/(4D_*t)}
          +O_A\bigl((1+|s|)t^{-2}\bigr)
\end{equation}
uniformly for $|s|\leq A\sqrt t$. The implicit constant may depend on
$A,\lambda_*,g$.
\end{lemma}

\begin{proof}
Pairing the terms at $\ell$ and $-\ell$ in the moments of $g$ gives
$M=0$, $m=2M_1$ and $Q=0$. Apply
\eqref{eq:zero-mass-expansion} in dimension one with $p=\lambda_*$.
By \eqref{eq:critical-system} and \eqref{eq:kappa},
$v(p)=c_*$ and $A(p)=D_*$. With $z=s/\sqrt t$ and
$\phi_D(z)=(4\pi D)^{-1/2}\ee^{-z^2/(4D)}$ for $D>0$, this gives
\begin{equation}\label{eq:longitudinal-full-expansion}
    w_1(t,j)=-\frac{2M_1}{t}\phi_{D_*}'(z)
              +\frac{c_*M_1}{3t^{3/2}}\phi_{D_*}^{(4)}(z)
              +O(t^{-2}).
\end{equation}
The two required derivatives are
\[
    \phi_D'(z)=-\frac z{2D}\phi_D(z),\qquad
    \phi_D^{(4)}(z)
       =\left(\frac{z^4}{16D^4}-\frac{3z^2}{4D^3}
                            +\frac3{4D^2}\right)\phi_D(z).
\]
Substituting these derivatives into
\eqref{eq:longitudinal-full-expansion} gives
\[
 w_1(t,j)=\frac{M_1\phi_{D_*}(z)}{D_*t^{3/2}}
   \left[s+\frac{c_*}{4D_*}
              -\frac{c_*z^2}{4D_*^2}
              +\frac{c_*z^4}{48D_*^3}\right]+O(t^{-2}).
\]
For $|z|\leq A$, the two nonconstant correction terms in the brackets
are bounded by $C_Az^2$. Since $z^2=s^2/t\leq A|s|/\sqrt t$,
their contribution to $w_1(t,j)$ is at most $C_A|s|t^{-2}$.
Since $\delta_*=c_*/(4D_*)$, this proves
\eqref{eq:longitudinal-asymptotic}. In particular, the discrete shifts
$j\mapsto j\pm1$ were handled before Fourier inversion; no discrete
difference was applied to a pointwise remainder estimate.
\end{proof}

\begin{proof}[Proof of Theorem~\ref{thm:linear-main}]
Denote the explicit leading terms in
\eqref{eq:longitudinal-asymptotic} and
\eqref{eq:transverse-asymptotic} by $w_{\rm app}(t,j)$ and
$h_{\rm app}(t,y)$, respectively. Their product is the leading term in
\eqref{eq:wd-asymptotic}, since the constants multiply to $K_*$ and
$3/2+(d-1)/2=\kappa_d$. By \eqref{eq:factorized-solution},
\[
 w-w_{\rm app}h_{\rm app}
   =(w_1-w_{\rm app})h_t+w_{\rm app}(h_t-h_{\rm app}).
\]
For $|s|\leq A\sqrt t$ and $|y|\leq A\sqrt t$, one has
$|w_{\rm app}(t,j)|\leq C_A(1+|s|)t^{-3/2}$ and
$|h_t(y)|\leq Ct^{-(d-1)/2}$ by their formulas and
\eqref{eq:transverse-asymptotic}. The two remainders in
\eqref{eq:longitudinal-asymptotic} and
\eqref{eq:transverse-asymptotic} therefore give
\[
 |w-w_{\rm app}h_{\rm app}|
 \leq C_A(1+|s|)\bigl(t^{-2-(d-1)/2}+t^{-3/2-d/2}\bigr).
\]
Both powers are $t^{-\kappa_d-1/2}$. Fixing
$C_A^{\rm err}=2C_A$ proves \eqref{eq:wd-asymptotic} for all
sufficiently large $t$.

Since $0<\delta_*<1/2$, for $s\geq0$ one has
$\delta_*(1+s)\leq s+\delta_*\leq1+s$. For
$0\leq s\leq A\sqrt t$ and $|y|\leq A\sqrt t$, the Gaussian in
\eqref{eq:wd-asymptotic} is at least
$g_A:=\exp(-A^2/(4D_*)-A^2/4)>0$. Thus its positive leading term
is at least
$K_*g_A\delta_*(1+s)t^{-\kappa_d}$, whereas the error is at most
$C_A^{\rm err}(1+s)t^{-\kappa_d-1/2}$. Increase $T_A$ so that
\eqref{eq:wd-asymptotic} holds for $t\geq T_A$ and
$C_A^{\rm err}T_A^{-1/2}\leq K_*g_A\delta_*/2$. The error is then at most
half of the leading term. Choose
$C_A^{\rm bd}\geq\max\{1,2/(K_*g_A\delta_*),
K_*+C_A^{\rm err}T_A^{-1/2}\}$. Subtracting and adding the error yields
\eqref{eq:wd-two-sided}, with this fixed $C_A^{\rm bd}$.
\end{proof}

\begin{corollary}\label{cor:center-correction}
Let $t_n=n/c_*$ and $j_n=n$, where $n\in\N$. For the data in
Theorem~\ref{thm:linear-main},
\[
    t_n^{\kappa_d}w(t_n,j_n,0)\to K_*\delta_*>0.
\]
In particular, the term $\delta_*$ cannot be omitted from
\eqref{eq:wd-asymptotic} if the error is claimed to be
$o((1+|j-c_*t|)t^{-\kappa_d})$ uniformly near $j=c_*t$.
\end{corollary}

\begin{proof}
For this sequence $j_n-c_*t_n=0$. Substitute it and $y=0$ into
\eqref{eq:wd-asymptotic} and multiply by $t_n^{\kappa_d}$.
\end{proof}

\begin{corollary}\label{cor:negative-log-scale}
Let $w$ be as in Theorem~\ref{thm:linear-main}, set
$b=\kappa_d/\lambda_*$, and use $m_d$ from \eqref{eq:m-scale}.
If $j_t\in\Z$ and $\xi_t:=j_t-m_d(t)$ remains bounded, then,
as $t\to\infty$,
\begin{equation}\label{eq:negative-log-scale}
    \ee^{\lambda_*\xi_t}
    \ee^{-\lambda_*(j_t-c_*t)}w(t,j_t,0)
       =-K_*b\log t+O(1).
\end{equation}
For each $B>0$, the remainder is uniform over all choices with
$|\xi_t|\leq B$. Consequently, $w(t,j_t,0)<0$ for all sufficiently
large $t$, with the required time allowed to depend on $B$.
\end{corollary}

\begin{proof}
Fix $B>0$ and assume $|\xi_t|\leq B$. Then
$s_t:=j_t-c_*t=-b\log t+\xi_t$ satisfies $|s_t|\leq\sqrt t$
for all sufficiently large $t$, so Theorem~\ref{thm:linear-main}
applies with $A=1$ and $y=0$. Moreover,
$\ee^{-s_t^2/(4D_*t)}=1+O_B((\log t)^2/t)$.
Since $\ee^{\lambda_*\xi_t}\ee^{-\lambda_*s_t}=t^{\kappa_d}$,
\eqref{eq:wd-asymptotic} gives
\[
 t^{\kappa_d}w(t,j_t,0)
   =K_*(s_t+\delta_*)+O_B((\log t)t^{-1/2})
                       +O_B((\log t)^3/t).
\]
Substitute $s_t=-b\log t+\xi_t$ and use $|\xi_t|\leq B$ to obtain
\eqref{eq:negative-log-scale}. Its right-hand side is negative for
all sufficiently large $t$, since $K_*b>0$.
\end{proof}

\section{Upper bounds}\label{sec:upper}

\begin{proof}[Proof of Theorem~\ref{thm:nonlinear-upper}]
Let $V(t)=\ee^{at}\ee^{t\Delta_{\Z^d}}u_0$, the solution of
\eqref{eq:linear-original}. The linear comparison
\eqref{eq:linear-comparison} gives $0\leq u\leq V$.
Formula \eqref{eq:conjugation} and the kernel bound
\eqref{eq:kernel-sup} therefore give,
for every $p\in\R^d$,
\begin{equation}\label{eq:exponential-envelope}
    u(t,x)\leq V(t,x)
       \leq C_{{\rm ker},d}t^{-d/2}\ee^{-p\cdot x+tH(p)}
                        \sum_z\ee^{p\cdot z}u_0(z),
       \qquad t\geq1.
\end{equation}
Set $P_0:=\sup_{e\in\mathbb S^{d-1}}|p_e|<\infty$, which is finite by
Lemma~\ref{lem:critical-pair}.  Since $u_0$ has finite support,
\[
 \sup_e\sum_z\ee^{p_e\cdot z}u_0(z)
 \leq\sum_{z\in\supp u_0}\ee^{P_0|z|}u_0(z)=:M_0<\infty.
\]
If $|x-re|\leq\rho$, write $x=re+\delta$ with $|\delta|\leq\rho$.
Using $p_e\cdot e=\alpha_e$ and
$H(p_e)=\alpha_e w_*(e)$ from \eqref{eq:duality},
\[
 -p_e\cdot x+tH(p_e)
 =-\alpha_e(r-w_*(e)t)-p_e\cdot\delta
 \leq-\alpha_e(r-w_*(e)t)+P_0\rho.
\]
Substituting this bound into \eqref{eq:exponential-envelope} and using
$u\leq1$ yields \eqref{eq:directional-upper} with, for example,
$C_\rho=\max\{1,C_{{\rm ker},d}M_0\ee^{P_0\rho}\}$.
Let $\alpha_0:=\min_e\alpha_e>0$. For $L\geq0$, if
$r\geq w_*(e)t-(d/(2\alpha_e))\log t+L$, then
\[
 C_\rho t^{-d/2}\ee^{-\alpha_e(r-w_*(e)t)}
 \leq C_\rho\ee^{-\alpha_e L}
 \leq C_\rho\ee^{-\alpha_0L}.
\]
Choose $L_{\theta,\rho}:=\alpha_0^{-1}\log(2C_\rho/\theta)>0$.
Then $C_\rho\ee^{-\alpha_0L_{\theta,\rho}}=\theta/2<\theta$,
which proves \eqref{eq:linear-scale-upper}.

To prove \eqref{eq:linear-profile}, apply
Lemma~\ref{lem:fourier-expansion} with
$g_{p_e}(z)=\ee^{p_e\cdot z}u_0(z)$.
The supports lie in $\supp u_0$, and $\|g_{p_e}\|_{\ell^1}\leq M_0$;
thus the lemma applies uniformly in $e$. Write $r_0=x-tv_e$.
Fix $B>0$ as in \eqref{eq:linear-profile} and set
$C_{\rm win}:=B+d/(2\alpha_0)$. On
$|r_0|\leq C_{\rm win}(1+\log t)$, formula
\eqref{eq:mass-expansion} gives, as $t\to\infty$,
\begin{equation}\label{eq:kernel-log-window}
    (\ee^{t\mathcal L_{p_e}}g_{p_e})(x)
       =t^{-d/2}\left[
            M(e)\gamma_{p_e}(0)
               +O\left(\frac{1+|r_0|^2}{t}\right)\right].
\end{equation}
Indeed, $\gamma_{p_e}$ is even, so its first and third derivatives
vanish at $0$. The explicit Gaussian formula and compactness of
$\mathcal P$ give uniform bounds on its second and fourth derivatives.
Taylor's formula therefore yields
$\gamma_{p_e}(r_0/\sqrt t)-\gamma_{p_e}(0)=O(|r_0|^2/t)$ and
$|\partial_i\gamma_{p_e}(r_0/\sqrt t)|
+|\partial_k^3\gamma_{p_e}(r_0/\sqrt t)|=O(|r_0|/\sqrt t)$.
Inserting these bounds into \eqref{eq:mass-expansion}, whose remainder
is $O(t^{-d/2-1})$, proves \eqref{eq:kernel-log-window} uniformly in $e$.

If $|q(t,x,e)|\leq B$, then
\[
    r_0=-\frac{d}{2\alpha_e}e\log t+q,
    \qquad
    \ee^{-p_e\cdot r_0}=t^{d/2}\ee^{-p_e\cdot q}.
\]
Since $H(p_e)=p_e\cdot v_e$, formulas \eqref{eq:conjugation} and
\eqref{eq:kernel-log-window} give
\[
 V(t,x)=\ee^{-p_e\cdot q}
       \left[M(e)\gamma_{p_e}(0)
                   +O_B\left(\frac{(1+\log t)^2}{t}\right)\right].
\]
Choose $z_0\in\supp u_0$ with $u_0(z_0)>0$. The bounds
$|p_e|\leq P_0$ and $\cosh(p_{e,k})\leq\cosh P_0$ imply
\[
 M(e)\gamma_{p_e}(0)
 \geq u_0(z_0)\ee^{-P_0|z_0|}(4\pi\cosh P_0)^{-d/2}>0.
\]
Dividing by this uniform positive lower bound and using
$\gamma_{p_e}(0)=(4\pi)^{-d/2}(\det A_e)^{-1/2}$ proves
\eqref{eq:linear-profile}.
\end{proof}

\begin{remark}\label{rem:axis-upper}
Set $p=\lambda_*e_1$ and $x=je_1$ in
\eqref{eq:exponential-envelope}. This bound applies to every $j\in\Z$.
With $L_{\theta,0}$ chosen in the proof of
Theorem~\ref{thm:nonlinear-upper}, it gives
\[
    R_\theta(t)\leq c_*t-\frac{d}{2\lambda_*}\log t+L_{\theta,0}
\]
for $t\geq1$ whenever the level set is nonempty.
Lemma~\ref{lem:amplification} ensures nonemptiness for large $t$.
The coefficient
$d/(2\lambda_*)$ in this estimate for $R_\theta(t)$ comes from
$u\leq V$, where $V$ solves \eqref{eq:linear-original}. Section~\ref{sec:sharp-upper} improves it to
$(d/2+1)/\lambda_*$ by proving \eqref{eq:critical-mass-decay} and using the
half-space survival estimate \eqref{eq:survival-upper}.
\end{remark}

\section{Weighted estimates for the upper bound}\label{sec:sharp-upper}

Recall $\mathcal P=\{p_e:e\in\mathbb S^{d-1}\}$ and
$v_p=\nabla H_0(p)$. Lemma~\ref{lem:critical-pair} shows that
$\mathcal P$ is compact and $\inf_{p\in\mathcal P}|p|>0$. We have
$H(p)=p\cdot v_p$ on $\mathcal P$. All constants in this section are
uniform for $p\in\mathcal P$.

We first prove the mass bounds \eqref{eq:population-mass} and
\eqref{eq:critical-mass-decay}, together with
\eqref{eq:weighted-positive-moment}. We then prove the survival bound
\eqref{eq:survival-upper} and apply it in the proof of
Theorem~\ref{thm:sharp-directional-upper}.

\begin{lemma}\label{lem:population-mass}
Under \eqref{eq:kpp}--\eqref{eq:u0}, there is $C_{\rm pop}>0$,
depending only on $d,f,u_0$, such that
\begin{equation}\label{eq:population-mass}
    \sum_{x\in\Z^d}u(t,x)\leq C_{\rm pop}(1+t)^d\qquad(t\geq0).
\end{equation}
\end{lemma}

\begin{proof}
Set
\[
   h:=a+2d(\cosh1-1),\qquad
   C_{\rm init}:=\max_{\sigma\in\{-1,1\}^d}
          \max_{z\in\supp u_0}u_0(z)\ee^{\sigma\cdot z}.
\]
For each $\sigma\in\{-1,1\}^d$, one has
$u_0(z)\leq C_{\rm init}\ee^{-\sigma\cdot z}$ for every $z\in\Z^d$.
The linear comparison \eqref{eq:linear-comparison} therefore gives
\[
\begin{split}
 u(t,x)
 &\leq \ee^{at}\sum_{z\in\Z^d}\pheat_d(t,x-z)u_0(z)\\
 &\leq C_{\rm init}\ee^{at-\sigma\cdot x}
       \sum_{w\in\Z^d}\pheat_d(t,w)\ee^{\sigma\cdot w}
  =C_{\rm init}\ee^{ht-\sigma\cdot x}.
\end{split}
\]
The exponential moment is
$\sum_{\xi\in\Z^d}\pheat_d(t,\xi)\ee^{\sigma\cdot\xi}
=\ee^{tH_0(\sigma)}$, as computed in the proof of
Lemma~\ref{lem:conjugation}.
For a fixed $x$, choose $\sigma_k=\operatorname{sgn}(x_k)$ when $x_k\neq0$
(and either sign when $x_k=0$).  Then $\sigma\cdot x=|x|_1$. Combining $u(t,x)\leq C_{\rm init}\ee^{ht-\sigma\cdot x}$
with $u(t,x)\leq1$ from Lemma~\ref{lem:comparison} yields
\[
   u(t,x)\leq\min\{1,C_{\rm init}\ee^{ht-|x|_1}\}.
\]
If $|x|_1\leq2ht$, then every coordinate satisfies
$|x_k|\leq2ht$, so the number of such lattice points is at most
$(2\lfloor2ht\rfloor+1)^d\leq(4h+1)^d(1+t)^d$.
If $|x|_1>2ht$, then
$ht-|x|_1\leq-|x|_1/2$, and the geometric product
\[
 \sum_{x\in\Z^d}\ee^{-|x|_1/2}
 =\left(1+2\sum_{m=1}^\infty\ee^{-m/2}\right)^d
\]
equals $(1+2/(\ee^{1/2}-1))^d$. Summing over
$\{|x|_1\leq2ht\}$ and $\{|x|_1>2ht\}$ proves
\eqref{eq:population-mass} with
$C_{\rm pop}:=(4h+1)^d+C_{\rm init}(1+2/(\ee^{1/2}-1))^d$.
\end{proof}

For $p\in\mathcal P$, use $X_s^p:=X^p(s)$ from the proof of
Lemma~\ref{lem:conjugation}. Its independent Poisson clocks
$N_k^{p,+}$ and $N_k^{p,-}$ have rates $\ee^{p_k}$ and $\ee^{-p_k}$.
For the Poisson construction and martingales, see
\cite[Sections~2.4, 2.6 and~4.1]{Norris}; stopping times are discussed in
\cite[Section~6.5]{Norris}. Define $Y^p_s:=p\cdot(X^p_s-sv_p)$.
Lemma~\ref{lem:conjugation} gives $\mathbb P\{X_s^p=x\}=K_p(s,x)$.
Since $\mathcal A_p=\mathcal L_{-p}$ and
$K_{-p}(s,z)=K_p(s,-z)$, formula \eqref{eq:tilted-convolution} yields,
for every bounded function $h$,
\begin{equation}\label{eq:two-semigroup-conventions}
 (\ee^{s\mathcal A_p}h)(x)=\mathbb E h(x+X^p_s),
 \qquad
 (\ee^{s\mathcal L_p}h)(x)=\mathbb E h(x-X^p_s).
\end{equation}
The formula for $\ee^{s\mathcal A_p}$ uses the path $x+X_s^p$;
the formula for $\ee^{s\mathcal L_p}$ uses $x-X_s^p$, which appears
in the stopped representation \eqref{eq:stopped-expectation}. Define the
natural filtration
\[
 \mathcal F_s^p:=\sigma\bigl(N_k^{p,\pm}(r):0\leq r\leq s,
                  \ 1\leq k\leq d\bigr).
\]
For $1\leq k\leq d$,
\begin{equation}\label{eq:tilted-coordinate-moments}
 \mathbb E X^p_{s,k}
   =s(\ee^{p_k}-\ee^{-p_k})=2s\sinh p_k=s(v_p)_k,
 \qquad
 \operatorname{Var}(X^p_{s,k})
   =s(\ee^{p_k}+\ee^{-p_k})=2s\cosh p_k.
\end{equation}
For $0\leq r\leq s$, independent Poisson increments give
\[
 \mathbb E\!\left[N_k^{p,\pm}(s)-s\ee^{\pm p_k}\mid\mathcal F_r^p\right]
 =N_k^{p,\pm}(r)+(s-r)\ee^{\pm p_k}-s\ee^{\pm p_k}
 =N_k^{p,\pm}(r)-r\ee^{\pm p_k}.
\]
The finite second moments of the Poisson variables therefore show that
$N_k^{p,\pm}(s)-s\ee^{\pm p_k}$ are square-integrable
$\mathcal F_s^p$-martingales; see \cite[Sections~2.4 and~4.1]{Norris}
for the Poisson and martingale background.
Consequently
$Y^p_s=p\cdot(X^p_s-sv_p)$ is a centered square-integrable
$\mathcal F_s^p$-martingale with stationary independent increments. Since
the coordinates are independent, their variances add, and therefore
\begin{equation}\label{eq:projected-variance}
    \mathbb E\bigl[(Y^p_s)^2\bigr]
       =2s\sum_{k=1}^dp_k^2\cosh p_k
       =2s\,p^{\mathsf T}A(p)p.
\end{equation}
A jump in the $k$th coordinate changes $Y^p$ by $p_k$ or $-p_k$; hence
its jump sizes have absolute value at most
$B_0:=\sup_{p\in\mathcal P}\max_k|p_k|<\infty$.

\begin{lemma}\label{lem:projected-crossing}
For every $D\geq0$, there are $s_D\geq1$ and $C_{{\rm cross},D}>0$ such that
\begin{equation}\label{eq:crossing-probability}
    \mathbb P\{\epsilon Y^p_s\leq-D-r\}
       \geq\frac25-C_{{\rm cross},D}\frac{1+r_+}{\sqrt s}
\end{equation}
for $s\geq s_D$, $p\in\mathcal P$, $\epsilon\in\{-1,1\}$ and
$r\in\R$. Here $r_+=\max\{r,0\}$.
\end{lemma}

To prove \eqref{eq:crossing-probability}, we use two estimates
uniform in $p\in\mathcal P$: an
interval concentration bound of order $s^{-1/2}$ for $Y_s^p$ and convergence
of $Y_s^p/\sqrt s$ to a nondegenerate centered Gaussian after projection.
Local central limit estimates of this type for lattice walks are developed
in \cite[Chapter~2]{LawlerLimic}.  We prove \eqref{eq:projected-concentration} and
\eqref{eq:projected-central-limit} because the increment law of $Y^p$
depends on $p$ and need not be supported on a common one-dimensional
lattice.

\begin{proof}
We first prove that there is $C_{\rm conc}>0$ such that every interval
$J\subset\R$ of length $\ell\geq0$ satisfies
\begin{equation}\label{eq:projected-concentration}
    \mathbb P\{Y^p_s\in J\}\leq C_{\rm conc}(1+\ell)s^{-1/2}
    \qquad(s\geq1).
\end{equation}
Let $p_*:=\inf_{q\in\mathcal P}|q|>0$.  For each
$p\in\mathcal P$, choose $k$ with
$|p_k|=\max_i|p_i|\geq p_*/\sqrt d$. Conditional
on $(X^p_{s,i})_{i\ne k}$, the condition $Y_s^p\in J$ is
an interval condition on the integer $X^p_{s,k}$ of length at most
$\ell/|p_k|$; hence it contains at most
$2+\ell\sqrt d/p_*$ possible integers.  Applying
\eqref{eq:kernel-sup} in dimension one to the $k$th tilted coordinate gives
\[
 \sup_{m\in\Z}\mathbb P\{X^p_{s,k}=m\}
 \leq C_{{\rm ker},1}s^{-1/2},
\]
with $C_{{\rm ker},1}$ independent of $p$, because the one-dimensional kernel bound in
Lemma~\ref{lem:conjugation} is uniform in the tilt.  Multiplying the
number of admissible integers by this point-mass bound and then removing
the conditioning gives \eqref{eq:projected-concentration} with
$C_{\rm conc}:=C_{{\rm ker},1}\max\{2,\sqrt d/p_*\}$.
Replacing $J$ by $-J$ gives
$\mathbb P\{-Y_s^p\in J\}\leq C_{\rm conc}(1+\ell)s^{-1/2}$ as well.

We also have, for each fixed $D\geq0$,
\begin{equation}\label{eq:projected-central-limit}
    \mathbb P\{\epsilon Y^p_s\leq-D\}\to\frac12
    \qquad(s\to\infty),
\end{equation}
uniformly in $p$ and $\epsilon$. We now prove this uniformity in
$p\in\mathcal P$ and $\epsilon\in\{-1,1\}$. Lemma~\ref{lem:fourier-expansion}, applied with $g=\1_{\{0\}}$,
gives
\[
    K_p(s,x)=s^{-d/2}\gamma_p\bigl((x-sv_p)/\sqrt s\bigr)
                    +O(s^{-(d+1)/2})
\]
uniformly in $p,x$. On a fixed cube $[-R,R]^d$ in the scaled variable,
summing this error gives $O_R(s^{-1/2})$. The leading sum is the Riemann sum of $\gamma_p$ on the translated
mesh $s^{-1/2}\Z^d-\sqrt s\,v_p$, restricted to
$[-R,R]^d\cap\{\epsilon p\cdot z\leq-D/\sqrt s\}$.  We record the
uniform error estimate.  Compactness of $\mathcal P$ gives
\[
 \sup_{p\in\mathcal P}\sup_{|z|_\infty\leq R+1}
      (|\gamma_p(z)|+|\nabla\gamma_p(z)|)\leq C_R.
\]
Choose $k$ with $|p_k|=\max_i|p_i|\geq p_*/\sqrt d$.
Over a fixed grid cell in the other $d-1$ coordinates, the cutting
hyperplane varies in the $k$th coordinate by at most
$(d-1)s^{-1/2}$. The slab between
$\{\epsilon p\cdot z=-D/\sqrt s\}$ and $\{p\cdot z=0\}$ adds width
at most $D\sqrt d\,s^{-1/2}/p_*$. Thus at most
$d+2+D\sqrt d/p_*$ cells in the $k$th direction meet this slab or its
boundary. There are $O_R(s^{(d-1)/2})$ transverse grid cells, each full
cell has volume $s^{-d/2}$, and their total contribution is
$O_{R,D}(s^{-1/2})$.  On all remaining cells the mean-value theorem and the
bound on $\nabla\gamma_p$ give an error at most
$\sqrt d\,C_Rs^{-1/2}$ times the total cell volume, again $O_R(s^{-1/2})$.
Cells meeting a face of $[-R,R]^d$ also number $O_R(s^{(d-1)/2})$,
so their total contribution is $O_R(s^{-1/2})$. The Riemann sum
therefore differs by $O_{R,D}(s^{-1/2})$, uniformly in $p$ and in the
translation of the mesh, from
$\int_{[-R,R]^d\cap\{\epsilon p\cdot z\leq0\}}\gamma_p(z)\,\dd z$.

Put $C_{\rm var}:=2\sup_{p\in\mathcal P}\operatorname{tr}A(p)<\infty$.
By \eqref{eq:tilted-coordinate-moments},
$\mathbb E|X^p_s-sv_p|^2\leq C_{\rm var}s$.
For each $p$, let $Z_p$ have density $\gamma_p$; then
$\mathbb E|Z_p|^2=2\operatorname{tr}A(p)\leq C_{\rm var}$.
Chebyshev's inequality bounds each probability outside $[-R,R]^d$,
for $(X_s^p-sv_p)/\sqrt s$ and for $Z_p$, by $C_{\rm var}/R^2$.
The normal variable $p\cdot Z_p$ has variance
$2p^{\mathsf T}A(p)p\geq2p_*^2>0$, so symmetry gives
$\mathbb P\{\epsilon p\cdot Z_p\leq0\}=1/2$ for either sign.
Consequently, for fixed $R\geq1,D\geq0$ and sufficiently large $s$,
\[
 \sup_{\substack{p\in\mathcal P\\\epsilon=\pm1}}
 \left|\mathbb P\{\epsilon Y_s^p\leq-D\}-\frac12\right|
 \leq C_{R,D}s^{-1/2}+2C_{\rm var}R^{-2}.
\]
First choosing $R$ large and then $s$ large proves
\eqref{eq:projected-central-limit}.

Choose $s_D$ so large that
$\mathbb P\{\epsilon Y_s^p\leq-D\}\geq2/5$ for every
$s\geq s_D$, $p\in\mathcal P$ and $\epsilon=\pm1$.  If $r\leq0$, then
$-D-r\geq-D$, so the event
$\{\epsilon Y_s^p\leq-D\}$ is contained in
$\{\epsilon Y_s^p\leq-D-r\}$ and \eqref{eq:crossing-probability}
follows.  If $r>0$, then
\[
 \begin{split}
 \mathbb P\{\epsilon Y_s^p\leq-D-r\}
 &\geq\mathbb P\{\epsilon Y_s^p\leq-D\}
   -\mathbb P\{-D-r<\epsilon Y_s^p\leq-D\}\\
 &\geq\frac25-C_{\rm conc}(1+r)s^{-1/2},
 \end{split}
\]
where the second line uses \eqref{eq:projected-concentration} for an
interval of length $r$. Thus \eqref{eq:crossing-probability} holds with
$C_{{\rm cross},D}:=C_{\rm conc}$; only $s_D$ needs to depend on $D$.
\end{proof}

For $p\in\mathcal P$ define
\begin{equation}\label{eq:weighted-nonlinear-functions}
\begin{split}
    W_p(t,x)&:=\ee^{p\cdot(x-tv_p)}u(t,x),\\
    \mathcal R_p(t,x)&:=\ee^{p\cdot(x-tv_p)}\bigl(au(t,x)-f(u(t,x))\bigr),\\
    M_p(t)&:=\sum_xW_p(t,x).
\end{split}
\end{equation}
Since $0\leq u\leq1$ by Lemma~\ref{lem:comparison} and
$f(u)\leq au$ by \eqref{eq:kpp}, definition
\eqref{eq:weighted-nonlinear-functions} gives $\mathcal R_p(t,x)\geq0$.
The cutoff $z_p$ in Lemma~\ref{lem:critical-weighted-mass} has
coefficient $L=d+3$, rather than $\kappa_d$. It controls
the negative part of a first moment; it is not a prescribed front position.

\begin{lemma}\label{lem:critical-weighted-mass}
There are $C_{\rm mass},C_{\rm mom}>0$, depending only on $d,f,u_0$, such that
\begin{equation}\label{eq:critical-mass-decay}
    M_p(t)\leq C_{\rm mass}(1+t)^{-1/2}
    \qquad(t\geq0,\ p\in\mathcal P).
\end{equation}
With $L=d+3$ and
$z_p(t,x)=p\cdot(x-tv_p)+L\log(1+t)$, one also has
\begin{equation}\label{eq:weighted-positive-moment}
    \sum_x(z_p(t,x))_+W_p(t,x)\leq C_{\rm mom}
    \qquad(t\geq0,\ p\in\mathcal P).
\end{equation}
\end{lemma}

\begin{proof}
For $p\in\mathcal P$, differentiate the definition of $W_p$ and use
\eqref{eq:main}:
\[
 \partial_tW_p
 =\ee^{p\cdot(x-tv_p)}
   \bigl(\Delta_{\Z^d}u+f(u)-p\cdot v_p\,u\bigr).
\]
Using
$W_p(t,x\pm e_k)=\ee^{p\cdot(x-tv_p)}\ee^{\pm p_k}u(t,x\pm e_k)$
in each term of \eqref{eq:general-generator}, the factors
$\ee^{\mp p_k}$ from the shifted $W_p$ cancel the coefficients
$\ee^{\pm p_k}$, and one obtains
\[
 \mathcal L_pW_p
 =\ee^{p\cdot(x-tv_p)}
   \bigl(\Delta_{\Z^d}u-H_0(p)u\bigr).
\]
Since $H(p)=H_0(p)+a=p\cdot v_p$ on $\mathcal P$, subtraction gives
\begin{equation}\label{eq:weighted-absorption-equation}
    \partial_tW_p=\mathcal L_pW_p-\mathcal R_p,
    \qquad 0\leq \mathcal R_p\leq aW_p.
\end{equation}
The bounds $0\leq\mathcal R_p\leq aW_p$ follow from
$0\leq au-f(u)\leq au$ on $[0,1]$. With $g_p(x)=\ee^{p\cdot x}u_0(x)$, the already established
linear comparison \eqref{eq:linear-comparison}, together with
\eqref{eq:conjugation} and $H(p)=p\cdot v_p$, gives directly
\[
 0\leq W_p(t,x)\leq(\ee^{t\mathcal L_p}g_p)(x).
\]
Thus $W_p$ is bounded on $[0,S]\times\Z^d$ for every $S<\infty$.
Choose a finite set $K\subset\Z^d$ containing the support of $u_0$,
which also contains the support of every $g_p$. Put
$P_0:=\sup_{p\in\mathcal P}|p|$. Then
$\sup_{p\in\mathcal P,\,z\in K}g_p(z)\leq
\ee^{P_0\max_{z\in K}|z|}$. Formula
\eqref{eq:tilted-coordinate-moments} gives
$\mathbb E|X_t^p|^2=t^2|v_p|^2+2t\operatorname{tr}A(p)$.
Since $v_p$ and $A(p)$ are bounded on $\mathcal P$,
\[
 \sup_{0\leq t\leq S,\ p\in\mathcal P}
 \mathbb E|X_t^p|^2<\infty\qquad(S<\infty).
\]
Since
$(\ee^{t\mathcal L_p}g_p)(x)=\sum_{z\in K}K_p(t,x-z)g_p(z)$,
changing variables $w=x-z$ and using
$K_p(t,w)=\mathbb P\{X_t^p=w\}$ gives
\[
 \sup_{0\leq t\leq S,\ p\in\mathcal P}
 \sum_x(1+|x|^2)(\ee^{t\mathcal L_p}g_p)(x)<\infty.
\]
The inequality $0\leq W_p\leq\ee^{t\mathcal L_p}g_p$ therefore
controls the weighted second moment of $W_p$, and
$0\leq\mathcal R_p\leq aW_p$ controls that of $\mathcal R_p$.
Finally, from
$\partial_tW_p=\mathcal L_pW_p-\mathcal R_p$ and
\eqref{eq:general-generator}, each term in $\mathcal L_pW_p$ is a unit
shift of $W_p$ multiplied by a rate $\ee^{\pm p_k}$.  Because
$1+|x|^2\leq3(1+|x\pm e_k|^2)$ and the rates are uniformly bounded, the
same weighted second-moment bound applies to $|\partial_tW_p|$.  Thus
\begin{equation}\label{eq:weighted-moment-domination}
 \sup_{\substack{0\leq t\leq S\\p\in\mathcal P}}
 \sum_x(1+|x|^2)
   \bigl(W_p(t,x)+\mathcal R_p(t,x)
                    +|\partial_tW_p(t,x)|\bigr)\leq C_S.
\end{equation}
For $|x|>R\geq1$, $1+|x|\leq2R^{-1}(1+|x|^2)$, so
\eqref{eq:weighted-moment-domination} also gives
\[
 \sum_{|x|>R}(1+|x|)\bigl(W_p+\mathcal R_p+|\partial_tW_p|\bigr)
 \leq 2C_S/R.
\]
Define
\[
\begin{aligned}
 J_p(t)&:=\sum_x z_p(t,x)W_p(t,x),&
 N_p(t)&:=\sum_x(z_p(t,x))_-W_p(t,x),\\
 F_p(t)&:=\sum_x(z_p(t,x))_+W_p(t,x),&
 D_p(t)&:=\sum_x\mathcal R_p(t,x).
\end{aligned}
\]
These sums converge absolutely by \eqref{eq:weighted-moment-domination}.
The tail bound and continuity at each site give
$W_p,\mathcal R_p\in C([0,S];\ell^1)$.
Together with \eqref{eq:weighted-absorption-equation}, they also give
$W_p\in C^1([0,S];\ell^1)$. Thus
\eqref{eq:variation-constants} applies in $\ell^1$.
The same tail bound, with the factor $1+|x|$, permits differentiation
of $M_p$ and $J_p$ under the sums. In particular, these functions are
locally absolutely continuous, and the balance identities derived
below hold almost everywhere.

Put $M_{\rm in}:=\sup_{p\in\mathcal P}M_p(0)<\infty$ and
$F_{\rm in}:=\sup_{p\in\mathcal P}F_p(0)<\infty$.
Summing
\eqref{eq:weighted-absorption-equation} over $x$ is legitimate by
\eqref{eq:weighted-moment-domination}.  In the sum of
$\mathcal L_pW_p$, the change of variables $x\mapsto x\pm e_k$ cancels
each incoming term with its corresponding diagonal term, so
$\sum_x\mathcal L_pW_p(t,x)=0$.  Therefore
\begin{equation}\label{eq:weighted-mass-balance}
    M_p'(t)=-D_p(t)\leq0,
    \qquad M_p(t)\leq M_p(0)\leq M_{\rm in}.
\end{equation}
The bounds on $M_{\rm in}$ and $F_{\rm in}$ follow from the finite
support of $u_0$ and boundedness of $\mathcal P$.
By \eqref{eq:population-mass},
\begin{equation}\label{eq:negative-moment-cap}
\begin{split}
    \sum_{z_p(t,x)<0}W_p(t,x)&\leq C_{\rm pop}(1+t)^{-3},\\
    N_p(t)&\leq C_{\rm pop}(1+t)^{-3}.
\end{split}
\end{equation}
Indeed,
$W_p=(1+t)^{-L}\ee^{z_p}u$.  On $\{z_p<0\}$,
$\ee^{z_p}\leq1$ and $(-z_p)\ee^{z_p}\leq1$.  Hence
\[
 \sum_{z_p<0}W_p
 \leq(1+t)^{-L}\sum_xu(t,x),\qquad
 N_p(t)\leq(1+t)^{-L}\sum_xu(t,x).
\]
Using $L=d+3$ and \eqref{eq:population-mass} gives both inequalities in
\eqref{eq:negative-moment-cap}.
For the affine function $z_p(t,x)$ one has
$z_p(t,x\pm e_k)-z_p(t,x)=\pm p_k$.  Therefore
\[
 \mathcal A_pz_p
 =\sum_{k=1}^d p_k(\ee^{p_k}-\ee^{-p_k})
 =2\sum_{k=1}^d p_k\sinh p_k
 =p\cdot v_p,
\]
while differentiating the definition of $z_p$ gives
$\partial_tz_p=-p\cdot v_p+L/(1+t)$. Although $z_p(t,\cdot)$ is
unbounded, its growth is linear, so \eqref{eq:weighted-moment-domination}
justifies the index shifts in \eqref{eq:adjoint-generator} with
$\phi=z_p$ and $\psi=W_p$. Differentiating
$J_p(t)=\sum_x z_p(t,x)W_p(t,x)$, inserting
$\partial_tW_p=\mathcal L_pW_p-\mathcal R_p$ from
\eqref{eq:weighted-absorption-equation}, and using
\eqref{eq:adjoint-generator} yields
\[
\begin{split}
J_p'(t)
 &=\sum_x\bigl(\partial_tz_p(t,x)\bigr)W_p(t,x)
   +\sum_xz_p(t,x)\mathcal L_pW_p(t,x)
   -\sum_xz_p(t,x)\mathcal R_p(t,x)\\
 &=\sum_x\bigl(\partial_tz_p(t,x)+\mathcal A_pz_p(t,x)\bigr)W_p(t,x)
   -\sum_xz_p(t,x)\mathcal R_p(t,x).
\end{split}
\]
The interchange of sums and the index shifts in the second equality are
justified by \eqref{eq:weighted-moment-domination}. Substituting $\partial_tz_p+\mathcal A_pz_p=L/(1+t)$
into the expression for $J_p'(t)$ yields
\begin{equation}\label{eq:weighted-moment-balance}
    J_p'(t)=\frac L{1+t}M_p(t)-\sum_xz_p(t,x)\mathcal R_p(t,x)
       \leq\frac L{1+t}M_p(t)+aC_{\rm pop}(1+t)^{-3}.
\end{equation}
Indeed, on $\{z_p<0\}$ one has $\mathcal R_p\leq aW_p$, while the
contribution from $\{z_p\geq0\}$ to
$-\sum_xz_p\mathcal R_p$ is nonpositive. Hence
$-\sum_xz_p\mathcal R_p\leq aN_p(t)$, and
\eqref{eq:negative-moment-cap} gives the last term in
\eqref{eq:weighted-moment-balance}.
Since $M_p(t)\leq M_{\rm in}$, integration of
\eqref{eq:weighted-moment-balance} from $0$ to $t$ gives
\[
 J_p(t)\leq F_{\rm in}+LM_{\rm in}\log(1+t)+\frac a2C_{\rm pop}.
\]
Since $F_p(t)=J_p(t)+N_p(t)$, adding
\eqref{eq:negative-moment-cap} gives
\begin{equation}\label{eq:preliminary-weighted-moment}
    F_p(t)\leq C(1+\log(1+t)).
\end{equation}

We next obtain a contraction of the mass on doubled time intervals.
Let $t\geq1$ and $E_{p,2t}:=\{x\in\Z^d:z_p(2t,x)<0\}$.
A tilted walk started at $y$ at time $t$ has position $y+X_t^p$ at
time $2t$. Its shifted coordinate satisfies
\[
    z_p(2t,y+X_t^p)
       =z_p(t,y)+Y^p_t+L\log\frac{1+2t}{1+t}.
\]
The logarithmic increment satisfies
$L\log((1+2t)/(1+t))<L\log2$.  Put
$D=L\log2+1$.  If the increment satisfies
$Y_t^p\leq-D-z_p(t,y)$, then the final coordinate is
\[
 z_p(t,y)+Y_t^p+L\log\frac{1+2t}{1+t}<-1<0,
\]
so the endpoint belongs to $E_{p,2t}$.  Applying Lemma~\ref{lem:projected-crossing} with
$\epsilon=1$ and $r=z_p(t,y)$ gives
\[
 \mathbb P\{y+X_t^p\in E_{p,2t}\}
 \geq\frac25-C_{{\rm cross},D}t^{-1/2}\bigl(1+(z_p(t,y))_+\bigr)
\]
for $t\geq s_D$. Multiply by $W_p(t,y)$ and sum over $y$;
the definitions of $M_p$ and $F_p$ give
\begin{equation}\label{eq:free-crossing-mass}
    \sum_{x\in E_{p,2t}}(\ee^{t\mathcal L_p}W_p(t))(x)
      \geq\frac25M_p(t)
              -\frac{C_{{\rm cross},D}}{\sqrt t}\bigl(M_p(t)+F_p(t)\bigr).
\end{equation}
Applying \eqref{eq:variation-constants} to
\eqref{eq:weighted-absorption-equation} on $[t,2t]$ gives
\[
    W_p(2t)=\ee^{t\mathcal L_p}W_p(t)
             -\int_t^{2t}\ee^{(2t-s)\mathcal L_p}\mathcal R_p(s)\,\dd s.
\]
The time integral is nonnegative. Its sum over $\Z^d$ equals
$\int_t^{2t}D_p(s)\,\dd s=M_p(t)-M_p(2t)$ by
\eqref{eq:weighted-mass-balance}. Its sum over $E_{p,2t}$ is therefore
at most $M_p(t)-M_p(2t)$.  On the other hand,
\eqref{eq:negative-moment-cap} at time $2t$ gives
$\sum_{x\in E_{p,2t}}W_p(2t,x)\leq C_{\rm pop}t^{-3}$.  Therefore
\[
 \sum_{x\in E_{p,2t}}(\ee^{t\mathcal L_p}W_p(t))(x)
 \leq C_{\rm pop}t^{-3}+M_p(t)-M_p(2t).
\]
Combining this with \eqref{eq:free-crossing-mass} and rearranging yields
\begin{equation}\label{eq:mass-dyadic-contraction}
    M_p(2t)\leq\frac35M_p(t)
          +\frac{C_{{\rm cross},D}}{\sqrt t}\bigl(M_p(t)+F_p(t)\bigr)+C_{\rm pop}t^{-3}.
\end{equation}

We use the dyadic contraction \eqref{eq:mass-dyadic-contraction} twice. First insert
\eqref{eq:preliminary-weighted-moment} and $M_p\leq M_{\rm in}$.
Set $t_0=s_D\geq1$, with $D=L\log2+1$, so that
\eqref{eq:mass-dyadic-contraction} holds for every $t\geq t_0$.
Put $t_n=2^nt_0$ and $b_n=b_n(p):=\sqrt{t_n}M_p(t_n)$.
Since $M_p(t_n)\leq M_{\rm in}$ and
$F_p(t_n)\leq C(1+\log t_n)\leq C(1+n)$, multiplying
\eqref{eq:mass-dyadic-contraction} by $\sqrt{2t_n}$ gives, for a fixed
$C_{{\rm iter},1}>0$ independent of $p,n$,
\[
    b_{n+1}\leq\frac{3\sqrt2}{5}b_n+C_{{\rm iter},1}(1+n).
\]
Put $\rho=3\sqrt2/5<1$. Iteration yields
$b_n\leq\rho^n b_0+C_{{\rm iter},1}\sum_{k=0}^{n-1}\rho^{n-1-k}(1+k)
\leq C(1+n)$, uniformly in $p$.
For $t_n\leq t<t_{n+1}$, \eqref{eq:weighted-mass-balance} gives
$M_p(t)\leq M_p(t_n)$ and $t<2t_n$, hence
$\sqrt t M_p(t)\leq\sqrt2 b_n$.
Choose a fixed $C_{\rm pre}>0$ large enough to cover this estimate
and $M_p(t)\leq M_{\rm in}$ for $0\leq t\leq t_0$. Then
\begin{equation}\label{eq:preliminary-critical-mass}
    M_p(t)\leq C_{\rm pre}\frac{1+\log(1+t)}{\sqrt{1+t}}
       \qquad(t\geq0,\ p\in\mathcal P).
\end{equation}
In particular,
\[
 \int_0^\infty\frac{M_p(s)}{1+s}\,\dd s
 \leq C_{\rm pre}\int_0^\infty
   \frac{1+\log(1+s)}{(1+s)^{3/2}}\,\dd s=6C_{\rm pre}.
\]
Integrating \eqref{eq:weighted-moment-balance}, using
$J_p(0)\leq F_{\rm in}$, and adding
$N_p(t)\leq C_{\rm pop}$ from \eqref{eq:negative-moment-cap}, gives
\[
 F_p(t)\leq F_{\rm in}+6LC_{\rm pre}+(1+a/2)C_{\rm pop}.
\]
Thus \eqref{eq:weighted-positive-moment} holds with the fixed choice
$C_{\rm mom}:=1+F_{\rm in}+6LC_{\rm pre}+(1+a/2)C_{\rm pop}$.
Substituting $F_p(t)<C_{\rm mom}$ and $M_p(t)\leq M_{\rm in}$ into
\eqref{eq:mass-dyadic-contraction}, then multiplying by
$\sqrt{2t_n}$, gives
\[
    b_{n+1}\leq\frac{3\sqrt2}{5}b_n+C_{{\rm iter},2},
\]
where one may take
$C_{{\rm iter},2}:=\sqrt2\,[C_{{\rm cross},D}(M_{\rm in}+C_{\rm mom})
+C_{\rm pop}t_0^{-5/2}]$.
Now $b_n\leq\rho^n b_0+C_{{\rm iter},2}/(1-\rho)$.
For $t_n\leq t<t_{n+1}$,
$M_p(t)\leq M_p(t_n)$ by \eqref{eq:weighted-mass-balance} and
$t<2t_n$, so $\sqrt t\,M_p(t)\leq\sqrt2\,b_n$.  The uniform bound
on $b_n$ gives $M_p(t)\leq C t^{-1/2}$ for $t\geq t_0$.
Since $t^{-1/2}\leq\sqrt2(1+t)^{-1/2}$ for $t\geq1$, choose
$C_{\rm mass}$ large enough to cover this bound and
$M_p(t)\leq M_{\rm in}$ on $[0,t_0]$. This proves
\eqref{eq:critical-mass-decay}. Thus the derivation of
\eqref{eq:critical-mass-decay} uses \eqref{eq:population-mass} and \eqref{eq:weighted-mass-balance} and \eqref{eq:weighted-moment-balance},
but no prior estimate for the nonlinear front position.
\end{proof}

G\"artner's proof in $\mathbb{R}^d$ compares Brownian densities killed
at a curved moving boundary with densities killed at its linear
interpolation; see \cite[Section~2]{Gartner} and the moving-boundary
first-exit analysis of Uchiyama \cite{UchiyamaFirstExit}.
For fixed terminal time $t$, the boundary in
\eqref{eq:terminal-halfspace} is affine in $(\sigma,y)$: the term
$\kappa_d\log t$ is held constant as $\sigma$ varies.
Its coordinate along the backward path is $z+1-Y_s^p$, as computed
in the proof of Theorem~\ref{thm:sharp-directional-upper}. The $s^{-1/2}$ survival scale is the
classical one-dimensional fluctuation scale; see, for example,
\cite{Spitzer}. We include a direct proof because the constant must be
uniform in $p\in\mathcal P$ and in both signs of the projected walk.

\begin{lemma}\label{lem:halfspace-survival}
For $r>0$ and $\epsilon\in\{-1,1\}$, let
\[
    \tau_{p,r}^{\epsilon}
       :=\inf\{s\geq0:r+\epsilon Y^p_s\leq0\}.
\]
Here $\inf\varnothing:=+\infty$. There is $C_{\rm surv}>0$, depending only on $d,a$, such that
\begin{equation}\label{eq:survival-upper}
    \mathbb P\{\tau_{p,r}^{\epsilon}>s\}
       \leq C_{\rm surv}\min\left\{1,\frac{1+r}{\sqrt s}\right\}
       \qquad(s>0,\ p\in\mathcal P).
\end{equation}
\end{lemma}

\begin{proof}
Fix $p\in\mathcal P$, $r>0$ and $\epsilon\in\{-1,1\}$. Write
$R_s:=r+\epsilon Y_s^p$, $\tau=\tau_{p,r}^{\epsilon}$ and
$h_r(s):=\mathbb P\{\tau>s\}$.
On each bounded interval, the paths have finitely many jumps and are
linear between jumps, with slope $-\epsilon p\cdot v_p$.
Whether a path has reached $(-\infty,0]$ by time $s$ is determined by
the clocks up to $s$, so $\tau$ is an $\mathcal F_s^p$-stopping time.
If $\tau<\infty$, then $-B_0\leq R_\tau\leq0$: exit by the drift
has value zero, and exit by a jump has overshoot at most $B_0$.

We justify the stopped expectation before using it. Fix $s>0$ and
put $T=\tau\wedge s$. The times
$T_n:=2^{-n}\lceil2^nT\rceil$ are stopping times on finite grids and
satisfy $T_n\downarrow T$ and $T_n\leq s+1$.
The martingale identity on each finite grid gives
$\mathbb E R_{T_n}=r$; see \cite[Section~4.1]{Norris} for the
discrete-time optional-sampling argument. Define
$\mathcal N_v^p:=\sum_k(N_k^{p,+}(v)+N_k^{p,-}(v))$ and
$C_v:=\sup_{p\in\mathcal P}|p\cdot v_p|<\infty$. Then
\[
 |R_{T_n}|\leq r+B_0\mathcal N_{s+1}^p+C_v(s+1).
\]
The right-hand side is integrable, since $\mathcal N_{s+1}^p$ is a
Poisson variable with finite mean. Right-continuity of $R$ and
dominated convergence give $\mathbb E R_{s\wedge\tau}=r$.
Splitting this expectation over $\{\tau>s\}$ and $\{\tau\leq s\}$ yields
\begin{equation}\label{eq:surviving-first-moment}
    \mathbb E\bigl[R_s\1_{\{\tau>s\}}\bigr]
       =r-\mathbb E\bigl[R_\tau\1_{\{\tau\leq s\}}\bigr]
       \leq r+B_0.
\end{equation}

Conditional on $\mathcal F_s^p$, the increment
$\epsilon(Y_{2s}^p-Y_s^p)$ is independent of the past and has the
law of $\epsilon Y_s^p$. On $\{\tau>s\}$ the height $R_s$ is positive.
Survival to $2s$ requires a positive endpoint, so
\eqref{eq:crossing-probability}, with $D=0$ and starting height $R_s$,
gives for $s\geq s_0$, where $s_0:=s_D|_{D=0}\geq1$,
\[
 \mathbb P\{\tau>2s\mid\mathcal F_s^p\}
 \leq\1_{\{\tau>s\}}
       \left[\frac35+\frac{C_{{\rm cross},0}}{\sqrt s}(1+R_s)\right].
\]
Taking expectations and using \eqref{eq:surviving-first-moment} and
$h_r(s)\leq1$ yields
\[
 h_r(2s)\leq\frac35h_r(s)
       +\frac{C_{{\rm cross},0}}{\sqrt s}(1+r+B_0)
 \leq\frac35h_r(s)+C_{\rm rec}\frac{1+r}{\sqrt s},
\]
where $C_{\rm rec}:=C_{{\rm cross},0}(1+B_0)$.
Put $s_n=2^ns_0$, $a_n=\sqrt{s_n}\,h_r(s_n)/(1+r)$ and
$\rho=3\sqrt2/5<1$. Then
$a_{n+1}\leq\rho a_n+\sqrt2C_{\rm rec}$ and $a_0\leq\sqrt{s_0}$,
so
\[
 a_n\leq\rho^na_0+\frac{\sqrt2C_{\rm rec}}{1-\rho}
       \leq A_{\rm surv}:=\sqrt{s_0}
                         +\frac{\sqrt2C_{\rm rec}}{1-\rho}.
\]
For $s_n\leq s<s_{n+1}$, monotonicity of $h_r$ gives
\[
 h_r(s)\leq h_r(s_n)\leq A_{\rm surv}\frac{1+r}{\sqrt{s_n}}
          \leq\sqrt2A_{\rm surv}\frac{1+r}{\sqrt s}.
\]
For $0<s<s_0$, one has
$h_r(s)\leq1\leq\sqrt{s_0}(1+r)/\sqrt s$.
Since $A_{\rm surv}\geq\sqrt{s_0}\geq1$, choosing
$C_{\rm surv}:=\sqrt2A_{\rm surv}$ and using $h_r(s)\leq1$ proves
\eqref{eq:survival-upper}. These constants depend only on $d,a$.
\end{proof}

\begin{proof}[Proof of Theorem~\ref{thm:sharp-directional-upper}]
Estimate \eqref{eq:critical-mass-main} is \eqref{eq:critical-mass-decay} from Lemma~\ref{lem:critical-weighted-mass}.
Fix $p\in\mathcal P$, $t\geq4$, and $x\in\Z^d$ such that
$z:=p\cdot(x-tv_p)+\kappa_d\log t\geq0$.
Keep the final time $t$ fixed throughout the stopping argument. Consider
$\widehat X_s=x-X^p_s$, whose generator on functions is $\mathcal L_p$,
and the time-dependent half-space
\begin{equation}\label{eq:terminal-halfspace}
    \mathcal D_t(\sigma)
      :=\{y\in\Z^d:p\cdot(y-\sigma v_p)+\kappa_d\log t+1>0\},
      \qquad t/2\leq\sigma\leq t.
\end{equation}
The affine normal coordinate of the backward path is
\[
    p\cdot(\widehat X_s-(t-s)v_p)+\kappa_d\log t+1
       =z+1-Y^p_s.
\]
Set $\tau:=\tau_{p,z+1}^{-1}
=\inf\{s\geq0:z+1-Y_s^p\leq0\}$ and
$\vartheta:=\tau\wedge(t/2)$. Thus $\tau$ is the first exit time of $z+1-Y_s^p$ from $(0,\infty)$.

The functions $W_p$, $\partial_tW_p$, and $\mathcal R_p$ are bounded on
$[t/2,t]\times\Z^d$.  Indeed,
$W_p(\sigma)\leq\ee^{\sigma\mathcal L_p}g_p$ and
\eqref{eq:kernel-sup} gives
$\|W_p(\sigma)\|_\infty\leq C_{{\rm ker},d}(t/2)^{-d/2}\|g_p\|_1$ for
$\sigma\in[t/2,t]$.  Since $0\leq\mathcal R_p\leq aW_p$ and
$\mathcal L_p$ is a finite linear combination of unit shifts with rates
$\ee^{\pm p_k}$, \eqref{eq:weighted-absorption-equation} also gives a
uniform bound on $\partial_tW_p$ over this time interval.
For $0\leq s\leq t/2$, put $F(s,y):=W_p(t-s,y)$.
We obtain the stopped identity from the compensated Poisson clocks;
for the generator-martingale formulation, see
\cite[Chapter~4]{EthierKurtz}. Write
$\widetilde N_k^{p,\pm}(s):=N_k^{p,\pm}(s)-s\ee^{\pm p_k}$ and
\[
 H_k^\pm(s):=F(s,\widehat X_{s-}\mp e_k)-F(s,\widehat X_{s-}),
 \qquad \widehat X_{s-}:=\lim_{r\uparrow s}\widehat X_r\quad(s>0),
\]
with $\widehat X_{0-}:=x$. The minus sign in $H_k^+$ corresponds to a positive jump of $X^p$,
which moves $\widehat X$ by $-e_k$.
Summing the changes of $F$ between and at jumps gives
\begin{equation}\label{eq:backward-jump-martingale}
\begin{split}
 \mathcal M_s
 &:=F(s,\widehat X_s)-F(0,x)
       -\int_0^s(\partial_r+\mathcal L_p)F(r,\widehat X_r)\,\dd r\\
 &=\sum_{k=1}^d\left[
       \int_0^s H_k^+(r)\,\dd\widetilde N_k^{p,+}(r)
       +\int_0^s H_k^-(r)\,\dd\widetilde N_k^{p,-}(r)\right].
\end{split}
\end{equation}
The integrands use the position just before a jump and are bounded by
$2\|F\|_\infty$. The independent Poisson increments therefore make
$\mathcal M_s$ a square-integrable martingale. The required second-moment
bound is
\[
 \sum_k\mathbb E\int_0^{t/2}
       \bigl(\ee^{p_k}|H_k^+(r)|^2+\ee^{-p_k}|H_k^-(r)|^2\bigr)\,\dd r
 \leq 2t\|F\|_\infty^2\sum_k(\ee^{p_k}+\ee^{-p_k})<\infty.
\]
Here changing $\widehat X_r$ to $\widehat X_{r-}$ in the time integral
has no effect, since the jump times have Lebesgue measure zero.
The space-time generator is
\[
 (\partial_s+\mathcal L_p)F(s,y)
 =(-\partial_t+\mathcal L_p)W_p(t-s,y)
 =\mathcal R_p(t-s,y)
\]
by \eqref{eq:weighted-absorption-equation}.
Sampling \eqref{eq:backward-jump-martingale} at the bounded stopping
time $\vartheta\leq t/2$ gives
\begin{equation}\label{eq:stopped-expectation}
    \mathbb E W_p(t-\vartheta,\widehat X_\vartheta)
      =W_p(t,x)+
          \mathbb E\int_0^\vartheta \mathcal R_p(t-s,\widehat X_s)\,\dd s
      \geq W_p(t,x).
\end{equation}
The stopped expectation is justified by rounding
$\vartheta$ up to finite grids in $[0,t/2]$ and using
right-continuity and boundedness of $F$, $\partial_sF$ and
$\mathcal L_pF$ in \eqref{eq:backward-jump-martingale}.
On $\{\tau\leq t/2\}$ the exit definition gives
\[
 p\cdot(\widehat X_\tau-(t-\tau)v_p)+\kappa_d\log t+1\leq0.
\]
Since $0\leq u\leq1$,
\[
 W_p(t-\tau,\widehat X_\tau)
 =\ee^{p\cdot(\widehat X_\tau-(t-\tau)v_p)}
   u(t-\tau,\widehat X_\tau)
 \leq\ee^{-1}t^{-\kappa_d}.
\]
On $\{\tau>t/2\}$ one has $\vartheta=t/2$.  Splitting the expectation
in \eqref{eq:stopped-expectation} over these two events gives
\begin{equation}\label{eq:stopped-upper-bound}
    W_p(t,x)\leq
       \mathbb E\bigl[W_p(t/2,\widehat X_{t/2})
                                      \1_{\{\tau>t/2\}}\bigr]
                +\ee^{-1}t^{-\kappa_d}.
\end{equation}
In \eqref{eq:stopped-upper-bound}, the exit contribution is evaluated at
$\widehat X_\tau$ itself; the possible jump overshoot across the half-space
boundary is retained.

The event $\{\tau>t/2\}$ is contained in $\{\tau>t/4\}$.
Condition on $\mathcal F_{t/4}^p$ and remove the survival requirement on
$[t/4,t/2]$.  By the independent increments of $X^p$ and
\eqref{eq:two-semigroup-conventions},
\[
 \mathbb E\left[W_p(t/2,\widehat X_{t/2})
        \mid\mathcal F_{t/4}^p\right]
 =(\ee^{(t/4)\mathcal L_p}W_p(t/2))(\widehat X_{t/4}).
\]
Therefore
\[
\begin{split}
    &\mathbb E\bigl[W_p(t/2,\widehat X_{t/2})
                                      \1_{\{\tau>t/2\}}\bigr]\\
    &\quad\leq
       \mathbb P\{\tau>t/4\}
                \|\ee^{(t/4)\mathcal L_p}W_p(t/2)\|_\infty.
\end{split}
\]
The affine coordinate of the backward path is $z+1-Y_s^p$, so
\eqref{eq:survival-upper} with $r=z+1$ and $\epsilon=-1$ gives
$\mathbb P\{\tau>t/4\}\leq2C_{\rm surv}(z+2)t^{-1/2}$. Also,
\eqref{eq:kernel-sup} and the definition of $M_p$ give
\[
 \|\ee^{(t/4)\mathcal L_p}W_p(t/2)\|_\infty
 \leq4^{d/2}C_{{\rm ker},d}t^{-d/2}M_p(t/2)
 \leq C_{\rm heat}t^{-(d+1)/2},
\]
where $C_{\rm heat}:=\sqrt2\,4^{d/2}C_{{\rm ker},d}C_{\rm mass}$ by
\eqref{eq:critical-mass-decay}. Put $C_{\rm stop}:=2C_{\rm surv}C_{\rm heat}$.
Then
\[
 \mathbb E\bigl[W_p(t/2,\widehat X_{t/2})
        \1_{\{\tau>t/2\}}\bigr]
 \leq C_{\rm stop}(z+2)t^{-\kappa_d}.
\]
Because $p\cdot(x-tv_p)=z-\kappa_d\log t$,
$u(t,x)=t^{\kappa_d}\ee^{-z}W_p(t,x)$.  Multiplying
\eqref{eq:stopped-upper-bound} by $t^{\kappa_d}\ee^{-z}$ gives
$u(t,x)\leq C_{\rm stop}(z+2)\ee^{-z}+\ee^{-1-z}$. For $z\geq0$ this is
at most $C_{\rm up}(1+z)\ee^{-z}$, with the fixed choice
$C_{\rm up}:=\max\{1,2C_{\rm stop}+\ee^{-1}\}$. This proves
\eqref{eq:sharp-directional-envelope}.

To prove \eqref{eq:sharp-directional-location}, fix
$e\in\mathbb S^{d-1}$, take $p=p_e$,
and write $x=re+\delta$ with $r\geq0$ and $|\delta|\leq\rho$.
Then $z=\alpha_e(r-w_*(e)t)+\kappa_d\log t+p_e\cdot\delta$.
Let
$\alpha_0:=\min_e\alpha_e>0$ and $P_0:=\max_e|p_e|<\infty$, whose
existence is proved in Lemma~\ref{lem:critical-pair}.  If
$r\geq w_*(e)t-(\kappa_d/\alpha_e)\log t+\ell$ with $\ell\geq0$,
then $z\geq\alpha_e\ell+p_e\cdot\delta\geq\alpha_0\ell-P_0\rho$.
For $Z\geq0$, the inequality $1+Z\leq2\ee^{Z/2}$ gives
$C_{\rm up}(1+Z)\ee^{-Z}\leq2C_{\rm up}\ee^{-Z/2}$.
Choose $Z_\theta:=2\log(4C_{\rm up}/\theta)>0$, so that this bound is
at most $\theta/2$ for $Z\geq Z_\theta$, and set
\[
 L_{\theta,\rho}:=\frac{Z_\theta+P_0\rho}{\alpha_0}.
\]
Then $z\geq\alpha_0L_{\theta,\rho}-P_0\rho=Z_\theta$, so
\eqref{eq:sharp-directional-envelope} implies
\eqref{eq:sharp-directional-location}.
\end{proof}
\section{A lower solution}\label{sec:lower}

The lower bound starts with a product of a solution of
\eqref{eq:time-dependent-logistic} and a normalized transverse heat
kernel. Lemma~\ref{lem:product-subsolution} defines the factors and
proves \eqref{eq:product-comparison}.

\begin{lemma}\label{lem:product-subsolution}
Put $n=d-1$ and choose
\[
    C_f:=\max\left\{a,\frac12\norm{f''}_{L^\infty([0,1])}\right\}.
\]
There is $\beta_\perp>0$, depending only on $n$, such that
\[
    0<H_T(t,y):=\beta_\perp(t+T)^{n/2}\pheat_n(t+T,y)\leq1
    \qquad(t\geq0,\ y\in\Z^n)
\]
for every $T\geq1$. Fix such a $T$, and let $q$ solve
\begin{equation}\label{eq:time-dependent-logistic}
\left\{
\begin{aligned}
    q_t(t,j)&=q(t,j-1)-2q(t,j)+q(t,j+1)
       +\left(a-\frac{n}{2(t+T)}\right)q(t,j)-C_fq(t,j)^2,\\
    0&\leq q(0,j)\leq a/C_f.
\end{aligned}
\right.
\end{equation}
Then, for every $y_0\in\Z^n$,
\begin{equation}\label{eq:product-subsolution}
    \underline u(t,j,y):=q(t,j)H_T(t,y-y_0)
\end{equation}
is a global subsolution of \eqref{eq:main}, with values in $[0,1]$.

Moreover, for the solution $u$ with initial data \eqref{eq:u0}, one can
choose $T>\max\{1,n/(2a)\}$, a site $(j_0,y_0)$ and
$\eta_q\in(0,a/C_f]$ such that the solution with
$q(0,j)=\eta_q\1_{\{j_0\}}(j)$ satisfies
\begin{equation}\label{eq:product-comparison}
    u(T+t,j,y)\geq q(t,j)H_T(t,y-y_0)
    \qquad(t\geq0,\ j\in\Z,\ y\in\Z^n).
\end{equation}
For this choice, there are $h_*>0$ and $t_*>0$ such that
\begin{equation}\label{eq:axis-product-lower}
    u(T+t,j,0)\geq h_*q(t,j)
    \qquad(t\geq t_*,\ j\in\Z).
\end{equation}
\end{lemma}

\begin{proof}
The kernel bound \eqref{eq:kernel-sup}, with $p=0$ in dimension
$n$, gives $\pheat_n(s,y)\leq C_{{\rm ker},n}s^{-n/2}$ for $s\geq1$.
Choose $\beta_\perp=(1+C_{{\rm ker},n})^{-1}$. Positivity in
\eqref{eq:heat-semigroup} then gives
$0<H_T(t,y)\leq\beta_\perp C_{{\rm ker},n}<1$.
Differentiating $(t+T)^{n/2}$ and using
$\partial_s\pheat_n=\Delta_{\Z^n}\pheat_n$ yields
\begin{equation}\label{eq:transverse-normalized-heat}
    \partial_tH_T=\Delta_{\Z^n}H_T+\frac{n}{2(t+T)}H_T.
\end{equation}
The function $q\equiv0$ solves \eqref{eq:time-dependent-logistic}.  For the constant
$q\equiv a/C_f$,
\[
  \partial_tq-\Delta_{\Z}q
   -\left(a-\frac{n}{2(t+T)}\right)q+C_fq^2
  =\frac{na}{2C_f(t+T)}\geq0,
\]
so $a/C_f$ is a supersolution of \eqref{eq:time-dependent-logistic}.
To construct $q$ without assuming its boundedness, define the 1-Lipschitz truncation
$\Pi(s)=\min\{a/C_f,\max\{0,s\}\}$ and solve on $\ell^\infty(\Z)$
\begin{equation}\label{eq:truncated-logistic}
 q_t=\Delta_{\Z}q+
   \left(a-\frac{n}{2(t+T)}\right)\Pi(q)-C_f\Pi(q)^2.
\end{equation}
Prescribe any fixed initial sequence with $0\leq q(0,j)\leq a/C_f$.
The right-hand side of \eqref{eq:truncated-logistic} is continuous in
$t$ and globally Lipschitz on $\ell^\infty(\Z)$, with constant at most
$4+a+n/(2T)$.
The fixed-point construction in the proof of Lemma~\ref{lem:comparison},
applied to \eqref{eq:truncated-logistic}, gives a unique global solution.
The functions $0$ and $a/C_f$ are respectively a subsolution and a
supersolution of \eqref{eq:truncated-logistic}.
Lemma~\ref{lem:comparison} with $D=\Z$ and
$F(t,s)=(a-n/(2(t+T)))\Pi(s)-C_f\Pi(s)^2$ gives
$0\leq q(t,j)\leq a/C_f\leq1$. Thus $\Pi(q)=q$, and $q$ is the
unique solution of \eqref{eq:time-dependent-logistic} satisfying
$0\leq q(t,j)\leq a/C_f$ for all $t\geq0$ and $j\in\Z$.

Taylor's formula and \eqref{eq:kpp} give
\begin{equation}\label{eq:R-quadratic}
    0\leq as-f(s)\leq\frac12\norm{f''}_{L^\infty([0,1])}s^2
       \leq C_fs^2\qquad(0\leq s\leq1).
\end{equation}
Write $H=H_T(t,y-y_0)$. The product $\underline u=qH$ is bounded
and continuously differentiable in time at each site. Since $0<H\leq1$,
one has $Hq^2\geq H^2q^2$. Equation~\eqref{eq:time-dependent-logistic}
and \eqref{eq:transverse-normalized-heat} therefore give
\begin{equation}\label{eq:product-residual}
\begin{split}
    \partial_t\underline u-\Delta_{\Z^d}\underline u
      &=H(q_t-\Delta_{\Z}q)+q(H_t-\Delta_{\Z^n}H)\\
      &=H(aq-C_fq^2)\\
      &\leq a\underline u-C_f\underline u^2
       \leq f(\underline u),
\end{split}
\end{equation}
where the last inequality uses \eqref{eq:R-quadratic} and
$0\leq\underline u\leq1$.

Choose a site $(j_0,y_0)$ with $u_0(j_0,y_0)=\eta_0>0$, and fix
$T>\max\{1,n/(2a)\}$. The lower inequality in
\eqref{eq:linear-comparison} gives
$u(T,j,y)\geq\eta_0\pheat_d(T,(j-j_0,y-y_0))$.  The Poisson representation
in Lemma~\ref{lem:conjugation} factors the $d$ coordinate processes, so
\[
   \pheat_d(T,(j-j_0,y-y_0))
      =\pheat_1(T,j-j_0)\pheat_n(T,y-y_0),
\]
and hence
\[
    u(T,j,y)\geq
       \eta_0\pheat_1(T,j-j_0)\pheat_n(T,y-y_0).
\]
Set
\[
    \eta_q:=\min\left\{\frac{a}{C_f},
       \frac{\eta_0\pheat_1(T,0)}{\beta_\perp T^{n/2}}\right\}>0,
    \qquad q(0,j)=\eta_q\1_{\{j_0\}}(j).
\]
At $j=j_0$, the definition of $\eta_q$ gives
$q(0,j_0)H_T(0,y-y_0)\leq
\eta_0\pheat_1(T,0)\pheat_n(T,y-y_0)$; at every other $j$, the
product is zero.
Thus $\underline u(0,j,y)\leq u(T,j,y)$ on $\Z^d$.
The function $\underline u$ satisfies \eqref{eq:product-residual},
and $u(T+\cdot)$ solves \eqref{eq:main}. Both take values in $[0,1]$.
Lemma~\ref{lem:comparison}, with $D=\Z^d$ and $F(t,s)=\widetilde f(s)$,
gives \eqref{eq:product-comparison}.
Finally, Lemma~\ref{lem:fourier-expansion} in dimension $n$, with
$p=0$ and initial datum $\1_{\{0\}}$, gives for the fixed vector $-y_0$
\[
 (t+T)^{n/2}\pheat_n(t+T,-y_0)
 \to (4\pi)^{-n/2}.
\]
Hence $H_T(t,-y_0)\to\beta_\perp(4\pi)^{-n/2}$.  Choose
$h_*:=\tfrac12\beta_\perp(4\pi)^{-n/2}$ and then choose $t_*$ so that
$H_T(t,-y_0)\geq h_*$ for $t\geq t_*$.  Evaluating
\eqref{eq:product-comparison} at $y=0$ gives
\eqref{eq:axis-product-lower}.
\end{proof}

It remains to bound $q$ from below near $c_*t-(\kappa_d/\lambda_*)\log(t+T)$. Lemmas~\ref{lem:autonomous-leading-lower}--\ref{lem:delayed-pulse} and Proposition~\ref{prop:scalar-logarithmic-lower} treat the growth rate $a-\beta/(t+T)$; Section~\ref{sec:amplification} then sets $\beta=(d-1)/2$.

The proof of Lemma~\ref{lem:logconcavity} uses preservation of
log-concavity under convolution for nonnegative summable sequences
whose positive supports are intervals. We prove the needed property
using $2\times2$ determinants and the finite Cauchy--Binet formula;
see Karlin \cite{Karlin} for the broader theory of total positivity.

\begin{lemma}\label{lem:logconcavity}
In dimension one, the tilted kernel $G(t,j):=K_{\lambda_*}(t,j)$ is
strictly positive and log-concave as a sequence in $j$ for every $t>0$:
\[
    G(t,j)^2\geq G(t,j-1)G(t,j+1)\qquad(j\in\Z).
\]
Consequently, $G(t,j+1)/G(t,j-1)$ is nonincreasing in $j$.
\end{lemma}

\begin{proof}
We first verify the convolution property needed here. Let
$v=(v_k)_{k\in\Z}$ be a nonnegative summable sequence whose support
is an interval of integers, possibly infinite. If
$v_k^2\geq v_{k-1}v_{k+1}$ for every $k$, then the Toeplitz matrix
$T(v)_{ij}=v_{i-j}$ has nonnegative $2\times2$ minors. Indeed,
for $i<j$ and $r<s$, the two indices in
$v_{i-r}v_{j-s}$ lie between the two indices in
$v_{i-s}v_{j-r}$, and their sums are equal. The monotonicity of
the ratios $v_{k+1}/v_k$ on the positive support gives
\[
    v_{i-r}v_{j-s}\geq v_{i-s}v_{j-r}.
\]
If the right-hand product is positive, interval support justifies all
successive-ratio comparisons between its two indices. If the right-hand
product is zero, the left-hand product is nonnegative because $v$ is
nonnegative, so
$v_{i-r}v_{j-s}\geq v_{i-s}v_{j-r}$ is automatic.

Let $w$ also be nonnegative, summable and log-concave, with interval
support, and define $(v*w)_m:=\sum_{\ell\in\Z}v_{m-\ell}w_\ell$.
Fix $i<j$ and $r<s$, and let $I_N=\{-N,\ldots,N\}$.
The finite Cauchy--Binet formula expresses the determinant of
$T(v)_{\{i,j\},I_N}T(w)_{I_N,\{r,s\}}$ as a sum of products of
$2\times2$ minors of $T(v)$ and $T(w)$, all nonnegative.
The product matrix converges entrywise to
$T(v*w)_{\{i,j\},\{r,s\}}$ as
$N\to\infty$, since the convolution sums converge absolutely.
Taking rows $\{k,k+1\}$ and columns $\{0,1\}$ gives
$(v*w)_k^2\geq(v*w)_{k-1}(v*w)_{k+1}$. The positive support of
$v*w$ is the sum of the two positive supports, hence is an interval.

For a Poisson random variable $N$ with mean $\rho>0$,
$\mathbb P\{N=k+1\}/\mathbb P\{N=k\}=\rho/(k+1)$ for $k\geq0$.
Its support is $\Z_{\geq0}$ and this ratio decreases with $k$, so its
distribution is log-concave. Reflection preserves
log-concavity. The kernel $G(t,\cdot)$ is the convolution of the
Poisson distribution with parameter $\ee^{\lambda_*}t$ and the
reflected Poisson distribution with parameter $\ee^{-\lambda_*}t$.
Preservation of log-concavity under convolution gives
$G(t,j)^2\geq G(t,j-1)G(t,j+1)$. By \eqref{eq:kernel-tilt},
$G(t,j)=\ee^{\lambda_*j-tH_0(\lambda_*e_1)}\pheat_1(t,j)>0$ for every
$t>0$ and $j\in\Z$.  Log-concavity therefore implies that the successive
ratios $G(t,j+1)/G(t,j)$ are nonincreasing in $j$; multiplying two
successive ratios shows that $G(t,j+1)/G(t,j-1)$ is also nonincreasing.
\end{proof}

\begin{lemma}\label{lem:autonomous-leading-lower}
Let $K>0$, and let $Q$ solve
\[
\left\{
\begin{aligned}
    Q_t&=\Delta_{\Z}Q+aQ-KQ^2,\\
    0&\leq Q(0,\cdot)\leq a/K.
\end{aligned}
\right.
\]
with nonzero finitely supported initial data. There are $c_0>0$,
$L_0\geq1$ and $t_0\geq1$, depending only on $a,K,Q(0,\cdot)$, such that
\begin{equation}\label{eq:autonomous-leading-lower}
    Q(t,j)\geq c_0t^{-3/2}(j-c_*t)
                         \ee^{-\lambda_*(j-c_*t)}
\end{equation}
whenever $t\geq t_0$ and $L_0\leq j-c_*t\leq\sqrt t$.
\end{lemma}

\begin{proof}
To construct $Q$, set $\Pi_K(s)=\min\{a/K,\max\{0,s\}\}$ and use
the globally Lipschitz reaction $a\Pi_K(s)-K\Pi_K(s)^2$ in place of
$as-Ks^2$. The fixed-point construction in Lemma~\ref{lem:comparison}
gives a global $C^1([0,\infty);\ell^\infty(\Z))$ solution.
The constants $0$ and $a/K$ solve the truncated equation.
Lemma~\ref{lem:comparison} and the initial inequalities give
$0\leq Q\leq a/K$, so the truncated and original reactions agree.

Take $g=\1_{\{1\}}-\1_{\{-1\}}$ and put
\[
    w(t,j)=G(t,j-1)-G(t,j+1),\qquad
    Z(t,j)=\ee^{-\lambda_*(j-c_*t)}w(t,j),\qquad W(t,j)=Z(t,j)_+.
\]
The tilt identity \eqref{eq:kernel-tilt} and
$c_*\lambda_*=2(\cosh\lambda_*-1)+a$ give
\begin{equation}\label{eq:positive-mode-representation}
    Z(t)=\ee^{at}\ee^{t\Delta_{\Z}}
       \bigl(\ee^{-\lambda_*}\1_{\{1\}}
                    -\ee^{\lambda_*}\1_{\{-1\}}\bigr).
\end{equation}
Thus $Z\in C^1([0,\infty);\ell^\infty(\Z))$ and
$Z_t=\Delta_{\Z}Z+aZ$. This equation also follows by substituting
$w=\ee^{\lambda_*(j-c_*t)}Z$ into $w_t=\mathcal L_\parallel w$:
in \eqref{eq:Lparallel}, the off-diagonal coefficients become one,
while the diagonal coefficient, including the time derivative of
the exponential factor, is $c_*\lambda_*-2\cosh\lambda_*=a-2$.

The positive-part map is 1-Lipschitz, so $W$ is locally Lipschitz in
time at each site. Where $Z(t,j)>0$, one has $W_t=Z_t$ and
$W(t,j\pm1)\geq Z(t,j\pm1)$, hence $W_t\leq\Delta_{\Z}W+aW$.
Where $Z(t,j)<0$, one has $W_t=W=0$ and
$\Delta_{\Z}W=W(t,j-1)+W(t,j+1)\geq0$. At a time with
$Z(t,j)=0$ at which $W(\cdot,j)$ is differentiable, $W(\cdot,j)$ has
a minimum and hence $W_t(t,j)=0$; there too $W=0$ and
$\Delta_{\Z}W\geq0$. For each site, the times at which $W$ is not
differentiable form a null set. Their union over $j\in\Z$ is a null
set, outside which
\begin{equation}\label{eq:positive-part-inequality}
    W_t\leq\Delta_{\Z}W+aW
\end{equation}
holds for every $j$ and almost every $t$.

We first prove the uniform estimate
\begin{equation}\label{eq:integrable-positive-mode}
    \sup_{j\in\Z}W(t,j)\leq C_{\rm mode}(1+t)^{-3/2}\qquad(t\geq0)
\end{equation}
with $C_{\rm mode}$ depending only on $a$. For the chosen $g$,
one has $M_1=1$. Put
$c_{\rm G}=(\sqrt{4\pi}D_*^{3/2})^{-1}>0$. By
\eqref{eq:longitudinal-asymptotic} with $A=2$, there are
$E_*>0$ and $T_\ell\geq1$, depending only on $a$, such that
\begin{equation}\label{eq:mode-window-error}
 \left|w(t,j)-c_{\rm G}(s+\delta_*)t^{-3/2}
                   \ee^{-s^2/(4D_*t)}\right|
 \leq E_*(1+|s|)t^{-2},\qquad s=j-c_*t,
\end{equation}
for $t\geq T_\ell$ and $|s|\leq2\sqrt t$.
Let $J(t)=\lfloor c_*t\rfloor-2$, so $-3<J(t)-c_*t\leq-2$.
Choose $T_{\rm sign}\geq\max\{T_\ell,9\}$ so that
$4E_*T_{\rm sign}^{-1/2}\leq(3/4)c_{\rm G}\ee^{-1/(4D_*)}$.
For $t\geq T_{\rm sign}$, formula \eqref{eq:mode-window-error}
and $0<\delta_*<1/2$ give
\[
 w(t,J(t))\leq-\frac32c_{\rm G}\ee^{-1/(4D_*)}t^{-3/2}
                     +4E_*t^{-2}<0.
\]
By Lemma~\ref{lem:logconcavity}, for $j\leq J(t)$,
\[
 \frac{G(t,j+1)}{G(t,j-1)}
 \geq\frac{G(t,J(t)+1)}{G(t,J(t)-1)}>1.
\]
Thus $w(t,j)<0$ there. For $t\geq T_{\rm sign}$, positivity of
$W(t,j)$ is therefore possible only when $s=j-c_*t>-3$.
For $t\geq T_{\rm sign}$ and $-3\leq s\leq\sqrt t$,
\eqref{eq:mode-window-error} gives
\[
 W(t,j)\leq C(1+|s|)\ee^{-\lambda_*s}t^{-3/2}
          \leq C_{\rm near}t^{-3/2},
\]
because $(1+|s|)\ee^{-\lambda_*s}$ is bounded on $[-3,\infty)$.
For $s>\sqrt t$, the probability bound $0\leq G\leq1$ gives
$|w|\leq2$, hence
$W(t,j)\leq2\ee^{-\lambda_*\sqrt t}\leq C_{\rm far}t^{-3/2}$.
The constants $C_{\rm near},C_{\rm far}$ depend only on $a$.
On $0\leq t\leq T_{\rm sign}$,
\eqref{eq:positive-mode-representation} and the heat-semigroup
contraction give $\|Z(t)\|_\infty\leq
\ee^{aT_{\rm sign}}(\ee^{-\lambda_*}+\ee^{\lambda_*})$.
Since $t^{-3/2}\leq2^{3/2}(1+t)^{-3/2}$ for $t\geq1$,
\eqref{eq:integrable-positive-mode} holds with the fixed choice
\[
\begin{split}
 C_{\rm mode}:={}&2^{3/2}(C_{\rm near}+C_{\rm far})\\
 &+(1+T_{\rm sign})^{3/2}\ee^{aT_{\rm sign}}
       (\ee^{-\lambda_*}+\ee^{\lambda_*}).
\end{split}
\]

Choose a site $j_0$ with $Q(0,j_0)>0$ and let
$\widetilde W(t,j)=W(t,j-j_0+1)$. Define
\[
    \chi(t)=\left(\chi_0^{-1}
                 +KC_{\rm mode}\int_0^t(1+s)^{-3/2}\dd s\right)^{-1},
\]
where
$\chi_0:=\min\{\ee^{\lambda_*}Q(0,j_0),a/(KC_{\rm mode})\}>0$.
Set $\chi_{\min}=(\chi_0^{-1}+2KC_{\rm mode})^{-1}$. Then
\[
    0<\chi_{\min}\leq\chi(t)\leq\chi_0.
\]
The integer translate $\widetilde W$ satisfies
\eqref{eq:positive-part-inequality}, since the nearest-neighbor
operator commutes with integer shifts. Differentiating $\chi$ gives
$\chi'=-KC_{\rm mode}(1+t)^{-3/2}\chi^2$. Therefore
\[
\begin{split}
    &(\partial_t-\Delta_{\Z})(\chi\widetilde W)
               -a\chi\widetilde W+K\chi^2\widetilde W^2\\
    &\qquad\leq
       \chi^2\widetilde W\bigl[-KC_{\rm mode}(1+t)^{-3/2}
                                      +K\widetilde W\bigr]\leq0,
\end{split}
\]
where the last inequality is exactly
$\widetilde W(t,j)\leq C_{\rm mode}(1+t)^{-3/2}$ from
\eqref{eq:integrable-positive-mode}.
At $t=0$, $\widetilde W(0,\cdot)$ is supported at $j_0$ and equals
$\ee^{-\lambda_*}$ there.  Hence
$\chi_0\widetilde W(0,j_0)\leq Q(0,j_0)$ by the definition of $\chi_0$,
and at every other site $\chi_0\widetilde W(0,j)=0\leq Q(0,j)$.
The functions $\chi\widetilde W$ and $Q$ take values in $[0,a/K]$,
where $s\mapsto as-Ks^2$ is Lipschitz.
Lemma~\ref{lem:comparison}, with $D=\Z$, $U=\chi\widetilde W$,
$W=Q$ and $F(t,s)=as-Ks^2$, gives $Q\geq\chi\widetilde W$.

Put $h=1-j_0$ and choose $L_0=2(1+|h|)$. If
$L_0\leq s=j-c_*t\leq\sqrt t$ and $t\geq h^2$, then
$s':=j-j_0+1-c_*t=s+h$ satisfies $1\leq s/2\leq s'\leq2\sqrt t$.
Set $c_{\rm long}=c_{\rm G}\ee^{-1/D_*}/2$. For
$t\geq T_\ell$ with $2E_*t^{-1/2}\leq c_{\rm long}$,
\eqref{eq:mode-window-error} gives
\[
 w(t,j-j_0+1)\geq
    \bigl(c_{\rm G}\ee^{-1/D_*}-2E_*t^{-1/2}\bigr)s't^{-3/2}
 \geq c_{\rm long}s't^{-3/2}>0.
\]
Consequently $\widetilde W=\ee^{-\lambda_*s'}w(t,j-j_0+1)$ there.
Using $Q\geq\chi\widetilde W$, $\chi\geq\chi_{\min}$ and
$\ee^{-\lambda_*s'}=\ee^{-\lambda_*h}\ee^{-\lambda_*s}$ yields
\[
 Q(t,j)\geq\frac{\chi_{\min}c_{\rm long}}2\ee^{-\lambda_*h}
               t^{-3/2}s\ee^{-\lambda_*s}.
\]
Choose $c_0=(\chi_{\min}c_{\rm long}/2)\ee^{-\lambda_*h}$ and
$t_0=\max\{1,T_\ell,h^2,(2E_*/c_{\rm long})^2\}$.
These choices prove \eqref{eq:autonomous-leading-lower}.
\end{proof}

The integrability of the bound in \eqref{eq:integrable-positive-mode}
ensures $\chi\geq\chi_{\min}>0$. The next lemma constructs a subsolution
whose argument contains the logarithmic delay.

\begin{lemma}\label{lem:delayed-pulse}
Fix $\beta\geq0$, $K>0$ and $T\geq1$, and set
$b_\beta:=(\beta+3/2)/\lambda_*$.
Choose $0<\mu<\lambda_*$ and define
\begin{equation}\label{eq:pulse-parameters}
\begin{gathered}
    \omega_\mu:=\frac1{2(\ee^\mu-1)},\qquad
    P_\mu(r):=r-\omega_\mu+\omega_\mu\ee^{-\mu r},\\
    \sigma_\mu:=2\bigl(\cosh(\lambda_*+\mu)-1\bigr)+a
                                      -c_*(\lambda_*+\mu)>0.
\end{gathered}
\end{equation}
For $\varepsilon>0$, let
\begin{equation}\label{eq:critical-pulse}
    \Phi_\varepsilon(r)=
    \begin{cases}
      \varepsilon\ee^{-\lambda_*r}P_\mu(r),&r>0,\\
      0,&r\leq0.
    \end{cases}
\end{equation}
There are $\varepsilon_*>0$ and $\tau_*>0$, depending only on
$a,K,\beta,\mu$, such that, for every
$0<\varepsilon\leq\varepsilon_*$, the function
\[
    \Psi(t,j):=\Phi_\varepsilon
          \bigl(j-c_*t+b_\beta\log(t+T)\bigr)
\]
has values in $[0,a/K]$ and satisfies
\begin{equation}\label{eq:delayed-pulse-subsolution}
    \Psi_t\leq\Delta_{\Z}\Psi+
               \left(a-\frac\beta{t+T}\right)\Psi-K\Psi^2
\end{equation}
for $t\geq0$ with $t+T\geq\tau_*$, in the almost-everywhere sense
at each site.
\end{lemma}

\begin{proof}
The function
$k\mapsto2(\cosh k-1)+a-c_*k$ is strictly convex. It and its first
derivative vanish at $k=\lambda_*$, so $\sigma_\mu>0$.
Since $\ee^\mu-1>\mu$, one has $\omega_\mu\mu<1/2$.
The identities $P_\mu(0)=0$ and
$P_\mu'(r)=1-\omega_\mu\mu\ee^{-\mu r}$ give
\[
    \frac r2\leq P_\mu(r)\leq r,\qquad
    \frac12<P_\mu'(r)\leq1\qquad(r\geq0).
\]
On $[-1,0]$, convexity of $P_\mu$ and the values
$P_\mu(-1)=-1/2$, $P_\mu(0)=0$ give $P_\mu\leq0$.

Temporarily define
$F_\varepsilon(r)=\varepsilon\ee^{-\lambda_*r}P_\mu(r)$
on the whole real line. For a function of a real argument, write
$\Delta_1F(r)=F(r-1)-2F(r)+F(r+1)$.
The critical identities \eqref{eq:critical-system} give the exact
calculation
\begin{equation}\label{eq:pulse-linear-residual}
    -c_*F_\varepsilon'(r)-\Delta_1F_\varepsilon(r)
                        -aF_\varepsilon(r)
       =-\varepsilon\omega_\mu\sigma_\mu
                                     \ee^{-(\lambda_*+\mu)r}.
\end{equation}
Indeed, the operator on the left annihilates both
$\ee^{-\lambda_*r}$ and $r\ee^{-\lambda_*r}$; applying it to
$\ee^{-(\lambda_*+\mu)r}$ gives the remaining term.

For $r>0$, the only possible difference between
$\Delta_1\Phi_\varepsilon(r)$ and $\Delta_1F_\varepsilon(r)$
is the value at $r-1$ when $0<r<1$. That value of
$F_\varepsilon$ is nonpositive, while its truncated value is zero.
Consequently
$\Delta_1\Phi_\varepsilon(r)\geq\Delta_1F_\varepsilon(r)$.
Set $M_\mu:=\sup_{r\geq0}r^2\ee^{-(\lambda_*-\mu)r}
=4/(\ee^2(\lambda_*-\mu)^2)$. Choose
\[
    0<\varepsilon_*\leq
      \min\left\{\frac{a\ee\lambda_*}{K},
                  \frac{\omega_\mu\sigma_\mu}{2KM_\mu}\right\}.
\]
Since $P_\mu(r)\leq r$ and
$\sup_{r\geq0}r\ee^{-\lambda_*r}=1/(\ee\lambda_*)$, the condition
$\varepsilon_*\leq a\ee\lambda_*/K$ gives
$\Phi_\varepsilon(r)\leq\varepsilon/(\ee\lambda_*)\leq a/K$.
Also
\[
 K\Phi_\varepsilon(r)^2
 \leq K\varepsilon^2r^2\ee^{-2\lambda_*r}
 =K\varepsilon^2
   \bigl(r^2\ee^{-(\lambda_*-\mu)r}\bigr)
   \ee^{-(\lambda_*+\mu)r}
 \leq K\varepsilon^2M_\mu\ee^{-(\lambda_*+\mu)r}.
\]
The condition $\varepsilon_*\leq\omega_\mu\sigma_\mu/(2KM_\mu)$
gives $K\varepsilon^2M_\mu\leq\varepsilon\omega_\mu\sigma_\mu/2$.  Combining this with
\eqref{eq:pulse-linear-residual} and
$\Delta_1\Phi_\varepsilon\geq\Delta_1F_\varepsilon$ gives
\begin{equation}\label{eq:pulse-nonlinear-residual}
    -c_*\Phi_\varepsilon'(r)-\Delta_1\Phi_\varepsilon(r)
                     -a\Phi_\varepsilon(r)+K\Phi_\varepsilon(r)^2
       \leq-\frac{\varepsilon\omega_\mu\sigma_\mu}{2}
                                  \ee^{-(\lambda_*+\mu)r}
       \qquad(r>0).
\end{equation}

Put $\tau=t+T$ and $r=j-c_*t+b_\beta\log\tau$. At $r>0$, the
full residual of $\Psi(t,j)=\Phi_\varepsilon(r)$ is
\begin{equation}\label{eq:pulse-full-residual}
\begin{split}
 &\Psi_t-\Delta_{\Z}\Psi-\left(a-\frac\beta\tau\right)\Psi+K\Psi^2\\
 &\quad=-c_*\Phi_\varepsilon'(r)-\Delta_1\Phi_\varepsilon(r)
          -a\Phi_\varepsilon(r)+K\Phi_\varepsilon(r)^2\\
 &\qquad+\frac{b_\beta\Phi_\varepsilon'(r)+\beta\Phi_\varepsilon(r)}\tau.
\end{split}
\end{equation}
Since $\beta-\lambda_*b_\beta=-3/2$, the last term satisfies
\begin{equation}\label{eq:pulse-temporal-residual}
\begin{split}
    \frac1\tau\bigl(b_\beta\Phi_\varepsilon'(r)
                                      +\beta\Phi_\varepsilon(r)\bigr)
      &=\frac{\varepsilon\ee^{-\lambda_*r}}\tau
          \left[b_\beta P_\mu'(r)-\frac32P_\mu(r)\right]\\
      &\leq\frac{\varepsilon\ee^{-\lambda_*r}}\tau
                       \left(b_\beta-\frac34r\right).
\end{split}
\end{equation}
For $r\geq R_0:=1+4b_\beta/3$, this quantity is nonpositive.
For $0<r<R_0$, the right-hand side of
\eqref{eq:pulse-temporal-residual} is at most
$\varepsilon b_\beta\tau^{-1}\ee^{-\lambda_*r}$.  Since
$\ee^{-\lambda_*r}\leq\ee^{\mu R_0}
\ee^{-(\lambda_*+\mu)r}$ in this interval, the choice
\[
    \tau\geq\tau_*:=
       \max\left\{1,\frac{2b_\beta}{c_*},
          \frac{2b_\beta\ee^{\mu R_0}}{\omega_\mu\sigma_\mu}\right\}
\]
implies
\[
 \frac{\varepsilon b_\beta}{\tau}\ee^{-\lambda_*r}
 \leq\frac{\varepsilon\omega_\mu\sigma_\mu}{2}
       \ee^{-(\lambda_*+\mu)r}.
\]
Thus \eqref{eq:pulse-full-residual} is nonpositive for $r>0$ by
\eqref{eq:pulse-nonlinear-residual}. The constants
$\varepsilon_*$ and $\tau_*$ are independent of $T$ and of the chosen
$\varepsilon\in(0,\varepsilon_*]$.
On the open set of times where
$r(t,j):=j-c_*t+b_\beta\log(t+T)<0$, one has
$\Psi(t,j)=\Psi_t(t,j)=0$.  Since both neighboring
values of $\Psi$ are nonnegative,
$\Delta_{\Z}\Psi(t,j)\geq0$, and hence
\[
 \Psi_t-\Delta_{\Z}\Psi
   -\left(a-\frac\beta{t+T}\right)\Psi+K\Psi^2\leq0
\]
at every such time. For $r>0$,
$|\Phi_\varepsilon'(r)|\leq
\varepsilon(1+\lambda_*r)\ee^{-\lambda_*r}\leq\varepsilon$.
Since $\Phi_\varepsilon(0)=0$ and $\Phi_\varepsilon=0$ on $(-\infty,0]$,
it is globally $\varepsilon$-Lipschitz. Hence $\Psi(\cdot,j)$ is locally
Lipschitz at each site. For $t+T\geq\tau_*$,
\[
 \frac{\dd}{\dd t}\bigl(j-c_*t+b_\beta\log(t+T)\bigr)
   =-c_*+\frac{b_\beta}{t+T}\leq-\frac{c_*}{2}<0.
\]
Thus the moving coordinate at any fixed site can equal zero at most once
for $t\geq\max\{0,\tau_*-T\}$. The inequality in
\eqref{eq:delayed-pulse-subsolution} is asserted only for almost every
time at each site, so these isolated crossing times require no value of
$\Phi_\varepsilon'$ at $0$. At every other time the ordinary chain rule
applies. The nonpositive residual in \eqref{eq:pulse-full-residual}
for $r>0$ and the inequality at $r<0$ prove
\eqref{eq:delayed-pulse-subsolution}.
\end{proof}

\begin{proposition}\label{prop:scalar-logarithmic-lower}
Let $\beta\geq0$, $K>0$ and $T\geq1$. Let $q$ solve
\begin{equation}\label{eq:scalar-growth-loss}
\left\{
\begin{aligned}
    q_t&=\Delta_{\Z}q+\left(a-\frac\beta{t+T}\right)q-Kq^2,\\
    q(0,\cdot)&=q_0.
\end{aligned}
\right.
\end{equation}
where $q_0$ is nonzero and finitely supported, and
$0\leq q_0\leq a/K$. Fix $0<\mu<\lambda_*$ and use
$\Phi_\varepsilon$ from Lemma~\ref{lem:delayed-pulse}.
There are $\varepsilon\in(0,\varepsilon_*]$ and $t_0\geq1$,
depending only on $a,K,\beta,\mu,T,q_0$, such that, with
\[
    r(t,j):=j-c_*t+b_\beta\log(t+T),\qquad
    b_\beta=\frac{\beta+3/2}{\lambda_*},
\]
one has
\begin{equation}\label{eq:scalar-profile-lower}
    q(t,j)\geq\Phi_\varepsilon(r(t,j))
    \qquad\bigl(t\geq t_0,\ r(t,j)\leq(t+T)^{1/4}\bigr).
\end{equation}
In particular, for some $\eta>0$,
\begin{equation}\label{eq:scalar-positive-seed}
    q\left(t,
       \left\lceil c_*t-b_\beta\log(t+T)+1\right\rceil\right)
       \geq\eta\qquad(t\geq t_0),
\end{equation}
after increasing $t_0$ if necessary.
\end{proposition}

\begin{proof}
Define $\Pi_K(s)=\min\{a/K,\max\{0,s\}\}$ and
$F_\Pi(t,s)=(a-\beta/(t+T))\Pi_K(s)-K\Pi_K(s)^2$.
The map $q\mapsto\Delta_{\Z}q+F_\Pi(t,q)$ is globally Lipschitz on
$\ell^\infty(\Z)$, with constant at most $4+a+\beta/T$, uniformly in
$t\geq0$. The fixed-point construction in Lemma~\ref{lem:comparison}
therefore gives a global $C^1([0,\infty);\ell^\infty(\Z))$ solution
with initial datum $q_0$.
For the equation with reaction $F_\Pi$, the constant $0$ is a solution,
and the residual at $q\equiv a/K$ is
$a\beta/(K(t+T))\geq0$. Lemma~\ref{lem:comparison}, with
$D=\Z$, $F=F_\Pi$ and $0\leq q_0\leq a/K$, gives
$0\leq q\leq a/K$. Hence $F_\Pi(t,q)=(a-\beta/(t+T))q-Kq^2$,
so this solution solves \eqref{eq:scalar-growth-loss}.
Since $aq-Kq^2\geq0$ on this range,
\[
  q_t\geq\Delta_{\Z}q-\frac\beta Tq.
\]
In \eqref{eq:variation-constants}, take
$A=\Delta_{\Z}-\beta/T$ and use
the nonnegative source
$g_q(t)=(a+\beta/T-\beta/(t+T))q(t)-Kq(t)^2$.
The semigroup $\ee^{tA}=\ee^{-\beta t/T}\ee^{t\Delta_{\Z}}$ preserves
nonnegativity, so
\begin{equation}\label{eq:scalar-positivity}
 q(t,j)\geq\ee^{-\beta t/T}
    \sum_{\ell\in\Z}\pheat_1(t,j-\ell)q_0(\ell).
\end{equation}
By \eqref{eq:heat-semigroup}, $\pheat_1(t,m)>0$ for $t>0$.
Since $q_0\not\equiv0$, formula \eqref{eq:scalar-positivity} gives
$q(t,j)>0$ for every $t>0$ and $j\in\Z$.
Let $Q$ be the solution in Lemma~\ref{lem:autonomous-leading-lower}
with initial datum $q_0$, and
put $\rho(t)=(T/(t+T))^\beta$.  Since
$\rho'(t)=-\beta\rho(t)/(t+T)$ and
$Q_t-\Delta_{\Z}Q=aQ-KQ^2$, direct algebra gives
\begin{equation}\label{eq:scaled-autonomous-residual}
\begin{split}
    &(\partial_t-\Delta_{\Z})(\rho Q)
         -\left(a-\frac\beta{t+T}\right)\rho Q+K\rho^2Q^2\\
    &\hspace{4em}=-K\rho(1-\rho)Q^2\leq0.
\end{split}
\end{equation}
Since $\rho Q,q\in[0,a/K]$, $\rho(0)Q(0)=q_0$, and
\eqref{eq:scaled-autonomous-residual} is nonpositive on $\Z$,
Lemma~\ref{lem:comparison} with
$D=\Z$ and $F(t,s)=(a-\beta/(t+T))s-Ks^2$ gives
\begin{equation}\label{eq:rhoQ-comparison}
    q(t,j)\geq\rho(t)Q(t,j).
\end{equation}
Let $c_0,L_0,T_{\rm aut}$ be the constants in
Lemma~\ref{lem:autonomous-leading-lower} for initial datum $q_0$,
writing $T_{\rm aut}$ for its time threshold. Substitution in
\eqref{eq:rhoQ-comparison} gives
\begin{equation}\label{eq:nonautonomous-leading-lower}
    q(t,j)\geq c_0T^\beta(t+T)^{-\beta}t^{-3/2}
                     (j-c_*t)\ee^{-\lambda_*(j-c_*t)}
\end{equation}
for $t\geq T_{\rm aut}$ and $L_0\leq j-c_*t\leq\sqrt t$.

Write $\tau=t+T$ and set
$B(t):=c_*t-b_\beta\log\tau+\tau^{1/4}$.  Its derivative is
\[
 B'(t)=c_*-\frac{b_\beta}{t+T}+\frac14(t+T)^{-3/4},
\]
Thus $B'(t)>c_*/2$ whenever $t+T\geq2b_\beta/c_*$.
The set $D(t)=\{j\in\Z:j\leq B(t)\}$ has the single exterior vertex
$\lfloor B(t)\rfloor+1$. At that vertex, equivalently
$B(t)<j\leq B(t)+1$,
one has
$\tau^{1/4}<r:=r(t,j)\leq\tau^{1/4}+1$ and
$s:=j-c_*t=r-b_\beta\log\tau$.
Choose $t_{\rm ext}\geq\max\{1,T_{\rm aut}\}$ so that, for all
$t\geq t_{\rm ext}$,
\[
 b_\beta\log\tau\leq\tfrac12\tau^{1/4},\qquad
 \tau^{1/4}\geq2L_0,\qquad \tau^{1/4}+1\leq\sqrt t.
\]
Such a choice exists because $\log(t+T)=o((t+T)^{1/4})$ and
$(t+T)^{1/4}=o(\sqrt t)$ as $t\to\infty$, with $T$ fixed.
For every adjacent exterior vertex these inequalities give
$s=r-b_\beta\log\tau\geq r/2\geq L_0$ and $s\leq\sqrt t$.
Since $\lambda_*b_\beta=\beta+3/2$, estimate
\eqref{eq:nonautonomous-leading-lower} becomes
\[
\begin{split}
    q(t,j)
      &\geq c_0T^\beta\left(\frac\tau t\right)^{3/2}
                                  s\ee^{-\lambda_*r}\\
      &\geq c_1r\ee^{-\lambda_*r},
       \qquad c_1:=c_0T^\beta/2>0.
\end{split}
\]
On the other hand,
$\Phi_\varepsilon(r)\leq\varepsilon r\ee^{-\lambda_*r}$.
Thus $\Phi_\varepsilon(r(t,j))\leq q(t,j)$ for
$t\geq t_{\rm ext}$ and $B(t)<j\leq B(t)+1$, provided
$\varepsilon\leq c_1$.

Set $t_0=\max\{1,t_{\rm ext},\tau_*-T,16-T\}$. Since
$\tau_*\geq2b_\beta/c_*$, this choice gives $B'>0$ on $[t_0,\infty)$,
$t_0+T\geq\tau_*$ and $(t_0+T)^{1/4}\geq2$.
The time $t_0$ is fixed before the amplitude $\varepsilon$ is chosen.
At time $t_0$, the set
\[
 \mathcal J_0:=\{j\in\Z:0<r(t_0,j)\leq(t_0+T)^{1/4}\}
\]
is finite and nonempty: its defining interval in $j$ has length
$(t_0+T)^{1/4}\geq2$. Estimate \eqref{eq:scalar-positivity} gives
$q(t_0,j)>0$ for every $j\in\mathcal J_0$.  Since
$\Phi_\varepsilon(r)=\varepsilon
\ee^{-\lambda_*r}P_\mu(r)$ for $r>0$, define
\[
 \varepsilon_0:=
 \min_{j\in\mathcal J_0}
 \frac{q(t_0,j)}{\ee^{-\lambda_*r(t_0,j)}P_\mu(r(t_0,j))}>0.
\]
Fix $\varepsilon=\min\{\varepsilon_*,c_1,\varepsilon_0\}>0$.  Then
$\Phi_\varepsilon(r(t_0,j))\leq q(t_0,j)$ for every
$j\in\mathcal J_0$; for $r(t_0,j)\leq0$ the pulse is zero.  Thus
\[
    \Phi_\varepsilon(r(t_0,j))\leq q(t_0,j)
                 \qquad(j\leq B(t_0)).
\]
Fix $t_1>t_0$ and put $\Psi(t,j)=\Phi_\varepsilon(r(t,j))$.
On $[t_0,t_1]$, $\Psi$ is locally Lipschitz in time at each site and
satisfies \eqref{eq:delayed-pulse-subsolution} in $j\leq B(t)$, while
$q$ is continuously differentiable in time and solves
\eqref{eq:scalar-growth-loss} there. Both take values in $[0,a/K]$.
For $F(t,s)=(a-\beta/(t+T))s-Ks^2$, one has
$|\partial_sF|\leq a+\beta/T$ on this range.
We have $\Psi(t_0,j)\leq q(t_0,j)$ for $j\leq B(t_0)$, and
$\Psi(t,j)\leq q(t,j)$ for $B(t)<j\leq B(t)+1$ at every
$t\in[t_0,t_1]$.
Together with $B'>0$, these verify all hypotheses of
Lemma~\ref{lem:moving-halfline-comparison}. Apply it with $U=\Psi$ and
$W=q$ to obtain \eqref{eq:scalar-profile-lower} on $[t_0,t_1]$.
Since $t_1>t_0$ is arbitrary, \eqref{eq:scalar-profile-lower} holds for
all $t\geq t_0$, including the integer-crossing times of $B(t)$.
Put $j_t=\lceil c_*t-b_\beta\log(t+T)+1\rceil$. Then
$1\leq r(t,j_t)<2\leq(t+T)^{1/4}$ for $t\geq t_0$.
For $1\leq r\leq2$, the bound $P_\mu(r)\geq r/2$ gives
$\Phi_\varepsilon(r)\geq(\varepsilon/2)\ee^{-2\lambda_*}$.
Thus \eqref{eq:scalar-profile-lower} at $j=j_t$ proves
\eqref{eq:scalar-positive-seed} with the fixed choice
$\eta=(\varepsilon/2)\ee^{-2\lambda_*}>0$, independently of the
fractional part of $c_*t-b_\beta\log(t+T)$.
\end{proof}

\section{Lower bound on coordinate axes}\label{sec:amplification}

Lemma~\ref{lem:amplification} uses finite boxes to prove that
$u(t_0,z)\geq\eta$ implies $u(t_0+s,z)\geq\theta$ for all
$s\geq S(\eta,\theta,d,f)$.

\begin{lemma}\label{lem:elliptic-liouville}
Suppose $q:\Z^d\to[0,1]$ satisfies
$\Delta_{\Z^d}q+f(q)=0$ and $\inf_{\Z^d}q>0$. Then $q\equiv1$.
\end{lemma}

\begin{proof}
Let $m=\inf q>0$ and choose $x_n$ so that $q(x_n)\to m$.
Every neighbor of $x_n$ has value at least $m$, and consequently
\[
    -f(q(x_n))=\Delta_{\Z^d}q(x_n)
       \geq-2d(q(x_n)-m).
\]
Since $q(x_n)\to m$ and $f$ is continuous, taking $n\to\infty$ in
$-f(q(x_n))\geq-2d(q(x_n)-m)$ gives $f(m)\leq0$.  The assumptions
$f>0$ on $(0,1)$ and $m\in(0,1]$ force $m=1$, and $q\leq1$ then gives
$q\equiv1$.
\end{proof}

\begin{lemma}\label{lem:box-equilibria}
For $L\in\N$, put $Q_L:=\{-L,\ldots,L\}^d$. If $L$ is sufficiently
large, the equation
\[
    \Delta_{\Z^d}q+f(q)=0\text{ in }Q_L,
    \qquad q=0\text{ on }\Z^d\setminus Q_L
\]
has a solution with $0<q<1$ in $Q_L$.
For any choice of such positive solutions $q_L$, one has
$q_L(x)\to1$ for every fixed $x\in\Z^d$ as $L\to\infty$.
\end{lemma}

\begin{proof}
Define
\[
    \phi_L(x):=\prod_{k=1}^d
           \cos\left(\frac{\pi x_k}{2L+2}\right)\quad(x\in Q_L),
    \qquad \phi_L=0\quad\text{outside }Q_L.
\]
Writing $h_L=\pi/(2L+2)$, the identity
$\cos(\xi+h_L)+\cos(\xi-h_L)=2\cos h_L\cos\xi$ gives
\[
    -\Delta_{\Z^d}\phi_L=\mu_L\phi_L\text{ in }Q_L,
    \qquad
    \mu_L=2d\left(1-\cos\frac{\pi}{2L+2}\right)\to0.
\]
This calculation also holds at boundary vertices because
$\cos(h_L(L+1))=\cos(-h_L(L+1))=0$.
Moreover, $0<\phi_L\leq1$ in $Q_L$ and $\phi_L(0)=1$.
Choose $L$ with $\mu_L<a$. Since $f(s)/s\to a$ as $s\downarrow0$,
there is $\varepsilon_L\in(0,1)$ such that
$f(s)>\mu_Ls$ for $0<s\leq\varepsilon_L$.
For $0<\varepsilon\leq\varepsilon_L$ and $x\in Q_L$,
\[
 \Delta_{\Z^d}(\varepsilon\phi_L)(x)
      +f(\varepsilon\phi_L(x))
 =-\mu_L\varepsilon\phi_L(x)+f(\varepsilon\phi_L(x))>0,
\]
so the time-independent function $\varepsilon\phi_L$ is a strict
subsolution.  Let $\mathbf 1_{Q_L}$ equal one on $Q_L$ and zero outside.
For $x\in Q_L$ one has
$\Delta_{\Z^d}\mathbf1_{Q_L}(x)\leq0$ and $f(1)=0$, so
$\mathbf1_{Q_L}$ is a supersolution for the differential equation
and zero exterior condition in \eqref{eq:box-dirichlet-evolution}.
Let $z_L$ solve the finite system
\begin{equation}\label{eq:box-dirichlet-evolution}
\left\{
\begin{aligned}
 (z_L)_t&=\Delta_{\Z^d}z_L+\widetilde f(z_L)
     &&\text{in }Q_L,\\
 z_L&=0
     &&\text{outside }Q_L,\\
 z_L(0)&=\varepsilon\phi_L.
\end{aligned}
\right.
\end{equation}
With zero exterior values, the vector field in
\eqref{eq:box-dirichlet-evolution} is globally Lipschitz on
$\R^{Q_L}$ in the maximum norm, with constant at most $4d+L_f$.
The fixed-point construction in Lemma~\ref{lem:comparison} therefore
gives a unique global solution. We have
$\varepsilon\phi_L=z_L(0)\leq\mathbf1_{Q_L}$.
The functions $\varepsilon\phi_L,z_L,\mathbf1_{Q_L}$ vanish on
$\partial_{\rm ext}Q_L$ and satisfy their subsolution, solution and
supersolution inequalities in $Q_L$, respectively.
Lemma~\ref{lem:comparison}, with $D=Q_L$ and $F(t,s)=\widetilde f(s)$, gives
\[
 \varepsilon\phi_L\leq z_L(t)\leq\mathbf1_{Q_L}\qquad(t\geq0).
\]
For $h>0$, $z_L(h)\geq z_L(0)$. The functions $U(t,x)=z_L(t,x)$ and $W(t,x)=z_L(t+h,x)$ satisfy
the differential equation and zero exterior condition in
\eqref{eq:box-dirichlet-evolution}, with $U(0)\leq W(0)$.
Lemma~\ref{lem:comparison} therefore gives $z_L(t+h)\geq z_L(t)$.  Thus each component is nondecreasing and bounded,
so $z_L(t)\to q_L\in[0,1]^{Q_L}$; extend $q_L$ by zero outside
$Q_L$. Equation~\eqref{eq:box-dirichlet-evolution} and continuity
of $f$ give
$(z_L)_t(t,x)\to\Delta_{\Z^d}q_L(x)+f(q_L(x))$ for each $x\in Q_L$.
If this limit were nonzero, $(z_L)_t(t,x)$ would eventually have a fixed
sign and absolute value bounded away from zero, contradicting convergence.
Hence $\Delta_{\Z^d}q_L+f(q_L)=0$ in $Q_L$, and
$q_L\geq\varepsilon\phi_L>0$ there.
If $q_L(x_0)=1$ at a vertex $x_0\in Q_L$, then
$0=\Delta q_L(x_0)+f(1)=\sum_{z\sim x_0}(q_L(z)-1)$; since every summand is
nonpositive, all neighbors of $x_0$ equal one.  Repeating this along a
nearest-neighbor path to the exterior contradicts the zero Dirichlet data.
Therefore $q_L<1$ in $Q_L$.

We next prove a lower bound independent of $L$ away from the exterior.
Fix $L_{\rm b}\in\N$ with $\mu_{L_{\rm b}}<a$, and choose
$\eta_{\rm b}\in(0,1)$ such that
$f(s)>\mu_{L_{\rm b}}s$ for $0<s\leq\eta_{\rm b}$.
Let $q_L$ satisfy $\Delta_{\Z^d}q_L+f(q_L)=0$ in $Q_L$, with
$0<q_L<1$ in $Q_L$ and zero exterior values.
For $z+Q_{L_{\rm b}}\subset Q_L$, define
\[
 \sigma_*:=\min\left\{\eta_{\rm b},
    \min_{x\in z+Q_{L_{\rm b}}}
        \frac{q_L(x)}{\phi_{L_{\rm b}}(x-z)}\right\}>0.
\]
Then $q_L-\sigma_*\phi_{L_{\rm b}}(\cdot-z)\geq0$ on the whole
lattice: the definition gives this inequality in the translated box,
and outside it $\phi_{L_{\rm b}}(\cdot-z)=0$ while $q_L\geq0$.
If $\sigma_*<\eta_{\rm b}$, the finite minimum is attained at some
$x\in z+Q_{L_{\rm b}}$, where
$q_L(x)=\sigma_*\phi_{L_{\rm b}}(x-z)$.
All neighboring values of the difference are nonnegative, so
\[
 \Delta_{\Z^d}\bigl(q_L-\sigma_*\phi_{L_{\rm b}}(\cdot-z)\bigr)(x)
 \geq0.
\]
At the same vertex, the stationary equation and the computed eigenvalue
give
\[
 \Delta_{\Z^d}\bigl(q_L-\sigma_*\phi_{L_{\rm b}}(\cdot-z)\bigr)(x)
 =\mu_{L_{\rm b}}\sigma_*\phi_{L_{\rm b}}(x-z)
       -f(\sigma_*\phi_{L_{\rm b}}(x-z))<0.
\]
This contradiction forces $\sigma_*=\eta_{\rm b}$. Evaluating at
$x=z$, where $\phi_{L_{\rm b}}(0)=1$, yields $q_L(z)\geq\eta_{\rm b}$.

Given any sequence $L_n\to\infty$, enumerate
$\Z^d=\{x_1,x_2,\ldots\}$.  Since every sequence
$(q_{L_n}(x_m))_n$ lies in $[0,1]$, the Bolzano--Weierstrass theorem
allows us to choose nested subsequences $(n_k^{(m)})_k$, $m\geq1$, such
that $q_{L_{n_k^{(m)}}}(x_i)$ converges for $i\leq m$.
The diagonal choice $n_k:=n_k^{(k)}$ then makes
$q_{L_{n_k}}(x_m)$ converge for every fixed $m$.  Denote the pointwise
limit by $q_\infty:\Z^d\to[0,1]$. For every fixed $z$, the estimate
$q_L(z)\geq\eta_{\rm b}$ holds once $z+Q_{L_{\rm b}}\subset Q_L$;
hence $q_\infty(z)\geq\eta_{\rm b}$ for every $z\in\Z^d$. For a fixed $x$, all $2d$ neighbors of $x$ also have convergent values on
this subsequence; continuity of $f$ therefore permits passage to the limit
in
$\Delta_{\Z^d}q_{L_{n_k}}(x)+f(q_{L_{n_k}}(x))=0$.  Thus
$\Delta_{\Z^d}q_\infty(x)+f(q_\infty(x))=0$ for every $x$.
Lemma~\ref{lem:elliptic-liouville} gives $q_\infty\equiv1$.
If $q_L(x)\not\to1$ at some fixed vertex $x$, one could choose
$\varepsilon>0$ and a subsequence with $q_L(x)\leq1-\varepsilon$.
A further diagonal subsequence would converge to $q_\infty\equiv1$, which
contradicts the inequality at $x$. Thus $q_L(x)\to1$ for every fixed
$x\in\Z^d$.
\end{proof}

\begin{lemma}\label{lem:amplification}
For every $\eta,\theta\in(0,1)$ there is
$S=S(\eta,\theta,d,f)>0$ such that every solution $u$ of
\eqref{eq:main} with values in $[0,1]$ satisfies
\begin{equation}\label{eq:amplification}
 u(t_0,z)\geq\eta
 \quad\Longrightarrow\quad
 u(t_0+s,z)\geq\theta\quad\text{for every }s\geq S
\end{equation}
for every $t_0\geq0$ and $z\in\Z^d$. The time $S$ does not
depend on $u$, $t_0$, or $z$.
\end{lemma}

\begin{proof}
Let $v$ solve \eqref{eq:main} with $v(0,x)=\eta\1_{\{0\}}(x)$.
The lower bound in \eqref{eq:linear-comparison} and the strict positivity
of the kernel in \eqref{eq:heat-semigroup} give
$v(1,x)\geq\eta\pheat_d(1,x)>0$ for every $x\in\Z^d$.
Use $Q_L,\phi_L,\mu_L$ from the proof of
Lemma~\ref{lem:box-equilibria}. For each $L$ with $\mu_L<a$, choose
$\varepsilon_L^{\rm sub}\in(0,1)$ such that
$f(s)>\mu_Ls$ for $0<s\leq\varepsilon_L^{\rm sub}$, and set
\[
 \varepsilon_L:=\min\left\{\varepsilon_L^{\rm sub},
    \min_{x\in Q_L}\frac{\eta\pheat_d(1,x)}{\phi_L(x)}\right\}>0.
\]
Let $z_L$ solve \eqref{eq:box-dirichlet-evolution} with
$\varepsilon=\varepsilon_L$. Its initial values satisfy
$z_L(0,x)\leq v(1,x)$ in $Q_L$. Both $z_L(s,x)$ and $v(1+s,x)$ solve
$w_s=\Delta_{\Z^d}w+f(w)$ in $Q_L$ and take values in $[0,1]$.
On $\partial_{\rm ext}Q_L$, one has $z_L(s,x)=0\leq v(1+s,x)$
for all $s\geq0$. Lemma~\ref{lem:comparison}, with $D=Q_L$,
$F(t,s)=\widetilde f(s)$, $U=z_L$ and $W(s,x)=v(1+s,x)$, gives
\[
 v(1+s,x)\geq z_L(s,x)\qquad(s\geq0,\ x\in Q_L).
\]
The proof of Lemma~\ref{lem:box-equilibria} gives
$z_L(s,0)\uparrow q_L(0)$ as $s\to\infty$ for each such $L$.
Its conclusion for any family of positive equilibria gives
$q_L(0)\to1$ as $L\to\infty$. Now fix $L$ with $q_L(0)>\theta$,
and choose $s_0\geq0$ such that $z_L(s,0)\geq\theta$ for every
$s\geq s_0$. Then $v(s,0)\geq\theta$ for every $s\geq S:=1+s_0$.
The choices of $\varepsilon_L^{\rm sub}$, $L$ and $s_0$ depend only
on $\eta,\theta,d,f$.

For a solution $u$ with $0\leq u\leq1$ and $u(t_0,z)\geq\eta$,
define $v_z(s,x):=v(s,x-z)$. Substitution in the nearest-neighbor
formula gives
$\partial_sv_z(s,x)=\Delta_{\Z^d}v_z(s,x)+f(v_z(s,x))$ and
$v_z(0,x)=\eta\1_{\{z\}}(x)\leq u(t_0,x)$.
Lemma~\ref{lem:comparison}, with $D=\Z^d$, $U=v_z$ and
$W(s,x)=u(t_0+s,x)$, yields $u(t_0+s,x)\geq v_z(s,x)$ for every
$s\geq0$. At $x=z$ and $s\geq S$, this proves
\eqref{eq:amplification}.
\end{proof}

\begin{theorem}\label{thm:nonlinear-lower}
Assume \eqref{eq:kpp}--\eqref{eq:u0}. For each $\theta\in(0,1)$,
there are $C_\theta^{-}>0$ and $T_\theta\geq2$ such that
\begin{equation}\label{eq:axial-inner-lower}
    u(t,je_1)\geq\theta
    \quad\text{if }t\geq T_\theta,
    \quad j\in\Z,\quad 0\leq j\leq m_d(t)-C_\theta^{-}.
\end{equation}
In particular,
\begin{equation}\label{eq:sharp-axial-lower}
    R_\theta(t)\geq c_*t-\frac{d+2}{2\lambda_*}\log t-C_\theta^{-}
    \qquad(t\geq T_\theta).
\end{equation}
Moreover, for every $A\geq0$ there are $C_{\theta,A}>0$ and
$T_{\theta,A}\geq2$ such that, for each $t\geq T_{\theta,A}$, an
integer $j_t$ can be chosen with
\begin{equation}\label{eq:transverse-section-lower}
    |j_t-m_d(t)|\leq C_{\theta,A},\qquad
    \inf_{\substack{y\in\Z^{d-1}\\ |y|\leq A\sqrt t}}
            u(t,j_t,y)\geq\theta.
\end{equation}
The constants may depend on $d,f,u_0,\theta,A$, but not on $t$ or the
lattice site.
\end{theorem}

\begin{proof}[Proof of Theorem~\ref{thm:nonlinear-lower}]
Choose $T$, $y_0$ and the one-site initial datum for $q$ as in
Lemma~\ref{lem:product-subsolution}. Apply
Proposition~\ref{prop:scalar-logarithmic-lower} with
$\beta=(d-1)/2$, $K=C_f$ and $\mu=\lambda_*/2$.
Write $b:=\kappa_d/\lambda_*=b_\beta$, and denote the time and positive
constant in \eqref{eq:scalar-positive-seed} by $t_{\rm sc}$ and
$\eta_{\rm sc}$, respectively. For original times $s\geq T+1$, set
\[
 a_0(s):=m_d(s)-c_*T+1,\qquad J(s):=\lceil a_0(s)\rceil.
\]
Since $a_0(s)=c_*(s-T)-b\log s+1$, the integer $J(s)$ is exactly
the site in \eqref{eq:scalar-positive-seed} at scalar time $s-T$.
Thus $q(s-T,J(s))\geq\eta_{\rm sc}$ for $s\geq T+t_{\rm sc}$.
The rounding error satisfies
$J(s)-m_d(s)+c_*T=J(s)-a_0(s)+1\in[1,2)$.

The transverse factor in \eqref{eq:product-comparison} is
$H_T(s-T,y-y_0)=\beta_\perp s^{(d-1)/2}\pheat_{d-1}(s,y-y_0)$.
Apply Lemma~\ref{lem:fourier-expansion} in dimension $d-1$ with
$p=0$ and initial datum $\1_{\{0\}}$. There are constants
$C_{\perp}^{\rm err}>0$ and $s_\perp\geq1$, depending only on $d$,
such that, for every $s\geq s_\perp$ and $y\in\Z^{d-1}$,
\[
 \left|s^{(d-1)/2}\pheat_{d-1}(s,y-y_0)
  -\gamma_0\left(\frac{y-y_0}{\sqrt s}\right)\right|
 \leq C_{\perp}^{\rm err}s^{-1/2},\qquad
 \gamma_0(z):=(4\pi)^{-(d-1)/2}\ee^{-|z|^2/4}.
\]
For $A\geq0$, put
\[
 g_A:=(4\pi)^{-(d-1)/2}\ee^{-(A+1)^2/4},\qquad
 h_A:=\frac{\beta_\perp g_A}{2},\qquad
 \eta_A:=\min\{1/2,\eta_{\rm sc}h_A\},
\]
and choose the fixed threshold
\[
 s_A^{\rm seed}:=\max\left\{T+t_{\rm sc},s_\perp,|y_0|^2,
              \left(\frac{2C_{\perp}^{\rm err}}{g_A}\right)^2\right\}.
\]
If $s\geq s_A^{\rm seed}$ and $|y|\leq A\sqrt s$, then
$|(y-y_0)/\sqrt s|\leq A+1$, so the Gaussian term is at least $g_A$
and the error is at most $g_A/2$. Hence $H_T(s-T,y-y_0)\geq h_A$.
Using \eqref{eq:product-comparison} and the scalar lower bound gives
\begin{equation}\label{eq:diffusive-section-seed}
 \inf_{\substack{y\in\Z^{d-1}\\ |y|\leq A\sqrt s}}
       u(s,J(s),y)\geq\eta_A\qquad(s\geq s_A^{\rm seed}).
\end{equation}

We now prove \eqref{eq:axial-inner-lower}. Fix $\theta\in(0,1)$ and
let $S_{\rm ax}:=S(\eta_0,\theta,d,f)$ be the amplification time in
Lemma~\ref{lem:amplification}, using \eqref{eq:diffusive-section-seed}
with $A=0$. Choose
$s_0\geq\max\{s_0^{\rm seed},2b/c_*\}$ such that $a_0(s_0)>1$.
Then $a_0'(s)=c_*-b/s\geq c_*/2$ for $s\geq s_0$, and
$a_0(s)\to\infty$ as $s\to\infty$.
Put $j_{\rm start}:=\lceil a_0(s_0)\rceil$.
For every integer $j\geq j_{\rm start}$, there is a unique
$s_j\geq s_0$ with $a_0(s_j)=j$. At this time $J(s_j)=j$, so
\eqref{eq:diffusive-section-seed} gives $u(s_j,je_1)\geq\eta_0$.
Lemma~\ref{lem:amplification} gives $u(t,je_1)\geq\theta$ for every
$t\geq s_j+S_{\rm ax}$. Consequently, if $t-S_{\rm ax}\geq s_0$,
then every integer $j$ with
$j_{\rm start}\leq j\leq a_0(t-S_{\rm ax})$ satisfies
$s_j\leq t-S_{\rm ax}$ and hence $u(t,je_1)\geq\theta$.
Moreover,
\[
 m_d(t)-m_d(t-S_{\rm ax})
 =c_*S_{\rm ax}-b\log\frac{t}{t-S_{\rm ax}}
 \leq c_*S_{\rm ax},
\]
so $a_0(t-S_{\rm ax})\geq m_d(t)-c_*(T+S_{\rm ax})+1$.

For the remaining sites $0\leq j<j_{\rm start}$, strict positivity
from Lemma~\ref{lem:comparison} gives
\[
 \eta_{\rm fin}:=\min_{\substack{j\in\Z\\0\leq j<j_{\rm start}}}
                 u(s_0,je_1)>0.
\]
Let $S_{\rm fin}:=S(\min\{\eta_{\rm fin},1/2\},\theta,d,f)$.
Lemma~\ref{lem:amplification} gives $u(t,je_1)\geq\theta$ at every
one of these sites for all $t\geq s_0+S_{\rm fin}$.
Set $D_\theta:=c_*(T+S_{\rm ax})+1$, and choose
$T_\theta\geq\max\{2,s_0+S_{\rm ax},s_0+S_{\rm fin}\}$ with
$m_d(T_\theta)\geq D_\theta+1$.
Since $m_d'\geq c_*/2$ on $[s_0,\infty)$, one has
$m_d(t)\geq D_\theta+1$ for every $t\geq T_\theta$.
For these times, the ranges $0\leq j<j_{\rm start}$ and
$j_{\rm start}\leq j\leq a_0(t-S_{\rm ax})$ cover every
integer $0\leq j\leq m_d(t)-D_\theta$.
At the nonnegative integer $j=\lfloor m_d(t)-D_\theta\rfloor$, we obtain
$R_\theta(t)\geq m_d(t)-D_\theta-1$.
Thus $C_\theta^-:=D_\theta+1$ gives both
\eqref{eq:axial-inner-lower} and \eqref{eq:sharp-axial-lower}.

To prove \eqref{eq:transverse-section-lower}, fix $A\geq0$, put
$A':=2A+1$, and use the constants $\eta_{A'}$ and $s_{A'}^{\rm seed}$
in \eqref{eq:diffusive-section-seed}. Set
$S_A:=S(\eta_{A'},\theta,d,f)$ and choose
\[
 T_{\theta,A}:=\max\{2,2S_A,S_A+s_{A'}^{\rm seed},S_A+2b/c_*\},
 \qquad C_{\theta,A}:=c_*(T+S_A)+2.
\]
For $t\geq T_{\theta,A}$, put $s=t-S_A$ and $j_t=J(s)$.
Then $s\geq s_{A'}^{\rm seed}$, $s\geq t/2$, and
$A\sqrt t\leq\sqrt2 A\sqrt s<A'\sqrt s$.
Thus every $y\in\Z^{d-1}$ with $|y|\leq A\sqrt t$ satisfies
$u(s,j_t,y)\geq\eta_{A'}$ by \eqref{eq:diffusive-section-seed}.
Lemma~\ref{lem:amplification}, applied at each site $(j_t,y)$ with
the same time $S_A$, gives $u(t,j_t,y)\geq\theta$ for all these $y$.
Finally, $s\geq2b/c_*$ implies
$0\leq m_d(t)-m_d(s)\leq c_*S_A$.
Together with $J(s)-m_d(s)+c_*T\in[1,2)$, this yields
$|j_t-m_d(t)|\leq c_*(T+S_A)+2=C_{\theta,A}$.
\end{proof}

\begin{corollary}\label{cor:nonempty}
Under \eqref{eq:kpp}--\eqref{eq:u0}, for every $\theta\in(0,1)$ the
set in \eqref{eq:level-definition} is nonempty and bounded above for all
sufficiently large $t$.
\end{corollary}

\begin{proof}
Lemma~\ref{lem:comparison} gives $\eta:=u(1,0)\in(0,1)$.
Lemma~\ref{lem:amplification}, with $t_0=1$ and $z=0$, gives
$u(t,0)\geq\theta$ for $t\geq1+S(\eta,\theta,d,f)$.
Let $C_{\rm ray}$ denote the constant $C_\rho$ in
Theorem~\ref{thm:nonlinear-upper} with $\rho=0$.
For $e=e_1$, $r=j\geq0$ and $x=je_1$, \eqref{eq:directional-upper}
gives
\[
 u(t,je_1)\leq C_{\rm ray}t^{-d/2}
    \ee^{-\lambda_*(j-c_*t)}\qquad(t\geq1,\ j\in\Z_{\geq0}).
\]
For each fixed $t\geq1$, this bound tends to zero as $j\to+\infty$.
Hence the level set in \eqref{eq:level-definition} is bounded above
and is nonempty once $t\geq1+S(\eta,\theta,d,f)$.
\end{proof}

\begin{proof}[Proof of Theorem~\ref{thm:nonlinear-target}]
Fix $\theta\in(0,1)$. For $e=e_1$, Lemma~\ref{lem:critical-pair}
gives $z_{e_1}(t,(j,y))=\lambda_*(j-m_d(t))$.
With $C_{\rm up}$ from \eqref{eq:sharp-directional-envelope}, choose
$L_\theta:=2\lambda_*^{-1}\log(4C_{\rm up}/\theta)>0$.
For $r\geq L_\theta$, the inequality $1+\lambda_*r\leq2\ee^{\lambda_*r/2}$
gives
$C_{\rm up}(1+\lambda_*r)\ee^{-\lambda_*r}
 \leq2C_{\rm up}\ee^{-\lambda_*r/2}\leq\theta/2$.
Equation~\eqref{eq:sharp-directional-envelope} gives
$u(t,j,y)\leq\theta/2$ uniformly in $j\geq m_d(t)+L_\theta$ and
$y\in\Z^{d-1}$ for $t\geq4$. Hence
\begin{equation}\label{eq:remaining-sharp-upper}
 \sup_{\substack{j\in\Z,\ y\in\Z^{d-1}\\
                  j\geq m_d(t)+L_\theta}}u(t,j,y)<\theta
       \qquad(t\geq4).
\end{equation}
In particular, $R_\theta(t)\leq m_d(t)+L_\theta+1$ whenever the
level set is nonempty and $t\geq4$.
Denote the time threshold in Theorem~\ref{thm:nonlinear-lower} by
$T_\theta^{\rm low}$. For $t\geq T_\theta^{\rm low}$,
\eqref{eq:sharp-axial-lower} gives
$R_\theta(t)\geq m_d(t)-C_\theta^->-\infty$, so the level set is nonempty.
Set $C_\theta:=\max\{C_\theta^-,L_\theta+1\}$ and
$T_\theta:=\max\{4,T_\theta^{\rm low}\}$. Then
\begin{equation}\label{eq:two-sided-current-bounds}
 c_*t-\frac{d+2}{2\lambda_*}\log t-C_\theta
 \leq R_\theta(t)\leq
 c_*t-\frac{d+2}{2\lambda_*}\log t+C_\theta
 \qquad(t\geq T_\theta),
\end{equation}
which is \eqref{eq:bramson-target}.

For $\varepsilon\in(0,1/2)$, apply \eqref{eq:axial-inner-lower} at
level $1-\varepsilon$ and \eqref{eq:remaining-sharp-upper} at level
$\varepsilon$. Put
$C_\varepsilon^{\rm tr}:=\max\{C_{1-\varepsilon}^-,L_\varepsilon\}$.
Choose $T_\varepsilon\geq\max\{4,T_{1-\varepsilon}^{\rm low},
2\kappa_d/(\lambda_*c_*)\}$ with
$m_d(T_\varepsilon)\geq C_\varepsilon^{\rm tr}$.
Since $m_d$ is increasing for $t\geq T_\varepsilon$, the interval
$0\leq j\leq m_d(t)-C_\varepsilon^{\rm tr}$ contains $j=0$.
The lower bound applies throughout this interval, and the uniform
upper bound applies for $j\geq m_d(t)+C_\varepsilon^{\rm tr}$.
These are the two final inequalities in
Theorem~\ref{thm:nonlinear-target}.
\end{proof}

\section{Convergence to the critical wave}\label{sec:profiles}

An entire solution $V$ of \eqref{eq:main} is defined on
$\R\times\Z^d$, with $t\mapsto V(t,x)$ continuously
differentiable and
$V_t(t,x)=\Delta_{\Z^d}V(t,x)+f(V(t,x))$ for every
$(t,x)\in\R\times\Z^d$. Local $C^1$ convergence means uniform
convergence of the functions and their time derivatives on bounded
time intervals and finite sets of lattice sites.
Let $U:I\times\Z^d\to[0,1]$ solve \eqref{eq:main} on an open time
interval $I$. The nearest-neighbor formula gives
$|\partial_tU(t,x)|\leq2d+\|f\|_{L^\infty([0,1])}$ at every site.
Integration in time makes $U$ locally Lipschitz as an
$\ell^\infty(\Z^d)$-valued function. Since the vector field
$\mathcal F$ in Lemma~\ref{lem:comparison} is Lipschitz,
$\mathcal F(U(t))$ is norm-continuous. The coordinatewise integral
equation therefore gives $U\in C^1(I;\ell^\infty(\Z^d))$.
Differentiating the equation, using $f\in C^2([0,1])$, yields
$U_{tt}=\Delta_{\Z^d}U_t+f'(U)U_t$ and
\begin{equation}\label{eq:profile-time-bounds}
 \|U_t(t)\|_\infty\leq C_{{\rm t},1}:=2d+\|f\|_{L^\infty([0,1])},
 \qquad
 \|U_{tt}(t)\|_\infty\leq C_{{\rm t},2}:=(4d+L_f)C_{{\rm t},1}.
\end{equation}
These bounds apply to $u$ for $t>0$ and to entire solutions with
values in $[0,1]$.

For $u_t=\Delta u+f(u)$ in $\R^d$,
\cite[Theorem~3.5]{BerestyckiHamel} shows that an entire front trapped
between two translates of one planar wave is itself a translate.
Lemma~\ref{lem:planar-lattice-rigidity} proves the lattice conclusion
from \eqref{eq:entire-rear-state} and \eqref{eq:entire-critical-tail}.
Its proof slides continuous time shifts and integer spatial shifts,
and rescales by the critical tail when the contact points tend to infinity.

\begin{lemma}\label{lem:entire-positive-comparison}
Let $V,W:\R\times\Z^d\to[0,1]$ be entire solutions of
\eqref{eq:main}. There are $\delta\in(0,1)$ and $\nu>0$, depending
only on $f$, such that, for every set $E\subset\R\times\Z^d$ with
$E^c:=(\R\times\Z^d)\setminus E$,
\[
 V\geq1-\delta\ \text{on }E,\qquad W\leq V\ \text{on }E^c
 \quad\Longrightarrow\quad W\leq V\ \text{on }\R\times\Z^d.
\]
Moreover, if $V-W\geq0$ on $\R\times\Z^d$ and
$(V-W)(t_0,x_0)=0$ at one point, then $V\equiv W$.
\end{lemma}

\begin{proof}
Choose $\delta$ and $\nu$ so that $f'\leq-\nu$ on $[1-\delta,1]$.
Set $D=(W-V)_+$. At each site, $D$ is locally Lipschitz in time,
with Lipschitz constant at most $2C_{{\rm t},1}$ by
\eqref{eq:profile-time-bounds}.
If $D(t,x)>0$, then $W(t,x)>V(t,x)$; the
assumption $W\leq V$ outside $E$ forces $(t,x)\in E$, so
$1-\delta\leq V(t,x)<W(t,x)\leq1$.  The mean-value theorem and
$f'\leq-\nu$ on $[1-\delta,1]$ give
$f(W)-f(V)\leq-\nu(W-V)$.  Moreover
$D(t,z)\geq W(t,z)-V(t,z)$ at every neighbor $z\sim x$, so
$\Delta D(t,x)\geq\Delta(W-V)(t,x)$.  Hence
$D_t\leq\Delta D-\nu D$ at such points. If $D(t,x)=0$ at a time
where $D(\cdot,x)$ is differentiable, nonnegativity makes that time a
minimum and gives $D_t(t,x)=0$. The exceptional times have measure
zero at each site. Also $\Delta D(t,x)=\sum_{z\sim x}D(t,z)\geq0$, so
$D_t\leq\Delta_{\Z^d}D-\nu D$ also holds for almost every time at
points where $D=0$. Thus
\[
 D_t\leq\Delta_{\Z^d}D-\nu D
\]
at every site for almost every time. Fix $t\in\R$ and $h>0$.
On $[t-h,t]$, compare $D$ with
the solution of $z_s=\Delta_{\Z^d}z-\nu z$ starting from $D(t-h)$.
Lemma~\ref{lem:comparison} on $\Z^d$, with $F(s,r)=-\nu r$, gives
\[
 0\leq D(t,x)\leq\ee^{-\nu h}
       (\ee^{h\Delta_{\Z^d}}D(t-h,\cdot))(x).
\]
Since $0\leq D\leq1$ and the heat kernel has total mass one, the
right-hand side is at most $\ee^{-\nu h}$.  Letting $h\to\infty$
gives $D(t,x)=0$ for every $(t,x)$, i.e. $W\leq V$.

To prove $V\equiv W$ when $V\geq W$ and
$V(t_0,x_0)=W(t_0,x_0)$, set $D=V-W$ and define
\[
 b_{V,W}(t,x):=\int_0^1
 f'\bigl(W(t,x)+\rho(V(t,x)-W(t,x))\bigr)\,d\rho.
\]
Then $|b_{V,W}|\leq L_f$ and
$D_t=\Delta_{\Z^d}D+b_{V,W}D$. Since $D\geq0$ and $b_{V,W}\geq-L_f$,
$D_t\geq\Delta_{\Z^d}D-L_fD$ on $\R\times\Z^d$.

For $h>0$, compare $D$ on $[t_0-h,t_0]$ with
$z(s)=\ee^{-L_f(s-t_0+h)}
\ee^{(s-t_0+h)\Delta_{\Z^d}}D(t_0-h)$, which solves
$z_s=\Delta_{\Z^d}z-L_fz$ and satisfies $z(t_0-h)=D(t_0-h)$.
Lemma~\ref{lem:comparison}, with $F(t,s)=-L_fs$, gives
\[
  0=D(t_0,x_0)\geq
  \ee^{-L_f h}
  \sum_{x\in\Z^d}\pheat_d(h,x_0-x)D(t_0-h,x).
\]
Every term in the sum is nonnegative and
$\pheat_d(h,x_0-x)>0$ for all $x$, so
$D(t_0-h,x)=0$ for every $x$.  Since this holds for every $h>0$,
$V=W$ for $t\leq t_0$. For any $h>0$, apply
Lemma~\ref{lem:comparison} to $V,W$ in both orders on $[t_0,t_0+h]$,
with spatial domain $\Z^d$ and reaction $\widetilde f$. Their initial
data agree at $t_0$, so $V=W$ also for $t\geq t_0$.
\end{proof}

\begin{lemma}\label{lem:planar-lattice-rigidity}
Let $V:\R\times\Z^d\to(0,1)$ be an entire solution of
\eqref{eq:main}, and put $r=j-c_*t$ for $x=(j,y)$. Suppose that
\begin{equation}\label{eq:entire-rear-state}
 \lim_{R\to\infty}\inf_{\substack{t\in\R,\ (j,y)\in\Z^d\\
                                  j-c_*t\leq-R}}V(t,j,y)=1,
\end{equation}
and that there are $R_0\geq1$ and $0<b_0\leq B_0<\infty$ such that
\begin{equation}\label{eq:entire-critical-tail}
 b_0(1+r)\ee^{-\lambda_*r}\leq V(t,j,y)
          \leq B_0(1+r)\ee^{-\lambda_*r}
 \qquad(r\geq R_0).
\end{equation}
All bounds are uniform in $t$ and $y$. Then, for some $\xi\in\R$,
\begin{equation}\label{eq:entire-planar-wave}
 V(t,j,y)=\phi_*(j-c_*t+\xi)
 \qquad(t\in\R,\ (j,y)\in\Z^d).
\end{equation}
\end{lemma}

\begin{proof}
Fix $A,B>0$ and choose $h>0$ with $-A+c_*h\geq R_0$.  At a fixed
site $(j,y)$,
\[
 \partial_tV(t,j,y)
 =\sum_{z\sim(j,y)}V(t,z)-2dV(t,j,y)+f(V(t,j,y))
 \geq-2dV(t,j,y),
\]
so integration from $t-h$ to $t$ gives
$V(t,j,y)\geq\ee^{-2dh}V(t-h,j,y)$.  If the current front coordinate
$r=j-c_*t$ lies in $[-A,B]$, then at time $t-h$ it is
$r+c_*h\in[-A+c_*h,B+c_*h]\subset[R_0,B+c_*h]$.  Therefore
\eqref{eq:entire-critical-tail} gives
\[
 V(t,j,y)\geq \ee^{-2dh}b_0
 \min_{\rho\in[-A+c_*h,B+c_*h]}
      (1+\rho)\ee^{-\lambda_*\rho}>0.
\]
For the fixed solution $V$, this is a positive constant independent
of $t,j,y$ in the strip $-A\leq j-c_*t\leq B$.

Fix an arbitrary integer shift $z=(q,\eta)\in\Z\times\Z^{d-1}$.
For $\tau\in\R$ define
\[
 W_\tau(t,j,y):=V(t+\tau,j+q,y+\eta),
 \qquad \sigma(\tau):=q-c_*\tau.
\]
We will prove that $W_{q/c_*}=V$. By \eqref{eq:entire-rear-state}, choose $A>0$ so large that
\[
 \inf_{\substack{t\in\R,(j,y)\in\Z^d\\j-c_*t\leq-A}}
 V(t,j,y)\geq1-\delta,
\]
where $\delta$ is the constant from
Lemma~\ref{lem:entire-positive-comparison}.
Let
$m_{A,R_0}:=\inf\{V(t,j,y):-A\leq j-c_*t\leq R_0\}>0$.
Choose $\Sigma_{\rm in}\geq A+R_0$ so large that, for all
$\sigma\geq\Sigma_{\rm in}$,
\[
 B_0(1+R_0+\sigma)\ee^{-\lambda_*(\sigma-A)}\leq m_{A,R_0},
 \qquad
 \frac{B_0}{b_0}(1+\sigma)\ee^{-\lambda_*\sigma}\leq1.
\]
Both left-hand sides tend to zero as $\sigma\to\infty$.
If $\tau\leq(q-\Sigma_{\rm in})/c_*$, then
$\sigma(\tau)\geq\Sigma_{\rm in}$. For $-A\leq r\leq R_0$,
the front coordinate of $W_\tau$ satisfies $r+\sigma(\tau)\geq R_0$.
The upper bound in \eqref{eq:entire-critical-tail} therefore gives
\[
 W_\tau(t,j,y)
 \leq B_0(1+R_0+\sigma(\tau))
             \ee^{-\lambda_*(\sigma(\tau)-A)}
 \leq m_{A,R_0}\leq V(t,j,y).
\]
For $r\geq R_0$, the two tail bounds give
\[
 \frac{W_\tau(t,j,y)}{V(t,j,y)}
 \leq\frac{B_0}{b_0}
       \frac{1+r+\sigma(\tau)}{1+r}
       \ee^{-\lambda_*\sigma(\tau)}
 \leq\frac{B_0}{b_0}(1+\sigma(\tau))
       \ee^{-\lambda_*\sigma(\tau)}\leq1.
\]

Set $E=\{(t,j,y):j-c_*t<-A\}$. We have $V\geq1-\delta$ on $E$ and,
for every $\tau\leq(q-\Sigma_{\rm in})/c_*$, $W_\tau\leq V$ on $E^c$.
Lemma~\ref{lem:entire-positive-comparison} gives $W_\tau\leq V$
on $\R\times\Z^d$ for all these shifts.

Set
\[
 \tau_*:=\sup\{\tau\in\R:\ W_{\tau'}\leq V
                 \text{ everywhere for every }\tau'\leq\tau\}.
\]
The number $\tau_*$ is finite.  Indeed, for fixed $(t,j,y)$ the front
coordinate of $W_\tau(t,j,y)$ is
$j+q-c_*(t+\tau)=r+\sigma(\tau)\to-\infty$ as $\tau\to\infty$;
\eqref{eq:entire-rear-state} gives $W_\tau(t,j,y)\to1$, while
$V(t,j,y)<1$.  To obtain the order at the endpoint, take
$\tau_m\uparrow\tau_*$ with $W_{\tau_m}\leq V$.  Continuity of
$t\mapsto V(t,j,y)$ at every site gives
$W_{\tau_*}(t,j,y)=\lim_mW_{\tau_m}(t,j,y)\leq V(t,j,y)$ for every
$(t,j,y)$. The definition of $\tau_*$ also gives $W_\tau\leq V$
for every $\tau<\tau_*$, since some admissible $\tau_m$ exceeds $\tau$.
Suppose that $\sigma_*:=q-c_*\tau_*>0$. There are
$\tau_n\in(\tau_*,\tau_*+1/n)$ and points $(t_n,j_n,y_n)$ such that
\begin{equation}\label{eq:sliding-failure}
 r_n:=j_n-c_*t_n\geq-A,
 \qquad W_{\tau_n}(t_n,j_n,y_n)>V(t_n,j_n,y_n).
\end{equation}
Indeed, if no such sequence existed, then $W_\tau\leq V$ would persist
on $r\geq-A$ for $\tau$ in a right neighborhood of $\tau_*$. Applying
Lemma~\ref{lem:entire-positive-comparison} with this set $E$ would contradict
the definition of $\tau_*$. We consider
the two possible behaviors of $r_n$.

If a subsequence of $r_n$ is bounded, take it so that $r_n\to r_\infty$.
Define
\[
 V_n(s,k,w):=V(t_n+s,j_n+k,y_n+w).
\]
Bounds \eqref{eq:profile-time-bounds} apply to $V_n$ and its time
derivatives. On each cylinder
$[-m,m]\times\{(k,w):|k|+|w|\leq m\}$, $V_n$ and $(V_n)_s$ are
uniformly bounded and equicontinuous. Apply Arzel\`a--Ascoli to both
families, take nested subsequences over $m\in\N$, and choose a
diagonal subsequence. The integral identity in time identifies the
limit of $(V_n)_s$ as the derivative of the limit $V_\infty$.
The convergence is thus locally $C^1$ in time on finite sets of sites.
The equation passes to the limit because each Laplacian uses only
$2d$ neighbors and $f$ is continuous, so $V_\infty$ is an entire
solution with values in $[0,1]$.

For fixed $(s,k,w)$ the original front coordinate is
$r_n+k-c_*s\to r_\infty+k-c_*s$. If the limit exceeds $R_0$,
the upper bound in \eqref{eq:entire-critical-tail} passes to
$V_\infty$ with that shifted coordinate. For the rear state, put
\[
 \omega(R):=\inf_{\substack{t\in\R,(j,y)\in\Z^d\\j-c_*t\leq-R}}
                    V(t,j,y),\qquad R>0.
\]
If $r_\infty+k-c_*s\leq-R-1$, then
$r_n+k-c_*s\leq-R$ for all large $n$, so
$V_\infty(s,k,w)\geq\omega(R)$.
By \eqref{eq:entire-rear-state}, $\omega(R)\to1$ as $R\to\infty$.

Also,
\[
 D_\infty(s,k,w):=V_\infty(s,k,w)
              -V_\infty(s+\tau_*,k+q,w+\eta)\geq0,
 \qquad D_\infty(0,0,0)=0.
\]
From $W_{\tau_*}\leq V$ we have
$D_\infty(0,0,0)\geq0$.  On the other hand,
\eqref{eq:sliding-failure} gives
$V(t_n+\tau_n,j_n+q,y_n+\eta)>V(t_n,j_n,y_n)$.
The uniform bound on $V_t$ implies
\[
  0\leq V(t_n,j_n,y_n)-V(t_n+\tau_*,j_n+q,y_n+\eta)
     \leq C_{{\rm t},1}|\tau_n-\tau_*|,
\]
so $D_\infty(0,0,0)=0$ after passing to the limit.  The second conclusion
of Lemma~\ref{lem:entire-positive-comparison} then gives
$D_\infty\equiv0$.
Iteration gives
$V_\infty(s,k,w)=V_\infty(s+\ell\tau_*,k+\ell q,w+\ell\eta)$
for every integer $\ell$. At the point on the right the front coordinate is
$r_\infty+k-c_*s+\ell\sigma_*$.  Since $\sigma_*>0$, the upper tail
implies
$V_\infty(s+\ell\tau_*,k+\ell q,w+\ell\eta)\to0$ as
$\ell\to+\infty$, whereas the rear-state limit implies convergence to
$1$ as $\ell\to-\infty$.  The invariance says both sequences are
identically equal to the fixed number $V_\infty(s,k,w)$, which is
impossible.

It remains to exclude $r_n\to\infty$. Discard finitely many terms
so that $r_n\geq R_0$, and define
\[
 a_n:=(1+r_n)\ee^{-\lambda_*r_n},\qquad
 Z_n(s,k,w):=a_n^{-1}V(t_n+s,j_n+k,y_n+w).
\]
For $L\geq0$, let $K_L:=\{(k,w)\in\Z^d:|k|+|w|\leq L\}$.
Fix $S>0$ and $L\geq0$. For all large $n$, the front coordinates
$r_n+k-c_*s$ exceed $R_0$ on $[-S,S]\times K_{L+2}$, and
\[
 \frac12\leq\frac{1+r_n+k-c_*s}{1+r_n}\leq2.
\]
Dividing \eqref{eq:entire-critical-tail} by $a_n$ therefore gives
\[
 \frac{b_0}{2}\ee^{-\lambda_*(L+2+c_*S)}
 \leq Z_n(s,k,w)\leq
 M_{S,L}:=2B_0\ee^{\lambda_*(L+2+c_*S)}
 \quad\text{on }[-S,S]\times K_{L+2}.
\]
Since $Z_n$ satisfies
$(Z_n)_s=\Delta_{\Z^d}Z_n+f(a_nZ_n)/a_n$ and
$|f(a_nZ_n)|/a_n\leq L_f|Z_n|$, one has
\[
 |(Z_n)_s|\leq(4d+L_f)M_{S,L}
 \quad\text{on }[-S,S]\times K_{L+1}.
\]
Differentiating the equation gives
$(Z_n)_{ss}=\Delta_{\Z^d}(Z_n)_s+f'(a_nZ_n)(Z_n)_s$, hence
$|(Z_n)_{ss}|\leq(4d+L_f)^2M_{S,L}$ on $[-S,S]\times K_L$.
These bounds do not depend on $n$. Arzel\`a--Ascoli applied to
$Z_n,(Z_n)_s$ on the cylinders $[-m,m]\times K_m$, followed by
diagonal extraction and integration in time, gives a locally
$C^1$ limit $Z$. Moreover, Taylor's formula at zero gives
\[
 \left|\frac{f(a_nZ_n)}{a_n}-aZ_n\right|
 \leq\frac12\|f''\|_{L^\infty([0,1])}a_nM_{S,L}^2\to0
 \quad\text{on }[-S,S]\times K_L.
\]
Thus $Z_s=\Delta_{\Z^d}Z+aZ$ on $\R\times\Z^d$.
For each fixed $(s,k,w)$,
$(1+r_n+k-c_*s)/(1+r_n)\to1$; passing to the limit in the
normalized tail bounds gives
\begin{equation}\label{eq:normalized-linear-tail}
 b_0\ee^{-\lambda_*(k-c_*s)}\leq Z(s,k,w)
 \leq B_0\ee^{-\lambda_*(k-c_*s)}.
\end{equation}
In particular, $Z>0$.

The order $W_{\tau_*}\leq V$, divided by $a_n$ and passed to the
limit, gives $Z(s,k,w)\geq Z(s+\tau_*,k+q,w+\eta)$.
At the origin, \eqref{eq:sliding-failure} and the endpoint order give
$Z_n(\tau_n,q,\eta)>Z_n(0,0,0)\geq Z_n(\tau_*,q,\eta)$.
Choose $S_{\rm ct}:=|\tau_*|+1$ and $L_{\rm ct}:=|q|+|\eta|$.
For all large $n$, both $\tau_n$ and $\tau_*$ lie in
$[-S_{\rm ct},S_{\rm ct}]$, and the derivative bound above applies
at $(q,\eta)$ with the fixed constant
$C_{\rm ct}:=(4d+L_f)M_{S_{\rm ct},L_{\rm ct}}$. Hence
\[
 0\leq Z_n(0,0,0)-Z_n(\tau_*,q,\eta)
 \leq C_{\rm ct}|\tau_n-\tau_*|\to0.
\]
Consequently,
$D(s,k,w):=Z(s,k,w)-Z(s+\tau_*,k+q,w+\eta)$ satisfies
$D\geq0$, $D(0,0,0)=0$, and $D_s=\Delta_{\Z^d}D+aD$.

We justify the evolution formula for $D$ in the weighted space
\[
 X_{\lambda_*}:=\left\{h:\Z^d\to\R:
           \|h\|_{X_{\lambda_*}}:=
           \sup_{(k,w)}\ee^{\lambda_*k}|h(k,w)|<\infty\right\}.
\]
Multiplication by $\ee^{\lambda_*k}$ is an isometry from this space
onto $\ell^\infty(\Z^d)$, so $X_{\lambda_*}$ is a Banach space.
For $(S_\pm h)(k,w)=h(k\pm1,w)$, one has
$\|S_\pm h\|_{X_{\lambda_*}}=\ee^{\mp\lambda_*}\|h\|_{X_{\lambda_*}}$;
transverse unit shifts preserve the norm. Therefore
$A_*:=\Delta_{\Z^d}+a$ is bounded on $X_{\lambda_*}$, with
\[
 \|A_*\|\leq2\cosh\lambda_*+2(d-1)+|a-2d|=:C_{\rm op}.
\]
Because $0\leq D\leq Z$, \eqref{eq:normalized-linear-tail} gives
$\|D(s)\|_{X_{\lambda_*}}\leq B_0\ee^{\lambda_*c_*s}$.
The coordinatewise equation implies
$\|D_s(s)\|_{X_{\lambda_*}}\leq C_{\rm op}B_0\ee^{\lambda_*c_*s}$.
Integrating at each site and taking the weighted supremum yields
\[
 \|D(t)-D(s)\|_{X_{\lambda_*}}
 \leq C_{\rm op}B_0\ee^{\lambda_*c_*S}|t-s|
 \qquad(s,t\in[-S,S]).
\]
Thus $D$ is norm-continuous, so $A_*D$ is norm-continuous as well.
The integral equation then holds in $X_{\lambda_*}$, and the
bounded-operator ODE gives $D(0)=\ee^{hA_*}D(-h)$ for $h>0$.

The shift-operator expansion defining the heat kernel in
\eqref{eq:heat-semigroup} also converges in operator norm on
$X_{\lambda_*}$. Its convolution is absolutely convergent: by
\eqref{eq:kernel-tilt} with $p=\lambda_*e_1$ and $\sum_xK_p(h,x)=1$,
\[
 \begin{split}
 \sum_{x\in\Z^d}\pheat_d(h,-x)|D(-h,x)|
 &\leq\|D(-h)\|_{X_{\lambda_*}}
       \sum_{x\in\Z^d}\pheat_d(h,-x)\ee^{-\lambda_*x_1}\\
 &=\|D(-h)\|_{X_{\lambda_*}}
       \ee^{2h(\cosh\lambda_*-1)}<\infty.
 \end{split}
\]
Evaluating $D(0)=\ee^{hA_*}D(-h)$ at the origin now gives
\[
 0=D(0,0,0)=\ee^{ah}
   \sum_{x\in\Z^d}\pheat_d(h,-x)D(-h,x).
\]
The summands are nonnegative and the kernel is strictly positive,
so $D(-h,x)=0$ for every $x$ and $h>0$. For $s\geq0$, the same
ODE gives $D(s)=\ee^{sA_*}D(0)=0$.
Hence $Z(s,k,w)=Z(s+\ell\tau_*,k+\ell q,w+\ell\eta)$ for every
integer $\ell$. But \eqref{eq:normalized-linear-tail} gives
\[
 0<Z(0,0,0)=Z(\ell\tau_*,\ell q,\ell\eta)
 \leq B_0\ee^{-\lambda_*\ell\sigma_*}\to0
 \qquad(\ell\to\infty),
\]
a contradiction.

The contradictions in the bounded and unbounded cases show that the
assumption $q-c_*\tau_*>0$ is impossible; hence
$\tau_*\geq q/c_*$ and $W_{q/c_*}\leq V$.  Since the integer shift $z=(q,\eta)$ was arbitrary, apply
$V(t+q/c_*,j+q,y+\eta)\leq V(t,j,y)$ with
$(q,\eta)$ replaced by $(-q,-\eta)$. This gives
$V(t-q/c_*,j-q,y-\eta)\leq V(t,j,y)$.  Replacing
$(t,j,y)$ by $(t+q/c_*,j+q,y+\eta)$ gives
$V(t,j,y)\leq V(t+q/c_*,j+q,y+\eta)$, the reverse inequality.  Therefore
\begin{equation}\label{eq:lattice-space-time-invariance}
 V(t+q/c_*,j+q,y+\eta)=V(t,j,y)
 \quad\text{for every }q\in\Z,\ \eta\in\Z^{d-1}.
\end{equation}
Taking $q=0$ proves independence of $y$. Applying
\eqref{eq:lattice-space-time-invariance} with the shifts $q=-j$ and
$\eta=-y$ gives
$V(t,j,y)=V(t-j/c_*,0,0)=\psi(j-c_*t)$, where
$\psi(r):=V(-r/c_*,0,0)$. The time regularity of $V$ gives
$\psi\in C^1(\R)$.
For the shift $(q,\eta)=(0,0)$, the sliding argument gives
$V(t+\tau,j,y)\leq V(t,j,y)$ for every $\tau\leq0$.
Taking $t=-r/c_*$ and $\tau=-\rho/c_*$ gives
$\psi(r+\rho)\leq\psi(r)$ for every $\rho>0$.
If equality held for some $r,\rho$, the ordered entire solutions
$V(t,j,y)$ and $V(t-\rho/c_*,j,y)$ would touch at
$(-r/c_*,0,0)$. Lemma~\ref{lem:entire-positive-comparison} would make
them identical, so $\psi$ would be $\rho$-periodic. This contradicts
$\psi(-\infty)=1$ and $\psi(+\infty)=0$, which follow from
\eqref{eq:entire-rear-state} and \eqref{eq:entire-critical-tail}.
Thus $\psi$ is strictly decreasing. Substitution in \eqref{eq:main}
shows that it solves \eqref{eq:critical-wave}.
The uniqueness up to translation for monostable lattice traveling waves
proved in \cite{ChenFuGuo} gives
$\psi=\phi_*(\cdot+\xi)$ for some $\xi\in\R$.

\end{proof}

\begin{lemma}\label{lem:axial-entire-limits}
Let $t_n\to\infty$, put $\ell_n=\lfloor m_d(t_n)\rfloor$, and
let $y_n\in\Z^{d-1}$ with $|y_n|\leq A\sqrt{t_n}$ for some fixed
$A\geq0$. Every sequence
\begin{equation}\label{eq:axial-translate-sequence}
 u_n(s,k,w):=u(t_n+s,\ell_n+k,y_n+w)
\end{equation}
has a subsequence converging in $C^1$ locally in time on finite sets of
sites to $\phi_*(k-c_*s+\xi)$ for some $\xi\in\R$.
The possible values of $\xi$ lie in a bounded interval depending only
on $d,f,u_0,A$.
\end{lemma}

\begin{proof}
The bounds \eqref{eq:profile-time-bounds} apply to $u_n$ on $s>-t_n$.
For each $m\in\N$, this domain contains $[-m,m]$ once $n$ is large.
Arzel\`a--Ascoli applied to $u_n,(u_n)_s$ on
$[-m,m]\times\{(k,w):|k|+|w|\leq m\}$, followed by diagonal
extraction, gives local $C^1$ convergence to a function $V$ with
$0\leq V\leq1$. Integration in time identifies the limiting
derivative, and the finitely many neighbors in each Laplacian allow
passage through \eqref{eq:main}. Thus $V$ is an entire solution.
Take the subsequence also so that
\[
 \delta_n:=m_d(t_n)-\ell_n\to\delta\in[0,1],
 \qquad y_n/\sqrt{t_n}\to\eta\in\R^{d-1},\quad |\eta|\leq A.
\]
The interval $[0,1]$ includes the possible limit $\delta=1$, although
each $\delta_n<1$. For fixed $s,k$,
\begin{equation}\label{eq:front-coordinate-limit}
 \ell_n+k-m_d(t_n+s)
 =k-c_*s-\delta_n+
       \frac{\kappa_d}{\lambda_*}\log(1+s/t_n)
 \to k-c_*s-\delta.
\end{equation}
This convergence is uniform for $s$ in bounded intervals: if
$|s|\leq S$ and $t_n\geq2S$, then
$|\log(1+s/t_n)|\leq2S/t_n$.

Write $r=k-c_*s$. At the translated point the variable in
\eqref{eq:sharp-directional-envelope} is
$z_{e_1}=\lambda_*(\ell_n+k-m_d(t_n+s))$.
For $r\geq4$ and all large $n$, \eqref{eq:front-coordinate-limit}
puts its factor in parentheses in $[r-2,r+2]$, and $t_n+s\geq4$.
Thus \eqref{eq:sharp-directional-envelope} gives
\[
 u_n(s,k,w)\leq C_{\rm up}(1+\lambda_*(r+2))
                     \ee^{-\lambda_*(r-2)}
 \leq B_0(1+r)\ee^{-\lambda_*r},
\]
where $B_0:=C_{\rm up}(1+2\lambda_*)\ee^{2\lambda_*}$.
Passing to the limit gives the upper bound in
\eqref{eq:entire-critical-tail}, independently of $\delta,\eta,s,k,w$.

For the lower bound, choose $T,y_0,q$ as in
Lemma~\ref{lem:product-subsolution}. Apply
Proposition~\ref{prop:scalar-logarithmic-lower} with
$\beta=(d-1)/2$, $K=C_f$ and $\mu=\lambda_*/2$.
Denote its fixed pulse amplitude by $\varepsilon_{\rm p}>0$ and
write $b:=\kappa_d/\lambda_*$.
At scalar time $t=t_n+s-T$, the argument of the pulse is
\[
 \ell_n+k-c_*(t_n+s-T)+b\log(t_n+s)
 \to k-c_*s-\delta+c_*T.
\]
For every fixed $(s,k)$ this argument is bounded, while
$(t+T)^{1/4}=(t_n+s)^{1/4}\to\infty$ and $t\to\infty$.
Hence both the time and spatial restrictions in
\eqref{eq:scalar-profile-lower} hold for all large $n$.
For fixed $(s,w)$,
$(y_n+w-y_0)/\sqrt{t_n+s}\to\eta$.
Lemma~\ref{lem:fourier-expansion}, in dimension $d-1$ with $p=0$
and initial datum $\1_{\{0\}}$, gives
\[
 H_T(t_n+s-T,y_n+w-y_0)\to
 h_\eta:=\beta_\perp(4\pi)^{-(d-1)/2}\ee^{-|\eta|^2/4}.
\]
In particular,
$h_\eta\geq h_{\min,A}:=\beta_\perp(4\pi)^{-(d-1)/2}\ee^{-A^2/4}>0$.
Passing to the limit in \eqref{eq:product-comparison} and
\eqref{eq:scalar-profile-lower} yields
\begin{equation}\label{eq:entire-pulse-lower}
 V(s,k,w)\geq h_\eta
       \Phi_{\varepsilon_{\rm p}}(k-c_*s-\delta+c_*T)
 \qquad(s\in\R,\ (k,w)\in\Z^d).
\end{equation}
Set $R_0:=\max\{4,2(c_*T+1)+1\}$.
If $r\geq R_0$ and $R:=r-\delta+c_*T$, then
$R\geq r/2\geq1$ and $R\geq(1+r)/3$.
The inequality $P_\mu(R)\geq R/2$ in
Lemma~\ref{lem:delayed-pulse} and \eqref{eq:entire-pulse-lower} give
\[
 V(s,k,w)\geq\frac{h_\eta\varepsilon_{\rm p}}{2}
             R\ee^{-\lambda_*R}
 \geq b_0(1+r)\ee^{-\lambda_*r},\qquad
 b_0:=\frac{\varepsilon_{\rm p}h_{\min,A}}{6}
                          \ee^{-\lambda_*c_*T}>0.
\]
The last inequality uses $-\delta+c_*T\leq c_*T$.
The constants $b_0,B_0,R_0$ do not depend on the subsequential
parameters $\delta,\eta$ or on the evaluation point $(s,k,w)$.
Although convergence is only local in $w$, the bound therefore holds
at every site with these same constants, as required by
\eqref{eq:entire-critical-tail}.

For the rear state, set
$\eta_{\rm seed}:=\min\{1/2,h_{\min,A}\Phi_{\varepsilon_{\rm p}}(1)\}>0$.
For each integer $k$, let $s_k:=(k-\delta+c_*T-1)/c_*$.
Equation~\eqref{eq:entire-pulse-lower} gives
$V(s_k,k,w)\geq\eta_{\rm seed}$ for every $w$.
For $\theta\in(0,1)$ put
$S_\theta^{\rm rear}:=S(\eta_{\rm seed},\theta,d,f)$, with $S$ from
Lemma~\ref{lem:amplification}. The function
$\widehat V_k(\tau,x):=V(s_k+\tau,x)$, $\tau\geq0$, solves
\eqref{eq:main}, takes values in $[0,1]$, and has
$\widehat V_k(0,k,w)\geq\eta_{\rm seed}$.
Applying that lemma with initial time zero and site $(k,w)$ yields
$V(s,k,w)\geq\theta$ for $s\geq s_k+S_\theta^{\rm rear}$, or equivalently
\[
 k-c_*s\leq\delta-c_*T+1-c_*S_\theta^{\rm rear}.
\]
For a rear-state error $\rho\in(0,1)$, choose
$R_\rho:=c_*(T+S_{1-\rho}^{\rm rear})$.
Since $0\leq\delta\leq1$, the condition $k-c_*s\leq-R_\rho$
implies the displayed inequality with $\theta=1-\rho$. Therefore
\[
 \inf_{\substack{s\in\R,(k,w)\in\Z^d\\k-c_*s\leq-R_\rho}}
 V(s,k,w)\geq1-\rho.
\]
For every $\rho>0$ this also holds with $R\geq R_\rho$, proving
\eqref{eq:entire-rear-state}.

To check $V>0$ at $(s,k,w)$, choose $h>0$ with
$k-c_*(s-h)\geq R_0$. The lower tail gives $V(s-h,k,w)>0$.
Since $\partial_sV\geq-2dV$ at each site,
$V(s,k,w)\geq\ee^{-2dh}V(s-h,k,w)>0$.
If $V=1$ at one point, Lemma~\ref{lem:entire-positive-comparison}
applied to the ordered entire solutions $1,V$ would give $V\equiv1$,
contradicting the upper tail. Thus $0<V<1$, and
Lemma~\ref{lem:planar-lattice-rigidity} gives
$V(s,k,w)=\phi_*(k-c_*s+\xi)$.

The constants $b_0,B_0,R_0$ and $R_{1/4}=c_*(T+S_{3/4}^{\rm rear})$
are independent of the chosen subsequential limit.
Choose a negative integer $K_-\leq-R_{1/4}$, and an integer
$K_+\geq R_0$ with $B_0(1+K_+)\ee^{-\lambda_*K_+}\leq1/4$.
Then $V(0,K_-,0)\geq3/4$ and $V(0,K_+,0)\leq1/4$.
Since $\phi_*$ is strictly decreasing, these inequalities give
\[
 \phi_*^{-1}(1/4)-K_+\leq\xi
 \leq\phi_*^{-1}(3/4)-K_-.
\]
Both endpoints depend only on $d,f,u_0,A$.
\end{proof}

\begin{proof}[Proof of Theorem~\ref{thm:axial-profile}]
Fix $B\geq0$. We first bound $u$ near both ends of the half-tube
$j\geq0$, $|y|\leq B$. Apply
\eqref{eq:sharp-directional-envelope} with $e=e_1$. Then
$p_e=\lambda_*e_1$ and $z_e(t,(j,y))=\lambda_*(j-m_d(t))$.
Since $(1+z)\ee^{-z}$ is nonincreasing for
$z\geq0$, for every $R>0$ we obtain
$\sup_{j\geq m_d(t)+R,\ |y|\leq B}u(t,j,y)\leq
C_{\rm up}(1+\lambda_*R)\ee^{-\lambda_*R}$ for every $t\geq4$.
This bound tends to zero as $R\to\infty$. For the left tail, apply
Theorem~\ref{thm:nonlinear-lower} to
$u^y(t,j,w):=u(t,j,w+y)$ for each $y\in\Z^{d-1}$ with $|y|\leq B$.
Substitution in \eqref{eq:main} verifies its equation, and its initial
datum satisfies \eqref{eq:u0}. Denote the constants in
\eqref{eq:axial-inner-lower} for $u^y$ by $C_{\theta,y}^-$ and
$T_{\theta,y}^{\rm low}$. The maxima
$C_{\theta,B}^{\rm rear}:=\max_{|y|\leq B}C_{\theta,y}^-$ and
$T_{\theta,B}^{\rm rear}:=\max_{|y|\leq B}T_{\theta,y}^{\rm low}$
are finite. Consequently, $u(t,j,y)\geq\theta$ whenever
$t\geq T_{\theta,B}^{\rm rear}$,
$0\leq j\leq m_d(t)-C_{\theta,B}^{\rm rear}$ and $|y|\leq B$.
Letting $\theta\uparrow1$ yields
\begin{equation}\label{eq:fixed-tube-rear}
 \lim_{R\to\infty}\liminf_{t\to\infty}
 \inf_{\substack{0\leq j\leq m_d(t)-R\\|y|\leq B}}u(t,j,y)=1.
\end{equation}

Put $\ell(t)=\lfloor m_d(t)\rfloor$ and
$\delta(t)=m_d(t)-\ell(t)$. Lemma~\ref{lem:axial-entire-limits} with $A=0$ gives a constant
$C_{\rm ph}$ such that every limit of
\eqref{eq:axial-translate-sequence} with $y_n=0$ has the form
$\phi_*(k-c_*s+\xi)$ with $|\xi|\leq C_{\rm ph}$.  Since $\phi_*$ is
strictly decreasing and takes values in $(0,1)$ at finite arguments,
\[
 \eta_{\rm ph}:=\frac12\min\{\phi_*(C_{\rm ph}),
                         1-\phi_*(-C_{\rm ph})\}>0.
\]
If the bound
$\eta_{\rm ph}\leq u(t,\ell(t),0)\leq1-\eta_{\rm ph}$ failed along a sequence
$t_n\to\infty$, Lemma~\ref{lem:axial-entire-limits} would give a further
subsequence converging at the origin to $\phi_*(\xi)$ with
$|\xi|\leq C_{\rm ph}$, contradicting the definition of $\eta_{\rm ph}$.
Therefore there exists $t_0\geq2$ such that
\[
 \eta_{\rm ph}\leq u(t,\ell(t),0)\leq1-\eta_{\rm ph}
 \qquad(t\geq t_0).
\]
Define, for $t\geq t_0$,
\begin{equation}\label{eq:canonical-bounded-phase}
 \zeta(t):=\phi_*^{-1}(u(t,\ell(t),0))+\delta(t),
\end{equation}
and set $\zeta=0$ for $2\leq t<t_0$. Define
$C_\zeta:=1+\max\{|\phi_*^{-1}(\eta_{\rm ph})|,
|\phi_*^{-1}(1-\eta_{\rm ph})|\}$. Then $|\zeta(t)|\leq C_\zeta$ for
$t\geq2$, and neither $C_\zeta$ nor $\zeta$ depends on $B$.

Let $t_n\to\infty$ be arbitrary. Extract a subsequence as in
Lemma~\ref{lem:axial-entire-limits} with $y_n=0$ and
$\delta(t_n)\to\delta$. Its limit is
$\phi_*(k-c_*s+\xi)$, and \eqref{eq:canonical-bounded-phase} gives
$\zeta(t_n)\to\xi+\delta$. At $s=0$, the translated solutions converge on every fixed set of $k$
and $|y|\leq B$ to $\phi_*(k+\xi)$. Since
$\zeta(t_n)\to\xi+\delta$, one also has
\[
 \ell(t_n)+k-m_d(t_n)+\zeta(t_n)
 =k-\delta(t_n)+\zeta(t_n)\to k+\xi,
\]
and continuity of $\phi_*$ shows that the profiles in
\eqref{eq:axial-profile-convergence} converge to $\phi_*(k+\xi)$ as well.
Fix $0<\varepsilon<1$. Choose $R\geq C_{1-\varepsilon/4,B}^{\rm rear}$
so large that
$C_{\rm up}(1+\lambda_*R)\ee^{-\lambda_*R}<\varepsilon/4$,
$\phi_*(R-C_\zeta)<\varepsilon/4$, and
$1-\phi_*(-R+C_\zeta)<\varepsilon/4$.
For $t\geq\max\{4,T_{1-\varepsilon/4,B}^{\rm rear}\}$, the upper
envelope controls $u$ on $j\geq m_d(t)+R$, and the translated lower
bounds give $1-u(t,j,y)\leq\varepsilon/4$ on
$0\leq j\leq m_d(t)-R$, $|y|\leq B$.
Since $|\zeta(t)|\leq C_\zeta$, the error between $u$ and the profile
is at most $\varepsilon/2$ in both regions.
For the remaining indices, $|j-m_d(t)|\leq R$ implies
$|j-\ell(t)|\leq R+1$. Thus all recentered indices lie in the fixed
finite set
\[
 F_{R,B}:=\{(k,y)\in\Z\times\Z^{d-1}:|k|\leq R+1,\ |y|\leq B\}.
\]
Along the chosen subsequence, the solution and profile converge to
$\phi_*(k+\xi)$ uniformly on $F_{R,B}$, so their difference is less
than $\varepsilon$ there for all large $n$.
This proves convergence of the supremum on the whole half-tube along
the chosen subsequence. If \eqref{eq:axial-profile-convergence} failed,
there would be a sequence $t_n\to\infty$ on which that supremum was
at least some $\varepsilon_{\rm tube}>0$.
Applying the extraction and tail estimates to that sequence would
produce a further subsequence on which the supremum tends to zero,
a contradiction. Hence \eqref{eq:axial-profile-convergence} holds
for the full family as $t\to\infty$.

Fix $S>0$. To prove \eqref{eq:phase-slow-variation}, suppose that there
are $\varepsilon_{\rm ph}>0$, $t_n\to\infty$ and $|s_n|\leq S$
such that $|\zeta(t_n+s_n)-\zeta(t_n)|\geq\varepsilon_{\rm ph}$ for every
$n$. Extract subsequences with $s_n\to s$, $\delta(t_n)\to\delta$,
$\zeta(t_n)\to z$, and $\zeta(t_n+s_n)\to z'$, so
$|z-z'|\geq\varepsilon_{\rm ph}$. Extract a further subsequence in
Lemma~\ref{lem:axial-entire-limits} centered
at $(t_n,\ell(t_n),0)$; write its phase as $\xi$.  By
\eqref{eq:canonical-bounded-phase},
$z=\lim\zeta(t_n)=\xi+\delta$.  Since $s_n\to s$, local $C^1$
convergence of the translated solutions gives
\[
 u(t_n+s_n,\ell(t_n),0)\to
 \phi_*(-c_*s+\xi)=\phi_*(-c_*s-\delta+z).
\]
For large $n$, $t_n+s_n\geq2$ and $\ell(t_n)\geq0$.
Thus \eqref{eq:axial-profile-convergence}, with $B=0$ at time
$t_n+s_n$, applies at the fixed integer site $\ell(t_n)$ and gives
\[
 u(t_n+s_n,\ell(t_n),0)
 -\phi_*\bigl(\ell(t_n)-m_d(t_n+s_n)+\zeta(t_n+s_n)\bigr)
 \to0.
\]
The exact identity
\[
 \ell(t_n)-m_d(t_n+s_n)
 =-\delta(t_n)-c_*s_n+
       \frac{\kappa_d}{\lambda_*}\log(1+s_n/t_n)
\]
and the bound $|\log(1+s_n/t_n)|\leq2S/t_n$ for $t_n\geq2S$
show that the profile argument tends to $-c_*s-\delta+z'$.  Thus $u(t_n+s_n,\ell(t_n),0)$ also converges to
$\phi_*(-c_*s-\delta+z')$. Strict monotonicity of $\phi_*$ gives $z=z'$,
a contradiction. Both limits evaluate $u$ at the integer site
$\ell(t_n)$ and require no derivative of $\ell$. They therefore remain
valid when $t_n$ or $t_n+s_n$ is a jump time of the floor function.

Finally, fix $S,L>0$, and let $t_n\to\infty$ and
$j_n\in\Z_{\geq0}$ satisfy $u(t_n,j_n,0)\to\theta\in(0,1)$. Choose
$0<\theta_-<\theta<\theta_+<1$. For large $n$,
$\theta_-<u(t_n,j_n,0)<\theta_+$.
Equation~\eqref{eq:axial-inner-lower} at level $\theta_+$ and
\eqref{eq:remaining-sharp-upper} at level $\theta_-$ give
$m_d(t_n)-C_{\theta_+}^-<j_n<m_d(t_n)+L_{\theta_-}$.
In particular,
$|j_n-m_d(t_n)|\leq\max\{C_{\theta_+}^{-},L_{\theta_-}\}$. Thus
$j_n-\ell(t_n)$ is a bounded sequence of integers. From any
subsequence extract one on which this integer is constant, and apply
Lemma~\ref{lem:axial-entire-limits}.  If
$q=j_n-\ell(t_n)$ is the constant integer on this subsequence and the
limit centered at $\ell(t_n)$ is $\phi_*(k-c_*s+\xi)$, then the limit
centered at $j_n$ is
$\phi_*(k-c_*s+q+\xi)$.  Writing $\gamma=q+\xi$, this is
$\phi_*(k-c_*s+\gamma)$. Its value at $(0,0,0)$ is
$\theta$, so $\gamma=\phi_*^{-1}(\theta)$. Thus every subsequence has a further subsequence converging uniformly on
$|s|\leq S$, $|k|+|y|\leq L$ to
$\phi_*(k-c_*s+\phi_*^{-1}(\theta))$.  If
\eqref{eq:level-profile-convergence} failed, there would be an
$\varepsilon_0>0$ and a subsequence for which the supremum in
\eqref{eq:level-profile-convergence} is
at least $\varepsilon_0$; the further convergent subsequence would make
that supremum tend to zero, a contradiction.  Hence
\eqref{eq:level-profile-convergence} holds for the full sequence.
For a signed permutation matrix $R$, apply
\eqref{eq:axial-profile-convergence}--\eqref{eq:level-profile-convergence}
to $u_R(t,x)=u(t,Rx)$, which solves \eqref{eq:main} with initial datum
$u_0\circ R$ satisfying \eqref{eq:u0}. Since
$u_R(t,je_1)=u(t,jRe_1)$, these three convergence statements hold
along $Re_1$.
\end{proof}

\begin{corollary}\label{cor:diffusive-level-profile}
Assume \eqref{eq:kpp}--\eqref{eq:u0}. Let $t_n\to\infty$ and
$(j_n,y_n)\in\Z^d$ satisfy
$|j_n-m_d(t_n)|\leq C_{\rm loc}$, $|y_n|\leq A\sqrt{t_n}$, and
$u(t_n,j_n,y_n)\to\theta\in(0,1)$, for fixed $A,C_{\rm loc}<\infty$.
Then, for every $S,L>0$,
\begin{equation}\label{eq:diffusive-profile-convergence}
 \sup_{\substack{|s|\leq S\\k\in\Z,\ w\in\Z^{d-1},\ |k|+|w|\leq L}}
 \left|u(t_n+s,j_n+k,y_n+w)
       -\phi_*(k-c_*s+\phi_*^{-1}(\theta))\right|\to0.
\end{equation}
\end{corollary}

\begin{proof}
Put $q_n=j_n-\lfloor m_d(t_n)\rfloor$. Since
$|q_n|\leq C_{\rm loc}+1$, this sequence takes only finitely many
integer values. From any subsequence, extract a further subsequence on which
$q_n=q$ is constant. Lemma~\ref{lem:axial-entire-limits}, applied with the
given transverse centers, then gives a limit
$\phi_*(k+q-c_*s+\xi)$ for the solution recentered at $(j_n,y_n)$.
Its value at $(s,k,w)=(0,0,0)$ is $\theta$, so
$q+\xi=\phi_*^{-1}(\theta)$. Thus every subsequence has a further
subsequence converging locally to
$\phi_*(k-c_*s+\phi_*^{-1}(\theta))$.
If the supremum in \eqref{eq:diffusive-profile-convergence} failed
to tend to zero, it would remain at least $\varepsilon_0>0$ along
a subsequence. The further locally convergent subsequence would
contradict this bound, proving \eqref{eq:diffusive-profile-convergence}.
\end{proof}

Uniform convergence on a fixed half-tube does not mean uniform convergence
to one planar wave over all $y\in\Z^{d-1}$. For fixed $t\geq2$ and $j\in\Z$, the
upper inequality in
\eqref{eq:linear-comparison} gives
$u(t,j,y)\leq \ee^{at}\sum_z\pheat_d(t,(j,y)-z)u_0(z)$; since $u_0$ has
finite support and $\pheat_d(t,\cdot)\in\ell^1(\Z^d)$ by
\eqref{eq:heat-semigroup}, each summand tends to zero as $|y|\to\infty$.
Hence $u(t,j,y)\to0$ as $|y|\to\infty$. Since the wave profile is independent of $y$,
\[
 \sup_{y\in\Z^{d-1}}
 \left|u(t,j,y)-\phi_*(j-m_d(t)+\zeta(t))\right|
 \geq\phi_*(j-m_d(t)+\zeta(t)).
\]
If $|j-m_d(t)|\leq1$, boundedness of $\zeta$ makes the right-hand side
uniformly positive. Thus this supremum cannot tend to zero.
Corollary~\ref{cor:diffusive-level-profile} gives local wave
convergence at the centers $(j_n,y_n)$ with $|y_n|\leq A\sqrt{t_n}$.

\section{Weighted estimates and other directions}\label{sec:remaining}

The exponent $1/2$ in \eqref{eq:critical-mass-main} is sharp on
the signed coordinate axes. Corollary~\ref{cor:two-sided-critical-mass}
derives the matching lower bound from \eqref{eq:product-comparison},
\eqref{eq:autonomous-leading-lower}, and the unit mass of the
transverse heat kernel in \eqref{eq:heat-semigroup}.

\begin{corollary}\label{cor:two-sided-critical-mass}
Under \eqref{eq:kpp}--\eqref{eq:u0}, there are $c_{\rm mass}>0$ and
$T_{\rm mass}\geq1$, depending only on $d,f,u_0$, such that, with
$C_{\rm mass}$ from \eqref{eq:critical-mass-main},
\begin{equation}\label{eq:two-sided-critical-mass}
    c_{\rm mass}t^{-1/2}\leq
       \sum_{(j,y)\in\Z^d}\ee^{\lambda_*(j-c_*t)}u(t,j,y)
           \leq C_{\rm mass}t^{-1/2}\qquad(t\geq T_{\rm mass}).
\end{equation}
For every $i\in\{1,\ldots,d\}$ and $\sigma\in\{-1,1\}$ there
are $c_{i,\sigma},C_{i,\sigma},T_{i,\sigma}>0$ such that
\[
 c_{i,\sigma}t^{-1/2}\leq
 \sum_{x\in\Z^d}\ee^{\lambda_*(\sigma x_i-c_*t)}u(t,x)
 \leq C_{i,\sigma}t^{-1/2}
 \qquad(t\geq T_{i,\sigma}).
\]
\end{corollary}

\begin{proof}
For $e=e_1$, Lemma~\ref{lem:critical-pair} gives
$p_e=\lambda_*e_1$ and $v_e=c_*e_1$. Thus
\eqref{eq:critical-mass-main} gives the upper bound, since
$(1+t)^{-1/2}\leq t^{-1/2}$ for $t>0$.

For the lower bound, choose $T_0,j_0,y_0,\eta_q$ as in
Lemma~\ref{lem:product-subsolution}, set
$q_0=\eta_q\1_{\{j_0\}}$, and put $\beta=(d-1)/2$.
Let $q$ solve \eqref{eq:time-dependent-logistic} with $T=T_0$ and
$q(0)=q_0$, and let $Q$ be the solution in
Lemma~\ref{lem:autonomous-leading-lower} with $K=C_f$ and $Q(0)=q_0$.
The initial datum is nonzero, finitely supported, and bounded by $a/C_f$.
Equation~\eqref{eq:rhoQ-comparison}, with these parameters, gives
$q(s,j)\geq(T_0/(s+T_0))^\beta Q(s,j)$. Substitution in
\eqref{eq:product-comparison} and the definition of $H_{T_0}$ gives
\[
 u(T_0+s,j,y)\geq\beta_\perp T_0^\beta Q(s,j)
       \pheat_{d-1}(s+T_0,y-y_0)
 \qquad(s\geq0).
\]
All summands below are nonnegative, and the left-hand sum is finite by
\eqref{eq:critical-mass-main}. Sum first in $y$ and use
$\sum_y\pheat_{d-1}(s+T_0,y-y_0)=1$ from
\eqref{eq:heat-semigroup}. This yields
\begin{equation}\label{eq:mass-product-sum}
 \sum_{(j,y)\in\Z^d}\ee^{\lambda_*(j-c_*(T_0+s))}u(T_0+s,j,y)
 \geq\beta_\perp T_0^\beta\ee^{-\lambda_*c_*T_0}
       \sum_{j\in\Z}\ee^{\lambda_*(j-c_*s)}Q(s,j).
\end{equation}

Write $c_Q,L_Q,T_Q$ for the constants in
Lemma~\ref{lem:autonomous-leading-lower} for this $Q$.
Set $S_{\rm mass}:=\max\{T_Q,16L_Q^2,64\}$.
For $s\geq S_{\rm mass}$, put
$J_s=[c_*s+\sqrt s/4,c_*s+\sqrt s/2]\cap\Z$.
Then $L_Q\leq j-c_*s\leq\sqrt s$ for every $j\in J_s$.
An interval of length $\sqrt s/4$ contains at least
$\sqrt s/4-1$ integers, so
$\#J_s\geq\sqrt s/4-1\geq\sqrt s/8$.
Equation~\eqref{eq:autonomous-leading-lower} gives
\[
 \sum_{j\in\Z}\ee^{\lambda_*(j-c_*s)}Q(s,j)
 \geq c_Qs^{-3/2}\sum_{j\in J_s}(j-c_*s)
 \geq\frac{c_Q}{32}s^{-1/2}.
\]
Choose
\[
 c_{\rm mass}:=\frac{\beta_\perp T_0^\beta c_Q}{32}
                  \ee^{-\lambda_*c_*T_0}>0,
 \qquad T_{\rm mass}:=T_0+S_{\rm mass}.
\]
For $t\geq T_{\rm mass}$, put $s=t-T_0$. Then
$s\geq S_{\rm mass}$ and $s^{-1/2}\geq t^{-1/2}$.
Equation~\eqref{eq:mass-product-sum} proves the lower inequality in
\eqref{eq:two-sided-critical-mass} with these fixed constants.

For a signed coordinate axis $\sigma e_i$, choose a signed permutation
matrix $R$ with $Re_1=\sigma e_i$ and set $u_R(t,x)=u(t,Rx)$.
Since $R$ permutes the nearest-neighbor steps,
$\Delta_{\Z^d}u_R(t,x)=(\Delta_{\Z^d}u)(t,Rx)$.
Thus $u_R$ solves \eqref{eq:main}, and its initial datum
$u_0(Rx)$ satisfies \eqref{eq:u0}. Apply
\eqref{eq:two-sided-critical-mass} to $u_R$, and denote its lower
constant and time threshold by $c_{i,\sigma}$ and $T_{i,\sigma}$.
With $z=Rx$ and $R^{-1}=R^{\mathsf T}$,
$(R^{-1}z)_1=(Re_1)\cdot z=\sigma z_i$, so
\[
 \sum_{x\in\Z^d}\ee^{\lambda_*(x_1-c_*t)}u_R(t,x)
 =\sum_{z\in\Z^d}\ee^{\lambda_*(\sigma z_i-c_*t)}u(t,z).
\]
For $e=\sigma e_i$, the formulas in \eqref{eq:I} and
\eqref{eq:pstar} give $v_e=c_*\sigma e_i$ and
$p_e=\lambda_*\sigma e_i$. Applying \eqref{eq:critical-mass-main}
to the original solution in this direction gives the upper bound with
$C_{i,\sigma}=C_{\rm mass}$.
\end{proof}

For $e\in\mathbb S^{d-1}$, the vector $n_e=p_e/|p_e|$ is the
outward unit normal to $\partial\mathcal W_a$ at $v_e=w_*(e)e$.
Lemma~\ref{lem:critical-pair} gives $\lambda(n_e)=|p_e|$ and
$\alpha_e=\lambda(n_e)n_e\cdot e>0$. Hence the logarithmic coefficient
in the upper location bound \eqref{eq:sharp-directional-location} is
\begin{equation}\label{eq:general-shift}
    \frac{\kappa_d}{\alpha_e}
       =\frac{d/2+1}{\lambda(n_e)\,n_e\cdot e}.
\end{equation}
On a signed coordinate axis, $n_e=e$ and $\lambda(n_e)=\lambda_*$,
so this coefficient is $\kappa_d/\lambda_*$.

An observation along a ray also requires a choice of lattice sites.
For example, for $r\geq0$ define
$x_e(r):=(\lfloor r(e)_k+1/2\rfloor)_{k=1}^d$, where $(e)_k$ is
the $k$th coordinate of $e$. Then $|x_e(r)-re|\leq\sqrt d/2$, and
\eqref{eq:sharp-directional-location} applies with $\rho=\sqrt d/2$.
Exact intersections with the ray need not exist: for
$e=(1,\sqrt2,0,\ldots,0)/\sqrt3$, no $r>0$ has $re\in\Z^d$,
because its first two nonzero coordinates have irrational ratio.

Under the hypotheses of \cite[Corollary~4.1]{Gartner}, the transition
region in $\mathbb{R}^d$ lies in an annulus of bounded width.
For \eqref{eq:main}, Theorem~\ref{thm:sharp-directional-upper}
gives the upper bound \eqref{eq:sharp-directional-location} uniformly
in $e$. The matching lower bound in
Theorem~\ref{thm:nonlinear-lower} uses the coordinate decomposition of
$\Delta_{\Z^d}$ and is proved only on the signed coordinate axes.
We do not obtain a bounded-width localization of the full transition set
$\{x\in\Z^d:\varepsilon<u(t,x)<1-\varepsilon\}$ that is uniform
over all directions, for $0<\varepsilon<1/2$.
Such a conclusion requires lower bounds uniform in $e$, with a
specified lattice approximation to each ray.
The uncorrected boundary $\partial(t\mathcal W_a)$ meets the positive
coordinate axis at $c_*te_1$, whereas \eqref{eq:bramson-target} gives
$c_*t-R_\theta(t)=(\kappa_d/\lambda_*)\log t+O(1)$ for each fixed
$\theta\in(0,1)$.

For initial data $u_0(j,y)=v_0(j)$ with $v_0:\Z\to[0,1]$, let $v$
solve $v_t=\Delta_{\Z}v+f(v)$ with $v(0)=v_0$.
The function $\overline u(t,j,y)=v(t,j)$ satisfies
$\Delta_{\Z^d}\overline u(t,j,y)=\Delta_{\Z}v(t,j)$, since every
transverse difference vanishes. It therefore solves \eqref{eq:main}
with initial datum $u_0$, and uniqueness in
Lemma~\ref{lem:comparison} gives $u=\overline u$.
If $v_0(j_0)>0$, the support of $u_0$ contains
$\{j_0\}\times\Z^{d-1}$ and is infinite because $d\geq2$.
Thus nonzero data of this form do not satisfy \eqref{eq:u0}.

For the finitely supported datum in \eqref{eq:u0}, fix
$\theta\in(0,1)$ and define the maximal projection by
\[
 \widehat R_\theta(t):=\sup\{j\in\Z:\ u(t,j,y)\geq\theta
     \text{ for some }y\in\Z^{d-1}\},\qquad
 \sup\varnothing:=-\infty.
\]
Let $T_\theta^{\rm low}$ denote the time threshold in
\eqref{eq:sharp-axial-lower}. For
$t\geq\max\{4,T_\theta^{\rm low}\}$, that lower bound makes the
axial level set nonempty. Equation~\eqref{eq:remaining-sharp-upper}
excludes every $j\geq m_d(t)+L_\theta$ from the projected level set,
uniformly in $y$. Both sets are therefore nonempty and bounded above
in $\Z$, so their suprema are maxima. In particular,
\[
 m_d(t)-C_\theta^-\leq R_\theta(t)\leq\widehat R_\theta(t)
       <m_d(t)+L_\theta
 \qquad\bigl(t\geq\max\{4,T_\theta^{\rm low}\}\bigr).
\]
This proves $\widehat R_\theta(t)=m_d(t)+O(1)$, with constants
depending only on $d,f,u_0,\theta$.

\begin{remark}\label{rem:integer-phase}
Fix $\theta\in(0,1)$. Corollary~\ref{cor:nonempty} implies that
$R_\theta(t)\in\Z$ for all sufficiently large $t$.
A finite limit of $R_\theta(t)-m_d(t)$ through all real times is
impossible. Indeed, $m_d'(t)=c_*-\kappa_d/(\lambda_*t)>0$ for
$t>\max\{2,\kappa_d/(\lambda_*c_*)\}$, and $m_d(t)\to\infty$.
If the displacement converged to $\sigma\in\R$, continuity and
eventual strict increase of $m_d$ would give $t_n\to\infty$ with
$m_d(t_n)+\sigma=n+1/2$ for every sufficiently large integer $n$.
Since $R_\theta(t_n)\in\Z$, one would have
$|R_\theta(t_n)-m_d(t_n)-\sigma|\geq1/2$, a contradiction.
Convergence through all real times can instead be formulated for an
interpolated level position or for the phase in a wave approximation.
Theorem~\ref{thm:axial-profile} provides a bounded phase $\zeta(t)$
satisfying \eqref{eq:phase-slow-variation}; it does not assert that
$\zeta(t)$ has a limit.
\end{remark}

An axial restriction does not satisfy a closed one-dimensional equation.
Let $\widetilde e_1,\ldots,\widetilde e_{d-1}$ be the coordinate vectors
of $\Z^{d-1}$ and set $r_j(t):=u(t,j,0)$. For $t>0$ and $j\in\Z$,
substitution in \eqref{eq:main} gives
\[
    \dot r_j=r_{j-1}-2r_j+r_{j+1}+f(r_j)
       +\sum_{k=1}^{d-1}
          \bigl(u(t,j,\widetilde e_k)+u(t,j,-\widetilde e_k)-2r_j\bigr).
\]
The transverse sum depends on values away from the axis and is not
determined by $(r_j(t))_{j\in\Z}$. Section~\ref{sec:profiles}
therefore takes limits of translates on $\Z^d$, as specified in
Lemma~\ref{lem:axial-entire-limits}. In that lemma, the lower inequality
in \eqref{eq:entire-critical-tail} comes from
\eqref{eq:product-comparison} and \eqref{eq:scalar-profile-lower};
the upper inequality comes from \eqref{eq:sharp-directional-envelope}.
Both bounds are uniform in the transverse variable.
Lemma~\ref{lem:planar-lattice-rigidity} then gives
\eqref{eq:lattice-space-time-invariance}. Taking $q=0$ in that
identity yields $V(t,j,y+\eta)=V(t,j,y)$ for every
$\eta\in\Z^{d-1}$, which proves that each such entire limit is
independent of $y$.

\end{document}